\documentclass{elsarticle}
\usepackage[margin=1in]{geometry}

\usepackage{bbm}
\usepackage{graphicx}
\usepackage{pstool}
\usepackage{wrapfig}
\usepackage{caption}
\usepackage{subcaption}
\usepackage[section]{placeins} 

\usepackage{fancyhdr} 
\usepackage{lastpage} 
\usepackage{extramarks} 

\usepackage{xcolor} 
\usepackage{enumerate} 
\usepackage{paralist} 
\usepackage{amsmath, amsthm, amssymb, mathtools}
\usepackage{mathabx, pifont, stmaryrd} 
\usepackage[explicit]{titlesec} 
\usepackage{etoolbox} 
\usepackage{bibentry} 
\makeatletter\let\saved@bibitem\@bibitem\makeatother 
\usepackage[colorlinks, bookmarksopen, bookmarksnumbered,
            citecolor=red,urlcolor=red]{hyperref} 
\makeatletter\let\@bibitem\saved@bibitem\makeatother 

\usepackage{algorithm}
\usepackage{algorithmic}

\newtheorem{theorem}{Theorem}
\newtheorem{lemma}{Lemma}

\newtheorem{remark}{Remark}

\theoremstyle{definition}
\newtheorem{assume}{Assumption}

\usepackage{bm} 
\usepackage{amsfonts} 

\newcommand{\func}[3]{\ensuremath{#1 : #2 \rightarrow #3}}

\newcommand{\norm}[1]{\ensuremath{\left\| #1 \right\|}}

\newcommand{\suchthat}{\mathrel{}\middle|\mathrel{}}

\newcommand{\optunc}[2]{\underset{#1}{\text{minimize}} ~~ #2}
\newcommand{\argoptunc}[2]{\underset{#1}{\arg\min} ~~ #2}

\newcommand{\optconTwo}[4]{
\begin{aligned}
& \underset{#1}{\text{minimize}}
& & #2 \\
& \text{subject to} & & #3 \\
& & & #4
\end{aligned}
}

\newcommand{\optconThree}[5]{
\begin{aligned}
& \underset{#1}{\text{minimize}}
& & #2 \\
& \text{subject to} & & #3 \\
& & & #4 \\
& & & #5
\end{aligned}
}

\newcommand{\pder}[2]{\ensuremath{\frac{\partial #1}{\partial #2}}}

\newcommand{\Ccal}{\ensuremath{\mathcal{C}}}
\newcommand{\Dcal}{\ensuremath{\mathcal{D}}}
\newcommand{\Ecal}{\ensuremath{\mathcal{E}}}

\newcommand{\Lcal}{\ensuremath{\mathcal{L}}}

\newcommand{\Pcal}{\ensuremath{\mathcal{P}}}

\newcommand{\Rcal}{\ensuremath{\mathcal{R}}}
\newcommand{\Scal}{\ensuremath{\mathcal{S}}}

\newcommand{\Ucal}{\ensuremath{\mathcal{U}}}

\newcommand{\Ycal}{\ensuremath{\mathcal{Y}}}

\newcommand{\Vboldcal}{\ensuremath{\boldsymbol{\mathcal{V}}}}

\newcommand{\Nbb}{\ensuremath{\mathbb{N} }}

\newcommand{\Rbb}{\ensuremath{\mathbb{R} }}

\newcommand\Abm{{\ensuremath{\bm{A}}}}

\newcommand\Dbm{{\ensuremath{\bm{D}}}}

\newcommand\Hbm{{\ensuremath{\bm{H}}}}

\newcommand\Pbm{{\ensuremath{\bm{P}}}}

\newcommand\Ubm{{\ensuremath{\bm{U}}}}
\newcommand\Vbm{{\ensuremath{\bm{V}}}}

\newcommand\cbm{{\ensuremath{\bm{c}}}}
\newcommand\dbm{{\ensuremath{\bm{d}}}}

\newcommand\gbm{{\ensuremath{\bm{g}}}}

\newcommand\rbm{{\ensuremath{\bm{r}}}}
\newcommand\sbm{{\ensuremath{\bm{s}}}}

\newcommand\ubm{{\ensuremath{\bm{u}}}}
\newcommand\vbm{{\ensuremath{\bm{v}}}}
\newcommand\wbm{{\ensuremath{\bm{w}}}}
\newcommand\xbm{{\ensuremath{\bm{x}}}}
\newcommand\ybm{{\ensuremath{\bm{y}}}}
\newcommand\zbm{{\ensuremath{\bm{z}}}}

\newcommand\lambdabold{{\ensuremath{\boldsymbol{\lambda}}}}

\newcommand\deltabold{{\ensuremath{\boldsymbol{\delta}}}}
\newcommand\mubold{{\ensuremath{\boldsymbol{\mu}}}}

\newcommand\rhobold{{\ensuremath{\boldsymbol{\rho}}}}
\newcommand\taubold{{\ensuremath{\boldsymbol{\tau}}}}

\newcommand\nubold{{\ensuremath{\boldsymbol{\nu}}}}

\newcommand\Phibold{{\ensuremath{\boldsymbol{\Phi}}}}

\newcommand\Xibold{{\ensuremath{\boldsymbol{\Xi}}}}

\newcommand\zerobold{\ensuremath{\mathbf{0}}}
\newcommand\onebold{\ensuremath{\mathbf{1}}}

\usepackage{tikz}
\usepackage{pgfplots}
\usepackage{pgfplotstable, filecontents, booktabs}
\pgfplotsset{compat=1.9}

\usetikzlibrary{pgfplots.groupplots}
\usepgfplotslibrary{fillbetween}
\usetikzlibrary{calc,fit,matrix,arrows,automata,positioning,shapes}
\usetikzlibrary{arrows.meta}

\pgfplotsset{select coords between index/.style 2 args={
    x filter/.code={
        \ifnum\coordindex<#1\fi
        \ifnum\coordindex>#2\fi
    }
}}

\tikzset{
 invisible/.style={opacity=0},
 visible on/.style={alt={#1{}{invisible}}},
 alt/.code args={<#1>#2#3}{%
   \alt<#1>{\pgfkeysalso{#2}}{\pgfkeysalso{#3}}
 },
}

\newcommand{\colorbarMatlabParula}[5]{
\begin{tikzpicture}
\begin{axis}[
   hide axis, scale only axis,
   height=0pt, width=0pt,
   colormap={parula}{rgb255=(62,38,168) rgb255=(62,39,172) rgb255=(63,40,175) rgb255=(63,41,178) rgb255=(64,42,180) rgb255=(64,43,183) rgb255=(65,44,186) rgb255=(65,45,189) rgb255=(66,46,191) rgb255=(66,47,194) rgb255=(67,48,197) rgb255=(67,49,200) rgb255=(67,50,202) rgb255=(68,51,205) rgb255=(68,52,208) rgb255=(69,53,210) rgb255=(69,55,213) rgb255=(69,56,215) rgb255=(70,57,217) rgb255=(70,58,220) rgb255=(70,59,222) rgb255=(70,61,224) rgb255=(71,62,225) rgb255=(71,63,227) rgb255=(71,65,229) rgb255=(71,66,230) rgb255=(71,68,232) rgb255=(71,69,233) rgb255=(71,70,235) rgb255=(72,72,236) rgb255=(72,73,237) rgb255=(72,75,238) rgb255=(72,76,240) rgb255=(72,78,241) rgb255=(72,79,242) rgb255=(72,80,243) rgb255=(72,82,244) rgb255=(72,83,245) rgb255=(72,84,246) rgb255=(71,86,247) rgb255=(71,87,247) rgb255=(71,89,248) rgb255=(71,90,249) rgb255=(71,91,250) rgb255=(71,93,250) rgb255=(70,94,251) rgb255=(70,96,251) rgb255=(70,97,252) rgb255=(69,98,252) rgb255=(69,100,253) rgb255=(68,101,253) rgb255=(67,103,253) rgb255=(67,104,254) rgb255=(66,106,254) rgb255=(65,107,254) rgb255=(64,109,254) rgb255=(63,110,255) rgb255=(62,112,255) rgb255=(60,113,255) rgb255=(59,115,255) rgb255=(57,116,255) rgb255=(56,118,254) rgb255=(54,119,254) rgb255=(53,121,253) rgb255=(51,122,253) rgb255=(50,124,252) rgb255=(49,125,252) rgb255=(48,127,251) rgb255=(47,128,250) rgb255=(47,130,250) rgb255=(46,131,249) rgb255=(46,132,248) rgb255=(46,134,248) rgb255=(46,135,247) rgb255=(45,136,246) rgb255=(45,138,245) rgb255=(45,139,244) rgb255=(45,140,243) rgb255=(45,142,242) rgb255=(44,143,241) rgb255=(44,144,240) rgb255=(43,145,239) rgb255=(42,147,238) rgb255=(41,148,237) rgb255=(40,149,236) rgb255=(39,151,235) rgb255=(39,152,234) rgb255=(38,153,233) rgb255=(38,154,232) rgb255=(37,155,232) rgb255=(37,156,231) rgb255=(36,158,230) rgb255=(36,159,229) rgb255=(35,160,229) rgb255=(35,161,228) rgb255=(34,162,228) rgb255=(33,163,227) rgb255=(32,165,227) rgb255=(31,166,226) rgb255=(30,167,225) rgb255=(29,168,225) rgb255=(29,169,224) rgb255=(28,170,223) rgb255=(27,171,222) rgb255=(26,172,221) rgb255=(25,173,220) rgb255=(23,174,218) rgb255=(22,175,217) rgb255=(20,176,216) rgb255=(18,177,214) rgb255=(16,178,213) rgb255=(14,179,212) rgb255=(11,179,210) rgb255=(8,180,209) rgb255=(6,181,207) rgb255=(4,182,206) rgb255=(2,183,204) rgb255=(1,183,202) rgb255=(0,184,201) rgb255=(0,185,199) rgb255=(0,186,198) rgb255=(1,186,196) rgb255=(2,187,194) rgb255=(4,187,193) rgb255=(6,188,191) rgb255=(9,189,189) rgb255=(13,189,188) rgb255=(16,190,186) rgb255=(20,190,184) rgb255=(23,191,182) rgb255=(26,192,181) rgb255=(29,192,179) rgb255=(32,193,177) rgb255=(35,193,175) rgb255=(37,194,174) rgb255=(39,194,172) rgb255=(41,195,170) rgb255=(43,195,168) rgb255=(44,196,166) rgb255=(46,196,165) rgb255=(47,197,163) rgb255=(49,197,161) rgb255=(50,198,159) rgb255=(51,199,157) rgb255=(53,199,155) rgb255=(54,200,153) rgb255=(56,200,150) rgb255=(57,201,148) rgb255=(59,201,146) rgb255=(61,202,144) rgb255=(64,202,141) rgb255=(66,202,139) rgb255=(69,203,137) rgb255=(72,203,134) rgb255=(75,203,132) rgb255=(78,204,129) rgb255=(81,204,127) rgb255=(84,204,124) rgb255=(87,204,122) rgb255=(90,204,119) rgb255=(94,205,116) rgb255=(97,205,114) rgb255=(100,205,111) rgb255=(103,205,108) rgb255=(107,205,105) rgb255=(110,205,102) rgb255=(114,205,100) rgb255=(118,204,97) rgb255=(121,204,94) rgb255=(125,204,91) rgb255=(129,204,89) rgb255=(132,204,86) rgb255=(136,203,83) rgb255=(139,203,81) rgb255=(143,203,78) rgb255=(147,202,75) rgb255=(150,202,72) rgb255=(154,201,70) rgb255=(157,201,67) rgb255=(161,200,64) rgb255=(164,200,62) rgb255=(167,199,59) rgb255=(171,199,57) rgb255=(174,198,55) rgb255=(178,198,53) rgb255=(181,197,51) rgb255=(184,196,49) rgb255=(187,196,47) rgb255=(190,195,45) rgb255=(194,195,44) rgb255=(197,194,42) rgb255=(200,193,41) rgb255=(203,193,40) rgb255=(206,192,39) rgb255=(208,191,39) rgb255=(211,191,39) rgb255=(214,190,39) rgb255=(217,190,40) rgb255=(219,189,40) rgb255=(222,188,41) rgb255=(225,188,42) rgb255=(227,188,43) rgb255=(230,187,45) rgb255=(232,187,46) rgb255=(234,186,48) rgb255=(236,186,50) rgb255=(239,186,53) rgb255=(241,186,55) rgb255=(243,186,57) rgb255=(245,186,59) rgb255=(247,186,61) rgb255=(249,186,62) rgb255=(251,187,62) rgb255=(252,188,62) rgb255=(254,189,61) rgb255=(254,190,60) rgb255=(254,192,59) rgb255=(254,193,58) rgb255=(254,194,57) rgb255=(254,196,56) rgb255=(254,197,55) rgb255=(254,199,53) rgb255=(254,200,52) rgb255=(254,202,51) rgb255=(253,203,50) rgb255=(253,205,49) rgb255=(253,206,49) rgb255=(252,208,48) rgb255=(251,210,47) rgb255=(251,211,46) rgb255=(250,213,46) rgb255=(249,214,45) rgb255=(249,216,44) rgb255=(248,217,43) rgb255=(247,219,42) rgb255=(247,221,42) rgb255=(246,222,41) rgb255=(246,224,40) rgb255=(245,225,40) rgb255=(245,227,39) rgb255=(245,229,38) rgb255=(245,230,38) rgb255=(245,232,37) rgb255=(245,233,36) rgb255=(245,235,35) rgb255=(245,236,34) rgb255=(245,238,33) rgb255=(246,239,32) rgb255=(246,241,31) rgb255=(246,242,30) rgb255=(247,244,28) rgb255=(247,245,27) rgb255=(248,247,26) rgb255=(248,248,24) rgb255=(249,249,22) rgb255=(249,251,21) },
   colorbar horizontal,
   point meta min=#1, point meta max=#5,
   colorbar style={width=10cm, xtick={#1,#2,#3,#4,#5}}
]
\addplot [draw=none] coordinates {(0,0)};
\end{axis}
\end{tikzpicture}
}

\theoremstyle{definition}

\newtheorem{corollary}{Corollary}[theorem]

\newcommand{\paren}[1]{\ensuremath{\left( #1 \right)}}
\newcommand{\bracket}[1]{\ensuremath{\left[ #1 \right]}}
\newcommand{\curlyb}[1]{\ensuremath{\left\{ #1 \right\} }}
\newcommand{\sens}[1]{\ensuremath{\pder{#1}{\mubold}}}
\newcommand{\abs}[1]{\ensuremath{\left| #1 \right|}}

\newcommand{\trbm}{\Tilde{\rbm}}

\newcommand{\hrbm}{\hat{\rbm}}

\newcommand{\hlam}{\hat{\lambdabold}}

\newcommand{\rlam}{\rbm^\lambda}
\newcommand{\glam}{\gbm^\lambda}

\newcommand{\hrlam}{\hat{\rbm}^\lambda}

\newcommand{\trlam}{\tilde{\rbm}^\lambda}

\newcommand\thetabold{{\ensuremath{\boldsymbol{\theta}}}}

\usepackage{setspace}
\usepackage{enumitem}
\usepackage{lettrine}
\usepackage{accents}
\newbool{fastcompile}
\setbool{fastcompile}{false}

\begin{document}
\title{An augmented Lagrangian trust-region method with inexact gradient evaluations to accelerate constrained optimization problems using model hyperreduction}

\author[rvt1]{Tianshu Wen\fnref{fn1}}
\ead{twen2@nd.edu}

\author[rvt1]{Matthew J. Zahr\fnref{fn2}\corref{cor1}}
\ead{mzahr@nd.edu}

\address[rvt1]{Department of Aerospace and Mechanical Engineering, University
               of Notre Dame, Notre Dame, IN 46556, United States}
\cortext[cor1]{Corresponding author}

\fntext[fn1]{Graduate Student, Department of Aerospace and Mechanical
             Engineering, University of Notre Dame}
\fntext[fn2]{Assistant Professor, Department of Aerospace and Mechanical
             Engineering, University of Notre Dame}


\begin{abstract}
We present an augmented Lagrangian trust-region method to efficiently solve constrained optimization problems governed by large-scale nonlinear systems with application to partial differential equation-constrained optimization. At each major augmented Lagrangian iteration, the expensive optimization subproblem involving the full nonlinear system is replaced by an empirical quadrature-based hyperreduced model constructed on-the-fly. To ensure convergence of these inexact augmented Lagrangian subproblems, we develop a bound-constrained trust-region method that allows for inexact gradient evaluations, and specialize it to our specific setting that leverages hyperreduced models. This approach circumvents a traditional training phase because the models are built on-the-fly in accordance with the requirements of the trust-region convergence theory. Two numerical experiments (constrained aerodynamic shape design) demonstrate the convergence and efficiency of the proposed work. A speedup of 12.7x (for all computational costs, even costs traditionally considered ``offline'' such as snapshot collection and data compression) relative to a standard optimization approach that does not leverage model reduction is shown.
\end{abstract}

\maketitle
\section{Introduction}
\label{sec:intro}
Optimization problems involving partial differential equations (PDEs) arise in many fields for product design, risk control, cost management, etc. Often, the optimization problem may have additional side constraints (i.e., constraints other than the PDE constraint), which increases the complexity of the problem and the computational cost required to solve it. These problems can be prohibitively expensive for PDE discretizations with many degrees of freedom (DoFs) due to repeated queries to the expensive PDE solver.

To mitigate the large computational cost of such problems, surrogate-based approaches have been developed to reduce the computational cost of expensive PDE-constrained optimization problems. Surrogate models utilize a relatively cheap model to replace the original one throughout the optimization iterations to reduce the computational cost. There are many types of surrogate models, e.g.,  adaptive spatial discretization \cite{ziems_adaptive_2011}, partially converged solutions \cite{forrester_optimization_2006,zahr_adaptive_2016},
response surfaces \cite{forrester_recent_2009}, projection-based reduced-order models (ROM) \cite{arian_trust-region_2000, qian_certified_2017, zahr_progressive_2015,zahr_efficient_2019, yano_globally_2021, heinkenschloss_reduced_2022, ivagnes_shape_2023, khamlich_physics-based_2023, donoghue_multi-fidelity_2022, barnett_mitigating_2023, zucatti_adaptive_2024}, and machine learning \cite{tao_application_2019, lee_model_2020, zhang_application_2022, wen2023reducedorder, bonneville2024comprehensive}, to name a few. In this work, we consider the projection-based model reduction with empirical quadrature procedure (EQP) \cite{yano_discontinuous_2019} hyperreduction (to accelerate assembly of nonlinear terms) as surrogates to accelerate the optimization iterations.


Traditionally, computational efficiency has been realized in the context of projection-based model reduction by leveraging an \textit{offline/online} approach, whereby a ROM is trained in an expensive offline phase and queried many times in an inexpensive online phase. Several approaches have been developed to utilize ROMs for large-scale optimization that preserve the offline/online decomposition \cite{legresley_airfoil_2000, manzoni_shape_2012, ripepi_reduced-order_2018,trehan_trajectory_2016,stefanescu_poddeim_2015,yang_efficient_2017,esmaeili_generalized_2020,amsallem_design_2015,choi_gradient-based_2020,scheffold_vibration_2018,renganathan_koopman-based_2020}. In these approaches, the ROM is constructed by sampling the full-fidelity model at many training points in the parameter space and the optimization problem is solved with the ROM instead of the original model. These approaches involve a computationally intensive offline phase that must train the ROM in all regions of the parameter space that may be visited by the optimizer, which can be difficult to amortize in the online phase. In addition, the optimal solution found is a critical point of the ROM, not the original model. On the contrary, the \textit{on-the-fly} methods in \cite{arian_trust-region_2000, yue_accelerating_2013, agarwal_trust-region_2013, zahr_progressive_2015, zahr_adaptive_2016, qian_certified_2017, zahr_efficient_2019, yano_globally_2021,keil_non-conforming_2021,banholzer_trust_2022, donoghue_multi-fidelity_2022, wen_globally_2023} do not require an offline training phase because they are built during the solution of the optimization problem. Furthermore, on-the-fly methods can guarantee global convergence to a critical point of the original model \cite{zahr_adaptive_2016,qian_certified_2017,yano_globally_2021,banholzer_trust_2022, wen_globally_2023}. However, these methods have been developed for problems without side constraints.

A common approach to solve side-constrained problems is the \textit{augmented Lagrangian method (ALM)} or \textit{method of multipliers} first proposed by Hestenes \cite{Hestenes.1969} and Powell \cite{Powell.1978}, and since studied by many authors \cite{Rockafellar.1973, bertsekas1989augmented, Birgin.2008}. In recent decades, ALM and its variations have been used to solve increasingly complex optimization problems (\cite{fortin.1983,Rodriguez.1998, Zhao.2010, Chan.2011, Li.2013, Wang.2014, Curtis.2015}). We consider the \textit{bound-constrained Lagrangian (BCL)} method, which is the foundation of the LANCELOT package \cite{conn_lancelot}, due to its simplicity and effectivity in the field of nonlinear programming. The BCL method converts the original constrained problem into an unconstrained or bound-constrained one. It utilizes an appropriate solver, such as modified trust-region methods \cite{Conn.1988, Lin.1999, Conn.1988lq, aguiloa_trust_nodate} for bound-constrained problems, to solve the augmented Lagrangian subproblem. In this work, we adapt the on-the-fly trust-region method developed in \cite{wen_globally_2023} to solve the augmented Lagrangian subproblem to yield a globally convergent method for using EQP-based surrogates to accelerate optimization problems with side constraints.

The main contributions of this work are three-fold. First, we develop a novel trust-region method for bound-constrained optimization problems that allows for inexact gradient evaluations equipped with asymptotic error bound of \cite{kouri_trust-region_2013,kouri_inexact_2021}. We use this method to solve the augmented Lagrangian subproblems that arise for optimization problems with general side constraints. Second, we extend the EQP method of \cite{yano_discontinuous_2019,wen_globally_2023} with additional linear constraints (training phase) that ensure the optimization Lagrangian and side constraints are well-approximated on the reduced quadrature rule. Finally, we specialize the proposed trust-region method to the setting where EQP surrogates built on-the-fly are used as the trust-region model (EQP/BTR). By adaptively building the reduced basis and using EQP tolerances rigorously informed by the trust-region convergence theory, the EQP/BTR avoids a potentially expensive training phase and global convergence is guaranteed. At each augmented Lagrangian step (outer loop), this EQP/BTR method is used to solve the augmented Lagrangian subproblem (inner loop) for problems with general side constraints. Global convergence and the efficiency of the method is demonstrated on two aerodynamic shape optimization examples.

The remainder of this paper is organized as follows. Section \ref{sec:pdeopt} introduces the governing nonlinear system of equations, poses the general optimization problem with side constraints, and formulates the reduced-space augmented Lagrangian subproblem. Section~\ref{sec:hyperreduction} introduces projection-based model reduction and EQP-based hyperreduction, including the extension of the method in \cite{wen_globally_2023} with augmented Lagrangian terms and relevant error estimates. Section~\ref{sec:trammo} introduces the new bound-constrained trust-region methods that allows for inexact gradient evaluations and asymptotic error bounds, a brief review of the augmented Lagrangian method in \cite{nocedal_numerical_2006}, and the proposed EQP/BTR method and establishes its global convergence condition. Finally, Section \ref{sec:numexp} demonstrates the efficiency and convergence of the proposed method on two aerodynamic shape optimization problems.

\section{Problem formulation}
\label{sec:pdeopt}
In this section, we formulate the governing high-dimensional optimization problem that we aim to accelerate using model hyperreduction in later sections. We consider general optimization problems constrained by a large-scale nonlinear system of equations, although our primary interest is fully discrete PDE-constrained optimization. To this end, we introduce the governing nonlinear system of equations (Section~\ref{sec:pdeopt:gov_eqn}) and the formulation of the complete optimization problem (Section~\ref{sec:pdeopt:auglag}), including an augmented Lagrangian formulation of the problem and adjoint method to compute gradients.

\subsection{Governing equations}
\label{sec:pdeopt:gov_eqn}
We consider a parametrized, large-scale system of nonlinear equations that we will refer to as the high-dimensional model (HDM):
given a collection of system parameters
$\mubold\in\Dcal\subset\Rbb^{N_\mubold}$, find $\ubm^\star\in\Rbb^{N_\ubm}$ such that
\begin{equation}\label{eqn:hdm_res}
    \rbm(\ubm^\star, \mubold) = \zerobold,
\end{equation}
where $\ubm^\star$ is the primal solution implicitly defined as the solution of (\ref{eqn:hdm_res}) and  $\rbm: \Rbb^{N_\ubm}\times\Dcal \mapsto \Rbb^{N_\ubm}$ with $\rbm : (\ubm, \mubold) \mapsto \rbm(\ubm,\mubold)$ denotes the residual. We assume that for every $\mubold\in\Dcal$, there exists a unique primal solution $\ubm^\star=\ubm^\star(\mubold)$ satisfying (\ref{eqn:hdm_res}) and that the implicit map $\mubold \mapsto \ubm^\star(\mubold)$ is continuously differentiable. In most practical applications, the dimension $N_\ubm$ is large, which causes (\ref{eqn:hdm_res}) to be computationally expensive to solve. We focus on the case where the residual arises from the high-fidelity discretization of a parametrized, partial differential equation (PDE), although the developments in this work generalize to naturally discrete systems. Furthermore, we assume the system ($\Omega$) consists of $N_\mathtt{e}$ elements ($\Omega_e$), i.e., $\Omega=\cup_{e=1}^{N_\mathtt{e}} \Omega_e$, and the residual can be written as an assembly of element residuals
\begin{equation}\label{eqn:hdm_elem_res}
    \rbm(\ubm, \mubold)  
    = \sum_{e=1}^{N_\mathtt{e}} \Pbm_e \rbm_e \paren{\ubm_e, \ubm_e', \mubold},
\end{equation}
where $\rbm_e: \Rbb^{N_\ubm^\mathtt{e}} \times \Rbb^{N_\ubm^{\mathtt{e}'}}\times\Dcal \rightarrow \Rbb^{N_\ubm^\mathtt{e}}$ with
$\rbm_e : (\ubm_e, \ubm_e', \mubold) \mapsto \rbm_e(\ubm_e,\ubm_e',\mubold)$
is the residual contribution of element $e$, $\ubm_e\in \Rbb^{N_\ubm^\mathtt{e}}$ are the DoFs associated with element
$e$, $\ubm_e'\in\Rbb^{N_\ubm^{\mathtt{e}'}}$ are the DoFs associated with elements neighboring element $e$ (if applicable),
and $\Pbm_e \in \Rbb^{N_\ubm\times N_\ubm^\mathtt{e}}$ is the assembly operator that maps element DoFs for
element $e$ to global DoFs. The element and global DoFs are related via assembly operators
$\ubm_e = \Pbm_e^T \ubm$ and $\ubm_e' = (\Pbm_e')^T\ubm$, where $\Pbm_e' \in \Rbb^{N_\ubm\times N_\ubm^{\mathtt{e}'}}$
is an assembly operator that maps element DoFs from neighbors of element $e$ to the corresponding
global DoFs. This structure will be used to facilitate hyperreduction of the nonlinear system of equations.

\begin{remark}
\label{rem:govern}
There are many PDE discretizations that possess the elemental decomposition
\ref{eqn:hdm_elem_res}, including cell-centered finite volume and most
finite-element-based discretizations, as well as naturally discrete systems
(e.g., direct stiffness analysis of trusses). In this work, use a discontinuous
Galerkin (DG) discretizations of the Euler equations of gasdynamics.
\end{remark}


\subsection{Optimization formulation}\label{sec:pdeopt:auglag}
Now, we pose the constrained optimization problem considered in this work as
\begin{equation}\label{eqn:hdm_full_space}
    \optconThree{\ubm \in \Rbb^{N_\ubm}, \mubold \in \Dcal}{j(\ubm, \mubold)}{\cbm(\ubm, \mubold) = \zerobold}{\rbm(\ubm, \mubold)=\zerobold}{\mubold_l \leq \mubold \leq \mubold_u,}
\end{equation}
where $\func{j}{\Rbb^{N_\ubm}\times\Dcal}{\Rbb}$ is the quantity to be minimized,
$\cbm : (\ubm,\mubold) \rightarrow \cbm(\ubm, \mubold)$ is a vector of $N_\cbm$
side constraints (i.e., constraints other than the governing equation $\rbm$
involving the primal variable $\ubm$), and $\mubold_l$ and $\mubold_u$ are vectors
of lower and upper bounds, respectively, on the design parameters. In the setting
where the residual $\rbm$ corresponds to a PDE discretization, this is a
\textit{discrete PDE-constrained optimization problem} (discretize-then-optimize).
We assume the objective function and side constraints can be written as a summation of
element contributions, i.e.,
\begin{equation}
 j(\ubm,\mubold) = \sum_{e=1}^{N_\mathtt{e}} j_e(\ubm_e,\mubold), \qquad
 \cbm(\ubm,\mubold) = \sum_{e=1}^{N_\mathtt{e}} \cbm_e(\ubm_e,\mubold),
\end{equation}
where $j_e : \Rbb^{N_\ubm^\mathtt{e}} \times \Dcal \rightarrow \Rbb$ and
$\cbm_e : \Rbb^{N_\ubm^\mathtt{e}} \times \Dcal \rightarrow \Rbb^{N_\cbm}$
are the element contributions of element $e$ to the objective function and
side constraints, respectively. Such a decomposition commonly arises in
PDE-constrained optimization applications where the objective and side constraints
are defined as integrals over the computational domain or its boundary.

\begin{remark}
The optimization problem (\ref{eqn:hdm_full_space}) is sufficiently general to
incorporate inequality side constraints of the form $\dbm(\ubm,\mubold) \geq 0$.
In this case, the parameter vector $\mubold = (\nubold, \sbm)$, where $\nubold$
are design parameters from the governing equations and $\sbm$ are slack variables
\cite{nocedal_numerical_2006}. In this case, the side constraints read
$\cbm(\ubm,\mubold) = \dbm(\ubm, \mubold) - \sbm$ and the lower bound reads
$\mubold_l = (\nubold_l, \zerobold)$, where $\nubold_l$ is the vector of
lower bounds for the design parameters $\nubold$.
\end{remark}

We convert (\ref{eqn:hdm_full_space}) to an optimization problem without side constraints using the classic augmented Lagrangian (AL) function, $\func{\ell}{\Rbb^{N_\ubm}\times\Dcal\times\Rbb^{N_\cbm}\times\Rbb}{\Rbb}$, defined as
\begin{equation}\label{eqn:auglag}
  \ell : (\ubm, \mubold; \thetabold, \tau) \mapsto
        \ell^{\mathtt{L}}(\ubm,\mubold;\thetabold) + \frac{\tau}{2} \norm{\cbm(\ubm, \mubold)}^2_2
\end{equation}
where the Lagrangian part, $\func{\ell^{\mathtt{L}}}{\Rbb^{N_\ubm}\times\Dcal\times\Rbb^{N_\cbm}}{\Rbb}$, is defined as 
\begin{equation}
    \ell^{\mathtt{L}} : (\ubm,\mubold;\thetabold) \mapsto j(\ubm, \mubold) - \thetabold^T \cbm(\ubm, \mubold),
\end{equation}
where $\thetabold \in \Rbb^{N_\cbm}$ are Lagrangian multiplier estimates, and $\tau \in \Rbb$ denotes the penalty parameter. Then, we replace the optimization problem in (\ref{eqn:hdm_full_space}) with a sequence of optimization problems of the form
\begin{equation}\label{eqn:hdm_full_space_auglag}
    \optconTwo{\ubm \in \Rbb^{N_\ubm}, \mubold \in \Dcal}{\ell(\ubm, \mubold;\thetabold,\tau)}{\rbm(\ubm, \mubold)=\zerobold}{\mubold_l \leq \mubold \leq \mubold_u,}
\end{equation}
where $\thetabold$ and $\tau$ are fixed. We eliminate the HDM constraint by
restricting the primal variable to the solution manifold of $\rbm$, i.e., define 
$f: \Dcal \times \Rbb^{N_\cbm} \times \Rbb \rightarrow \Rbb$ 
such that
\begin{equation}
    f : (\mubold;\thetabold,\tau) \mapsto \ell(\ubm^\star(\mubold), \mubold;\thetabold,\tau),
\end{equation}
which leads to the following bound-constrained, reduced-space optimization problem
\begin{equation}\label{eqn:auglag_reduced}
    \optunc{\mubold \in \Ccal}{f(\mubold;\thetabold,\tau) \coloneqq \ell(\ubm^\star(\mubold), \mubold; \thetabold,\tau)},
\end{equation}
where $\Ccal$ is the set of bound constraints defined by
\begin{equation}\label{eqn:bound_con}
\Ccal \coloneqq \curlyb{\mubold \in \Dcal \suchthat \mubold_l \leq \mubold \leq \mubold_u}.
\end{equation}

We use the adjoint method to compute the gradient of the objective function because the
parameter-space dimension $N_\mubold$ can be large. The corresponding HDM adjoint problem
is: given $\mubold\in\Dcal$ and the corresponding primal solution
$\ubm^\star\in\Rbb^{N_\ubm}$ satisfying (\ref{eqn:hdm_res}), find the adjoint
solution $\lambdabold^\star\in\Rbb^{N_\ubm}$ satisfying
\begin{equation} \label{eqn:hdm_adj_eqn}
 \rlam(\lambdabold^\star,\ubm^\star,\mubold;\thetabold,\tau) = \zerobold,
\end{equation}
where the adjoint residual, $\rlam : \Rbb^{N_\ubm}\times\Rbb^{N_\ubm}\times \Dcal \times \Rbb^{N_\cbm} \times \Rbb \rightarrow \Rbb^{N_\ubm}$, is defined as
\begin{equation} \label{eqn:hdm_adj_res}
 \rlam : (\zbm, \ubm, \mubold; \thetabold, \tau) \mapsto
 \pder{\rbm}{\ubm}(\ubm,\mubold)^T\zbm - \pder{\ell}{\ubm}(\ubm,\mubold;\thetabold,\tau)^T.
\end{equation}
We expand $\pder{\ell}{\ubm}(\ubm,\mubold;\thetabold,\tau)$ and define the Lagrangian part, $\rbm^{\mathtt{L},\lambda}: \Rbb^{N_\ubm}\times\Rbb^{N_\ubm}\times \Dcal \times \Rbb^{N_\cbm} \rightarrow \Rbb^{N_\ubm}$, as
\begin{equation}\label{eqn:hdm_al_lag}
    \rbm^{\mathtt{L},\lambda} : (\zbm, \ubm, \mubold;\thetabold) \mapsto \pder{\rbm}{\ubm}(\ubm,\mubold)^T\zbm - \pder{j}{\ubm}(\ubm, \mubold)^T + \pder{\cbm}{\ubm}(\ubm,\mubold)^T\thetabold,
\end{equation}
which reduces (\ref{eqn:hdm_adj_res}) to
\begin{equation}\label{eqn:hdm_adj_res_regroup}
    \rlam(\zbm, \ubm, \mubold; \thetabold, \tau) = \rbm^{\mathtt{L},\lambda}(\zbm, \ubm, \mubold;\thetabold) - \tau\pder{\cbm}{\ubm}(\ubm, \mubold)^T\cbm(\ubm, \mubold).
\end{equation}
Assuming the residual Jacobian is well-defined and invertible for every $(\ubm^\star(\mubold),\mubold)$ pair with $\mubold\in\Dcal$,
then there exists a unique adjoint solution
\begin{equation}
 \lambdabold^\star(\mubold;\thetabold, \tau) = \pder{\rbm}{\ubm}(\ubm^\star(\mubold), \mubold)^{-T}\pder{\ell}{\ubm}(\ubm^\star(\mubold),\mubold;\thetabold,\tau)^T.
\end{equation}
making the implicit map $(\mubold;\thetabold,\tau)\mapsto\lambdabold^\star(\mubold;\thetabold,\tau)$ well-defined for all $\mubold\in\Dcal$.
From the primal-adjoint pair $(\ubm^\star,\lambdabold^\star)$ satisfying (\ref{eqn:hdm_res}) and (\ref{eqn:hdm_adj_eqn}),
the gradient of the reduced AL function, $\func{\nabla f}{\Dcal \times \Rbb^{N_\cbm} \times \Rbb}{\Rbb^{N_\mubold}}$, is computed as
\begin{equation}\label{eqn:hdm_grad}
 \nabla f : (\mubold;\thetabold,\tau) \mapsto \glam(\lambdabold^\star(\mubold;\thetabold,\tau), \ubm^\star(\mubold), \mubold;\thetabold,\tau)^T = \sens{\ell}(\ubm^\star(\mubold), \mubold; \thetabold, \tau)^T -  \sens{\rbm}(\ubm^\star(\mubold), \mubold)^T\lambdabold^\star(\mubold;\thetabold,\tau),
\end{equation}
where the operator that reconstructs the gradient from the adjoint solution,
$\glam : \Rbb^{N_\ubm}\times\Rbb^{N_\ubm}\times \Dcal \times \Rbb^{N_\cbm} \times \Rbb \rightarrow \Rbb^{1\times N_\mubold}$, is
\begin{equation}
 \glam : (\zbm, \ubm, \mubold;\thetabold,\tau) \mapsto \sens{\ell}(\ubm, \mubold; \thetabold, \tau) - \zbm^T \sens{\rbm}(\ubm, \mubold).
\end{equation}
Similarly to (\ref{eqn:hdm_al_lag}) and (\ref{eqn:hdm_adj_res_regroup}), we define $\gbm^{\mathtt{L},\lambda} : \Rbb^{N_\ubm}\times\Rbb^{N_\ubm}\times \Dcal \times \Rbb^{N_\cbm}\rightarrow \Rbb^{1\times N_\mubold}$ as:
\begin{equation}\label{eqn:hdm_grad_lag}
\gbm^{\mathtt{L},\lambda} : (\zbm, \ubm, \mubold;\thetabold) \mapsto \pder{j}{\mubold}(\ubm,\mubold) - \thetabold^T\pder{\cbm}{\mubold}(\ubm,\mubold) - \zbm^T\pder{\rbm}{\mubold}(\ubm,\mubold),
\end{equation}
and the gradient becomes
\begin{equation}\label{eqn:hdm_grad_regroup}
    \glam(\zbm, \ubm, \mubold;\thetabold,\tau) = \gbm^{\mathtt{L},\lambda}(\zbm,\ubm,\mubold;\thetabold) + 
    \tau\cbm(\ubm,\mubold)^T\pder{\cbm}{\mubold}(\ubm,\mubold).
\end{equation}

Next, we use the disassembly operator $\Pbm$ to write the primal residual (\ref{eqn:hdm_elem_res}), AL function (\ref{eqn:auglag}), adjoint residual (\ref{eqn:hdm_adj_res_regroup}) and AL gradient (\ref{eqn:hdm_grad}) in unassembled form by considering the Lagrangian and penalty contributions separately. The unassembled form will be used to construct optimization-informed hyperreduced models in the next section.
The AL function can be expanded into elemental contributions as
\begin{equation}\label{eqn:hdm_elem_qoi}
\ell(\ubm,\mubold;\thetabold,\tau) 
= \sum^{N_\mathtt{e}}_{e=1} \ell_{e}^{\mathtt{L}}(\ubm_e,\mubold;\thetabold)+\frac{\tau}{2}\norm{\sum^{N_\mathtt{e}}_{e=1} \cbm_e(\ubm_e,\mubold)}^2_2,
\end{equation}
where the elemental contribution to the Lagrangian function, $\ell_{e}^{\mathtt{L}} : \Rbb^{N_\ubm^\mathtt{e}}\times\Dcal\times\Rbb^{N_\cbm} \rightarrow \Rbb$, is defined as
\begin{equation}
\ell_{e}^{\mathtt{L}} : (\ubm_e,\mubold;\thetabold) \mapsto j_e(\ubm_e,\mubold)-\thetabold^T\cbm_e(\ubm_e,\mubold).
\end{equation}
Similarly, the adjoint residual can be written as
\begin{equation}\label{eqn:hdm_elem_adj_res}
\rlam(\zbm, \ubm, \mubold;\thetabold,\tau)
= \sum^{N_\mathtt{e}}_{e=1} \rbm_{e}^{\mathtt{L},\lambda}(\zbm_e,\ubm_e,\mubold;\thetabold)
    -\tau\bracket{\sum^{N_\mathtt{e}}_{e=1}\Pbm_e \pder{\cbm_e}{\ubm_e}(\ubm_e,\mubold)^T}\bracket{\sum^{N_\mathtt{e}}_{e=1}\cbm_e(\ubm_e,\mubold)},
\end{equation}
where the elemental contribution to the Lagrangian adjoint residual,
$\rbm_e^{\mathtt{L},\lambda} : \Rbb^{N_\ubm^\mathtt{e}}\times\Rbb^{N_\ubm^\mathtt{e}}\times\Dcal\times\Rbb^{N_\cbm}\rightarrow \Rbb^{N_\ubm^\mathtt{e}}$, is defined as
\begin{equation}
\rbm_{e}^{\mathtt{L},\lambda} : (\zbm_e, \ubm_e, \mubold; \thetabold) \mapsto
    -\Pbm_e\pder{j_e}{\ubm_e}(\ubm_e,\mubold)^T
    +\Pbm_e\pder{\cbm_e}{\ubm_e}(\ubm_e,\mubold)^T\thetabold
    +\Pbm_e\pder{\rbm_e}{\ubm_e}(\ubm_e, \ubm_e', \mubold)^T \zbm_e + \Pbm_e'\pder{\rbm_e}{\ubm_e'}(\ubm_e, \ubm_e',\mubold)^T \zbm_e
\end{equation}
and $\zbm_e = \Pbm_e^T\zbm$. Finally, we can rewrite the gradient as
\begin{equation}\label{eqn:hdm_elem_grad}
\glam(\zbm,\ubm,\mubold;\thetabold,\tau) 
= \sum^{N_\mathtt{e}}_{e=1} \gbm_{e}^{\mathtt{L},\lambda}(\zbm_e,\ubm_e,\mubold;\thetabold)+\tau\bracket{\sum^{N_\mathtt{e}}_{e=1}\cbm_e(\ubm_e,\mubold)}^T\bracket{\sum_{e=1}^{N_\mathtt{e}}\sens{\cbm_e}(\ubm_e,\mubold)},
\end{equation}
where the elemental contribution to the gradient,
$\gbm_{e}^{\mathtt{L},\lambda} : \Rbb^{N_\ubm^\mathtt{e}}\times\Rbb^{N_\ubm^\mathtt{e}}\times\Dcal\times\Rbb^{N_\cbm}\rightarrow \Rbb^{1\times N_\mubold}$, is defined as
\begin{equation}
\gbm_{e}^{\mathtt{L},\lambda}(\zbm_e,\ubm_e,\mubold;\thetabold) \mapsto
\sens{j_e}(\ubm_e,\mubold)-\thetabold^T\sens{\cbm_e}(\ubm_e,\mubold)-\zbm_e^T\sens{\rbm_e}(\ubm_e, \ubm_e', \mubold)
\end{equation}

\begin{remark} \label{rem:auglag}
The quadratic penalty in the augmented Lagrangian (\ref{eqn:auglag}) introduces a nonlinear dependence on the elemental quantities that precludes its direct elemental expansion of the
form $\ell(\ubm,\mubold;\thetabold,\tau) = \sum_{e=1}^{N_\mathtt{e}} \ell_e(\ubm,\mubold;\thetabold,\tau)$. This requires a departure from previous work in this area \cite{wen_globally_2023} and necessitates the term-by-term elemental expansion described in Section~\ref{sec:pdeopt:auglag}.
\end{remark}

\section{Hyperreduction}
\label{sec:hyperreduction}
In this section, we enhance the empirical quadrature procedure (EQP) developed in \cite{yano_discontinuous_2019,wen_globally_2023} with additional constraints required in the optimization setting that will be used to prove global convergence when embedded in a trust-region framework. To this end, we introduce standard projection-based model reduction (Section \ref{sec:hyperreduction:rom}), the adapted EQP method (Section~\ref{sec:hyperreduction:lp_eqp}), and corresponding error estimates (Section~\ref{sec:hyperreduction:eqp_err_est}).

\subsection{Projection-based model reduction}
\label{sec:hyperreduction:rom}
Projection-based model reduction begins with the ansatz that the primal state lies in a low-dimensional subspace
$\Vboldcal_\Phibold = \{\Phibold \hat\ybm \mid \hat\ybm\in\Rbb^n\} \subset \Rbb^{N_\ubm}$, where $\Phibold \in \Rbb^{N_\ubm \times n}$ with $n \ll N_\ubm$ is the reduced basis. That is, 
we approximate the primal state $\ubm^\star$ as
\begin{equation} \label{eqn:rom_ansatz}
    \ubm^\star \approx \hat\ubm_\Phibold^\star \coloneqq \Phibold \hat\ybm^\star,
\end{equation}
where $\hat\ubm_\Phibold^\star\in\Vboldcal_\Phibold$ is the subspace approximation of the primal
state $\ubm_\star$ and $\hat\ybm_\Phibold^\star \in \Rbb^n$ contains the corresponding reduced coordinates.
The primal reduced coordinates are implicitly defined as the solution of the Galerkin reduced-order model:
given $\mubold\in\Dcal$, find $\hat\ybm_\Phibold^\star\in\Rbb^n$ such that
\begin{equation}\label{eqn:rom_res_eqn}
    \hrbm_\Phibold(\hat\ybm_\Phibold^\star, \mubold) = \zerobold,
\end{equation}
where $\hrbm_\Phibold : \Rbb^{n}\times\Dcal \rightarrow \Rbb^n$ with
\begin{equation}\label{eqn:rom_res}
    \hrbm_\Phibold : (\hat\ybm,\mubold) \mapsto \Phibold^T \rbm(\Phibold\hat\ybm, \mubold),
\end{equation}
which is obtained by substituting the ROM ansatz (\ref{eqn:rom_ansatz}) into the governing equation (\ref{eqn:hdm_res})
and requiring the resulting residual to be orthogonal to the reduced subspace (Galerkin projection). We assume that for
every $\mubold\in\Dcal$, there exists a unique primal solution $\hat\ybm_\Phibold^\star=\hat\ybm_\Phibold^\star(\mubold)$
satisfying (\ref{eqn:rom_res_eqn}) and that the implicit map $\mubold \mapsto \hat\ybm_\Phibold^\star(\mubold)$ is continuously
differentiable. The reduced residual ($\hrbm_\Phibold$) inherits an unassembled structure from the HDM
\begin{equation} \label{eqn:rom_res_elem}
 \hrbm_\Phibold(\hat\ybm, \mubold) =
 \sum_{e=1}^{N_\mathtt{e}} \Phibold_e^T \rbm_\Phibold(\Phibold_e\hat\ybm, \Phibold_e'\hat\ybm, \mubold).
\end{equation}
where $\Phibold_e \coloneqq \Pbm_e^T\Phibold \in \Rbb^{N_\ubm^\mathtt{e} \times n}$ and
$\Phibold_e' \coloneqq (\Pbm_e')^T\Phibold \in \Rbb^{N_\ubm^{\mathtt{e}'} \times n}$ are the
reduced bases restricted to the elemental degrees of freedom.

Each optimization functional is reduced by substituting the ROM approximation into the original functional, including the augmented Lagrangian function. 
That is, let
\begin{equation} \label{eqn:rom_f}
\hat{\ell}_\Phibold : (\hat\ybm,\mubold;\thetabold,\tau) \mapsto \ell(\Phibold\hat\ybm,\mubold;\thetabold, \tau), \quad
\hat{\ell}^{L}_\Phibold : (\hat\ybm,\mubold;\thetabold) \mapsto \ell^{L}(\Phibold\hat\ybm,\mubold;\thetabold), \quad
\hat{f}_\Phibold : (\mubold;\thetabold,\tau) \mapsto \hat{\ell}_\Phibold(\hat\ybm_\Phibold^\star(\mubold),\mubold;\thetabold,\tau),
\end{equation}
where $\hat{\ell}_\Phibold : \Rbb^n\times\Dcal \times \Rbb^{N_\cbm} \times \Rbb \rightarrow \Rbb$ is the reduced augmented Lagrangian function, $\hat{\ell}^\mathtt{L}_\Phibold : \Rbb^n\times\Dcal \times \Rbb^{N_\cbm} \rightarrow \Rbb$ is the Lagrangian component of $\hat{\ell}_\Phibold$, and $\func{\hat f_\Phibold}{\Dcal \times \Rbb^{N_\cbm} \times \Rbb}{\Rbb}$ is the restriction of the reduced augmented Lagrangian function to the ROM solution manifold.
The reduced AL function can be expanded as
\begin{equation}
  \hat{\ell}_\Phibold(\hat\ybm,\mubold;\thetabold,\tau) = \ell^{\mathtt{L}}_{\Phibold}(\hat\ybm,\mubold;\thetabold) + \frac{\tau}{2} \norm{\cbm(\Phibold\hat\ybm, \mubold)}^2_2
\end{equation}
and inherits an unassembled structure from the HDM
\begin{equation} \label{eqn:rom_elem_f}
 \hat{\ell}_\Phibold(\hat\ybm,\mubold;\thetabold,\tau) 
= \sum^{N_\mathtt{e}}_{e=1} \hat{\ell}^\mathtt{L}_{\Phibold,e}(\hat\ybm,\mubold;\thetabold) +\frac{\tau}{2}\norm{\sum^{N_\mathtt{e}}_{e=1} \cbm_e(\Phibold_e\hat\ybm,\mubold)}^2_2,
\end{equation}
where the elemental contribution to the AL function, $\hat{\ell}^\mathtt{L}_{\Phibold,e} : \Rbb^n \times \Dcal \times \Rbb^{N_\cbm} \rightarrow \Rbb$, is defined as
\begin{equation}
\hat{\ell}^\mathtt{L}_{\Phibold,e} : (\hat\ybm,\mubold;\thetabold) \mapsto
 j_e(\Phibold_e\hat\ybm,\mubold)-\thetabold^T \cbm_e(\Phibold_e\hat\ybm,\mubold).
\end{equation}

The reduced adjoint residual, $\hrlam_\Phibold : \Rbb^n\times\Rbb^n\times\Dcal \times \Rbb^{N_\cbm} \times \Rbb \rightarrow \Rbb^n$, is defined as
\begin{equation} \label{eqn:rom_adj_res}
 \hrlam_\Phibold : (\hat\zbm,\hat\ybm,\mubold;\thetabold,\tau) 
 \mapsto \pder{\hat\rbm_\Phibold}{\hat\ybm}(\hat\ybm,\mubold)^T\hat\zbm - \pder{\hat{\ell}_\Phibold}{\hat\ybm}(\hat\ybm,\mubold;\thetabold,\tau)^T
  = \Phibold^T \rlam(\Phibold\hat\zbm,\Phibold\hat\ybm,\mubold;\thetabold,\tau),
\end{equation}
where the equality follows directly from the definitions in (\ref{eqn:hdm_adj_res}), (\ref{eqn:rom_res}), and (\ref{eqn:rom_f}).
Following (\ref{eqn:hdm_al_lag}) and (\ref{eqn:hdm_adj_res_regroup}), we let the Lagrangian component $\rbm^{\mathtt{L},\lambda}_\Phibold: \Rbb^n \times\Rbb^n \times \Dcal \times \Rbb^{N_\cbm} \rightarrow \Rbb^n$ be
\begin{equation}
  \hat\rbm^{\mathtt{L},\lambda}_\Phibold : (\hat\zbm,\hat\ybm,\mubold;\thetabold)
  \mapsto \pder{\hat\rbm_\Phibold}{\hat\ybm}(\hat\ybm,\mubold)^T\hat\zbm - \Phibold^T\pder{j}{\ubm}(\Phibold\hat\ybm, \mubold)^T + \bracket{\pder{\cbm}{\ubm}(\Phibold\ybm,\mubold)\Phibold}^T\thetabold = \Phibold^T\rbm^{\mathtt{L},\lambda}(\Phibold\hat\zbm,\Phibold\hat\ybm,\mubold; \thetabold),
\end{equation}
and (\ref{eqn:rom_adj_res}) becomes
\begin{equation}
\hrlam_\Phibold(\hat\zbm,\hat\ybm,\mubold;\thetabold,\tau) 
= \hat\rbm^{\mathtt{L},\lambda}_\Phibold(\hat\zbm,\hat\ybm,\mubold;\thetabold) - \tau \Phibold^T\pder{\cbm}{\ubm}(\Phibold\hat\ybm, \mubold)^T\cbm(\Phibold\hat\ybm, \mubold).
\end{equation}
The reduced adjoint problem is: given $\mubold\in\Dcal$ and the corresponding reduced primal solution
$\hat\ybm_\Phibold^\star$ satisfying (\ref{eqn:rom_res_eqn}), find the reduced adjoint solution $\hat\lambdabold_\Phibold^\star\in\Rbb^n$
satisfying
\begin{equation}
 \hrlam_\Phibold(\hat\lambdabold_\Phibold^\star, \hat\ybm_\Phibold^\star, \mubold;\thetabold,\tau) = \zerobold.
\end{equation}
Assuming the reduced residual Jacobian is well-defined and invertible for every $(\hat\ybm_\Phibold^\star(\mubold),\mubold)$
pair with $\mubold\in\Dcal$, there exists a unique reduced adjoint solution
\begin{equation}
 \hat\lambdabold_\Phibold^\star(\mubold;\thetabold,\tau) = \pder{\hat\rbm_\Phibold}{\hat\ybm}(\hat\ybm_\Phibold^\star(\mubold),\mubold)^{-T}\pder{\hat{\ell}_\Phibold}{\hat\ybm}(\hat\ybm_\Phibold^\star(\mubold),\mubold;\thetabold,\tau)^T,
\end{equation}
making the implicit map $(\mubold;\thetabold,\tau)\mapsto \hat\lambdabold_\Phibold^\star(\mubold;\thetabold,\tau)$ well-defined for all $\mubold\in\Dcal$.
The reduced adjoint residual inherits an unassembled structure from the HDM using
the relationships in (\ref{eqn:hdm_elem_adj_res}) and (\ref{eqn:rom_adj_res})
\begin{equation} \label{eqn:rom_elem_adj_res}
\hrlam_\Phibold(\hat\zbm,\hat\ybm,\mubold;\thetabold,\tau) 
= \sum^{N_\mathtt{e}}_{e=1} \hat\rbm_{\Phibold,e}^{\mathtt{L},\lambda}(\hat\zbm,\hat\ybm,\mubold;\thetabold) -\tau\bracket{\sum^{N_\mathtt{e}}_{e=1}\Phibold_e^T \pder{\cbm_e}{\ubm_e}(\Phibold_e\hat\ybm,\mubold)^T}\bracket{\sum^{N_\mathtt{e}}_{e=1}\cbm_e(\Phibold_e\hat\ybm,\mubold)}.
\end{equation}
where the elemental contribution, $\hat\rbm_{\Phibold,e}^{\mathtt{L},\lambda} : \Rbb^n \times \Rbb^n \times \Dcal \times \Rbb^{N_\cbm} \rightarrow \Rbb^n$, is defined as
\begin{equation}
\begin{aligned}
\hat\rbm_{\Phibold,e}^{\mathtt{L},\lambda} : (\hat\zbm,\hat\ybm,\mubold;\thetabold) 
&\mapsto -\Phibold_e^T\pder{j_e}{\ubm_e}(\Phibold_e\hat\ybm,\mubold)^T
   +\Phibold_e^T\pder{\cbm_e}{\ubm_e}(\Phibold_e\hat\ybm,\mubold)^T\thetabold
   +\Phibold_e^T\pder{\rbm_e}{\ubm_e}(\Phibold_e\hat\ybm, \Phibold_e'\hat\ybm, \mubold)^T \Phibold_e \hat\zbm +\\
&\quad \paren{\Phibold_e'}^T\pder{\rbm_e}{\ubm_e'}(\Phibold_e\hat\ybm, \Phibold_e'\hat\ybm, \mubold)^T \Phibold_e\hat\zbm.
\end{aligned}
\end{equation}
From the primal-adjoint pair ($\hat\ybm_\Phibold^\star,\hat\lambdabold_\Phibold^\star$), the gradient of the reduced AL ($\hat{f}_\Phibold$),
$\nabla \hat{f}_\Phibold : \Dcal \times \Rbb^{N_\cbm} \times \Rbb \rightarrow {\Rbb^{N_\mubold}}$, is computed as
\begin{equation}
 \nabla \hat{f}_\Phibold : (\mubold;\thetabold,\tau) \mapsto \hat\gbm_\Phibold^\lambda(\hat\lambdabold_\Phibold^\star(\mubold;\thetabold,\tau), \hat\ybm_\Phibold^\star(\mubold), \mubold;\thetabold,\tau)^T =
 \pder{\hat{\ell}_\Phibold}{\mubold}(\hat\ybm_\Phibold^\star(\mubold),\mubold;\thetabold,\tau)^T - \pder{\hat\rbm_\Phibold}{\mubold}(\hat\ybm_\Phibold^\star(\mubold),\mubold)^T\hat\lambdabold_\Phibold^\star(\mubold;\thetabold,\tau),
\end{equation}
where the operator that reconstructs the reduced AL gradient from the reduced adjoint solution, $\hat\gbm_\Phibold^\lambda : \Rbb^n\times\Rbb^n\times\Dcal \times \Rbb^{N_\cbm} \times \Rbb \rightarrow \Rbb^{1\times N_\mubold}$, is
\begin{equation}
 \hat\gbm_\Phibold^\lambda : (\hat\zbm, \hat\ybm, \mubold;\thetabold,\tau) \mapsto
 \pder{\hat{\ell}_\Phibold}{\mubold}(\hat\ybm,\mubold;\thetabold,\tau) - \hat\zbm^T\pder{\hat\rbm_\Phibold}{\mubold}(\hat\ybm,\mubold) =
 \gbm^\lambda(\Phibold\hat\zbm, \Phibold\hat\ybm,\mubold;\thetabold,\tau).
\end{equation}
The Lagrangian part of the gradient reconstruction operator, $\hat\gbm^{\mathtt{L},\lambda} : \Rbb^n\times\Rbb^n\times \Dcal \times \Rbb^{N_\cbm} \rightarrow \Rbb^{1\times N_\mubold}$, is
\begin{equation}\label{eqn:rom_grad_lag}
\hat\gbm_\Phibold^{\mathtt{L},\lambda} : (\hat\zbm, \hat\ybm, \mubold;\thetabold) \mapsto \pder{j}{\mubold}(\Phibold\ybm,\mubold) - \thetabold^T\pder{\cbm}{\mubold}(\Phibold\ybm,\mubold) - \hat\zbm^T\pder{\hat\rbm_\Phibold}{\mubold}(\ybm,\mubold),
\end{equation}
which allows the gradient operator to be rewritten as
\begin{equation}\label{eqn:rom_grad_regroup}
    \hat\gbm_\Phibold^\lambda(\zbm,\ybm, \mubold;\thetabold,\tau) = \hat\gbm_\Phibold^{\mathtt{L},\lambda}(\hat\zbm, \ybm, \mubold;\thetabold) + \tau\cbm(\Phibold\ybm,\mubold)^T\pder{\cbm}{\mubold}(\Phibold\ybm,\mubold).
\end{equation}
The reduced gradient operator also inherits an unassembled structure from the HDM in (\ref{eqn:hdm_elem_grad})
\begin{equation} \label{eqn:rom_elem_grad}
\begin{aligned}
\hat\gbm_\Phibold^\lambda(\hat\zbm, \hat\ybm, \mubold;\thetabold,\tau) 
&= \sum^{N_\mathtt{e}}_{e=1} \hat\gbm_{\Phibold,e}^{\mathtt{L},\lambda}(\hat\zbm, \hat\ybm, \mubold;\thetabold)+\tau\bracket{\sum^{N_\mathtt{e}}_{e=1}\cbm_e(\Phibold_e\hat\ybm,\mubold)}^T\bracket{\sum^{N_\mathtt{e}}_{e=1}\sens{\cbm_e}(\Phibold_e\hat\ybm,\mubold)},
\end{aligned}
\end{equation}
where the elemental contribution, $\hat\gbm_{\Phibold,e}^{\mathtt{L},\lambda} : \Rbb^n \times \Rbb^n \times \Dcal \times \Rbb^{N_\cbm} \mapsto \Rbb^{1\times N_\mubold}$, is defined as
\begin{equation}
\hat\gbm_{\Phibold,e}^{\mathtt{L},\lambda}(\hat\zbm, \hat\ybm, \mubold;\thetabold) 
= \sens{j_e}(\Phibold_e\hat\ybm,\mubold)-\thetabold^T\sens{\cbm_e}(\Phibold_e\hat\ybm,\mubold)-\hat\zbm^T\Phibold_e^T\sens{\rbm_e}(\Phibold_e\hat\ybm, \Phibold_e\hat\ybm', \mubold).
\end{equation}

Finally, we consider the reduced sensitivity residual as in \cite{wen_globally_2023}, $\hat\rbm^\partial_\Phibold : \Rbb^{n\times N_\mubold}\times\Rbb^n\times\Dcal \rightarrow \Rbb^{n\times N_\mubold}$, which will be used to accelerate convergence of the optimization framework. The sensitivity residual is defined as
\begin{equation} \label{eqn:rom_sens}
 \hat\rbm^\partial_\Phibold : (\hat\wbm,\hat\ybm,\mubold) \mapsto \pder{\hat\rbm_\Phibold}{\hat\ybm}(\hat\ybm,\mubold)\hat\wbm + \pder{\hat\rbm_\Phibold}{\mubold}(\hat\ybm,\mubold),
\end{equation}
and its unassembled form
\begin{equation} \label{eqn:rom_sens_elem}
  \hat\rbm^\partial_\Phibold(\hat\wbm,\hat\ybm,\mubold) =
   \sum^{N_\mathtt{e}}_{e=1} \bracket{
       \Phibold_e^T\pder{\rbm_e}{\ubm_e}(\Phibold_e\hat\ybm, \Phibold_e'\hat\ybm, \mubold) \Phibold_e\hat\wbm +
       \Phibold_e^T\pder{\rbm_e}{\ubm_e'}(\Phibold_e\hat\ybm, \Phibold_e'\hat\ybm, \mubold) \Phibold_e'\hat\wbm +
       \Phibold_e^T\pder{\rbm_e}{\mubold}(\Phibold_e\hat\ybm, \Phibold_e'\hat\ybm, \mubold)
     }.
\end{equation}
The reduced sensitivity problem is: given $\mubold \in \Dcal$ and the corresponding reduced primal solution $\hat\ybm_\Phibold^\star$, find the reduced sensitivity solution $\partial_\mubold\hat\ybm_\Phibold^\star \in \Rbb^{n \times N_\mubold}$ such that
\begin{equation}
  \hat\rbm^\partial_\Phibold(\partial_\mubold\hat\ybm_\Phibold^\star,\hat\ybm_\Phibold^\star,\mubold) = \zerobold.
\end{equation}

Despite the potentially significant reduction in degrees of freedom between the HDM system (\ref{eqn:hdm_res}) and the
reduced system (\ref{eqn:rom_res_eqn}), computational efficiency will not necessarily be achieved due to the cost of
constructing the nonlinear terms. This can be clearly seen from the unassembled form of the primal residual
(\ref{eqn:rom_res_elem}), objective function (\ref{eqn:rom_elem_f}), adjoint residual (\ref{eqn:rom_elem_adj_res}), and AL gradient (\ref{eqn:rom_elem_grad}) as each of these operations requires elemental operations \textit{for all elements}. For systems comprised of many elements, e.g., high-fidelity discretization of PDEs, which can be a serious bottleneck. To accelerate the formation of the nonlinear terms, we turn to EQP-based hyperreduction.

\subsection{An optimization-aware empirical quadrature procedure}
\label{sec:hyperreduction:lp_eqp}
To accelerate the assembly of the nonlinear terms in the various reduced quantities introduced in
Section~\ref{sec:hyperreduction:rom}, we use the empirical quadrature procedure \cite{wen_globally_2023, yano_lp_2019, yano_discontinuous_2019}. In the adjoint-based
optimization setting, we use EQP to accelerate any operation that involves assembly over all elements (Section~\ref{sec:hyperreduction:eqp_form}),
i.e., evaluation of the primal and adjoint residuals, the optimization functionals, and gradient reconstruction.
To ensure all hyperreduced quantities are accurate with respect to their reduced counterparts, we include additional
constraints on the original EQP linear program introduced in \cite{yano_discontinuous_2019} (Section~\ref{sec:hyperreduction:eqp_training}).

\subsubsection{Formulation}
\label{sec:hyperreduction:eqp_form}
The EQP construction replaces the unassembled form of the reduced primal residual with a weighted (hyperreduced) version,
$\tilde\rbm_\Phibold : \Rbb^n \times \Dcal \times \Rcal \rightarrow \Rbb^n$, where
\begin{equation}\label{eqn:eqp_elem_res}
 \tilde\rbm_\Phibold : (\tilde\ybm, \mubold; \rhobold) \mapsto \sum_{e=1}^{N_\mathtt{e}} \rho_e\Phibold_e^T \rbm_e(\Phibold_e\tilde\ybm, \Phibold_e'\tilde\ybm,\mubold)
\end{equation}
and $\rhobold\in\Rcal$ is the vector of weights with $\Rcal\subset\Rbb^{N_\mathtt{e}}$ the set of admissible weights
such that $\onebold\in\Rcal$ ($\onebold$ is the vector with each entry equal to one). For each element $\Omega_e\in\Ecal_h$ with $\rho_e=0$, the operations on the element can be completely skipped so computational efficiency is achieved when the vector of weights is highly sparse; construction of $\rhobold$ is deferred to Section~\ref{sec:hyperreduction:eqp_training}. The primal EQP problem reads: given $\mubold\in\Dcal$ and $\rhobold\in\Rcal$, find $\tilde\ybm_\Phibold^\star\in\Rbb^n$ such that
\begin{equation}\label{eqn:eqp:res_eqn}
 \tilde\rbm_\Phibold(\tilde\ybm_\Phibold^\star,\mubold;\rhobold) = \zerobold.
\end{equation}
For each $\rhobold\in\Rcal$, we assume there is a unique primal solution
$\tilde\ybm_\Phibold^\star = \tilde\ybm_\Phibold^\star(\mubold;\rhobold)$ satisfying (\ref{eqn:eqp:res_eqn}) for every $\mubold\in\Dcal$.

In the optimization setting, computational efficiency is also required for the evaluation of the optimization functionals, the adjoint residual
(similar to \cite{wen_globally_2023,yano2020goal,du_adaptive_2021}), and the gradient reconstruction. We use the same approach of introducing weights
($\rhobold$) into the unassembled form so efficiency is achieved when $\rhobold$ is sparse.
The hyperreduced AL, $\tilde{\ell}_\Phibold : \Rbb^n \times \Dcal \times \Rcal \times \Rbb^{N_\cbm} \times \Rbb \rightarrow \Rbb$, its Lagrangian part $\tilde{\ell}^\mathtt{L}_\Phibold : \Rbb^n \times \Dcal \times \Rcal \times \Rbb^{N_\cbm} \rightarrow \Rbb$, and the restriction of the hyperreduced AL function to the EQP solution manifold, $\tilde{f}_\Phibold : \Dcal\times \Rcal \times \Rbb^{N_\cbm} \times \Rbb \rightarrow \Rbb$,
are defined as
\begin{equation}\label{eqn:eqp_elem_f}
\begin{aligned}
\tilde{\ell}_\Phibold : (\tilde\ybm, \mubold; \rhobold,\thetabold,\tau) &\mapsto
\tilde{\ell}^\mathtt{L}_\Phibold(\tilde\ybm, \mubold; \rhobold,\thetabold) + \frac{\tau}{2}\norm{\sum^{N_\mathtt{e}}_{e=1}\rho_e \cbm_e(\Phibold_e\tilde\ybm,\mubold)}^2_2, \\
\tilde{\ell}^\mathtt{L}_\Phibold:(\tilde\ybm, \mubold; \rhobold,\thetabold)
&\mapsto \sum_{e=1}^{N_\mathtt{e}} \rho_e j_e(\Phibold_e\tilde\ybm,\mubold) - \thetabold^T\sum^{N_\mathtt{e}}_{e=1} \rho_e \cbm_e(\Phibold_e\tilde\ybm,\mubold), \\
\tilde{f}_\Phibold : (\mubold; \rhobold, \thetabold,\tau)
&\mapsto \tilde{\ell}_\Phibold(\tilde\ybm_\Phibold^\star(\mubold;\rhobold),\mubold;\rhobold,\thetabold,\tau).
\end{aligned}
\end{equation}

The hyperreduced adjoint residual, $\trlam_\Phibold : \Rbb^n\times\Rbb^n\times\Dcal\times\Rcal \times \Rbb^{N_\cbm} \times \Rbb \rightarrow \Rbb^n$, is defined as
\begin{equation} \label{eqn:eqp:res_adj}
 \trlam_\Phibold : (\tilde\zbm,\tilde\ybm,\mubold;\rhobold,\thetabold,\tau) \mapsto \pder{\tilde\rbm_\Phibold}{\tilde\ybm}(\tilde\ybm,\mubold;\rhobold)^T\tilde\zbm - \pder{\tilde{\ell}_\Phibold}{\tilde\ybm}(\tilde\ybm,\mubold;\rhobold,\thetabold,\tau)^T,
\end{equation}
or, in unassembled form,
\begin{equation} \label{eqn:eqp_elem_rlam}
\trlam_\Phibold(\tilde\zbm,\tilde\ybm,\mubold;\rhobold,\thetabold,\tau) 
= \tilde\rbm_\Phibold^{\mathtt{L},\lambda}(\tilde\zbm,\tilde\ybm,\mubold;\rhobold,\thetabold)-\tau\bracket{\sum^{N_\mathtt{e}}_{e=1}\rho_e\Phibold_e^T \pder{\cbm_e}{\ubm_e}(\Phibold_e\tilde\ybm,\mubold)^T}\bracket{\sum^{N_\mathtt{e}}_{e=1}\rho_e\cbm_e(\Phibold_e\tilde\ybm,\mubold)}.
\end{equation}
where its Lagrangian part $\tilde\rbm_\Phibold^{\mathtt{L},\lambda}: \Rbb^n\times\Rbb^n\times\Dcal\times\Rcal \times \Rbb^{N_\cbm} \rightarrow \Rbb^n$ is defined as
\begin{equation}
\begin{aligned}
\tilde\rbm_\Phibold^{\mathtt{L},\lambda}:(\tilde\zbm,\tilde\ybm,\mubold;\rhobold,\thetabold)
& \mapsto 
\sum^{N_\mathtt{e}}_{e=1} \rho_e\left[
-\Phibold_e^T\pder{j_e}{\ubm_e}(\Phibold_e\tilde\ybm,\mubold)^T
+\Phibold_e^T\pder{\cbm_e}{\ubm_e}(\Phibold_e\tilde\ybm,\mubold)^T\thetabold
+\Phibold_e^T\pder{\rbm_e}{\ubm_e}(\Phibold_e\tilde\ybm, \Phibold_e'\tilde\ybm, \mubold)^T \Phibold_e \tilde\zbm \right. + \\
&\quad\left. \paren{\Phibold_e'}^T\pder{\rbm_e}{\ubm_e'}(\Phibold_e\tilde\ybm, \Phibold_e'\tilde\ybm, \mubold)^T \Phibold_e\tilde\zbm \right].
\end{aligned}
\end{equation}
The hyperreduced adjoint problem is: given $\mubold\in\Dcal$ and the corresponding hyperreduced primal solution
$\tilde\ybm_\Phibold^\star$ satisfying (\ref{eqn:eqp:res_eqn}), find the hyperreduced adjoint solution $\tilde\lambdabold_\Phibold^\star\in\Rbb^n$ satisfying
\begin{equation}
 \trlam_\Phibold(\tilde\lambdabold_\Phibold^\star, \tilde\ybm_\Phibold^\star, \mubold; \rhobold,\thetabold,\tau) = \zerobold.
\end{equation}
For each $\rhobold\in\Rcal$, we assume the hyperreduced Jacobian is well-defined and invertible for every $(\tilde\ybm_\Phibold^\star(\mubold;\rhobold),\mubold)$ pair with $\mubold\in\Dcal$ so there exists a unique hyperreduced adjoint solution
\begin{equation}
 \tilde\lambdabold_\Phibold^\star(\mubold;\rhobold,\thetabold,\tau) = \pder{\tilde\rbm_\Phibold}{\tilde\ybm}(\tilde\ybm_\Phibold^\star(\mubold;\rhobold),\mubold;\rhobold)^{-T}\pder{\tilde{\ell}_\Phibold}{\tilde\ybm}(\tilde\ybm_\Phibold^\star(\mubold;\rhobold),\mubold;\rhobold,\thetabold,\tau)^T,
\end{equation}
making the implicit map $(\mubold;\rhobold,\thetabold,\tau)\mapsto\tilde\lambdabold_\Phibold^\star(\mubold;\rhobold,\thetabold,\tau)$ well-defined.
From the primal-adjoint pair ($\tilde\ybm_\Phibold^\star,\tilde\lambdabold_\Phibold^\star$), the gradient of the AL function in reduced-space,
$\nabla \tilde{f}_\Phibold : \Dcal\times \Rcal \times \Rbb^{N_\cbm} \times \Rbb \rightarrow \Rbb^{N_\mubold}$, is computed as
\begin{equation}
\begin{aligned}
  \nabla \tilde{f}_\Phibold : (\mubold;\rhobold,\thetabold,\tau)
  &\mapsto \tilde\gbm_\Phibold^\lambda(\tilde\lambdabold_\Phibold^\star(\mubold;\rhobold,\thetabold,\tau), \tilde\ybm_\Phibold^\star(\mubold;\rhobold),\mubold; \rhobold,\thetabold,\tau)^T \\
  &=\pder{\tilde{\ell}_\Phibold}{\mubold}(\tilde\ybm_\Phibold^\star(\mubold;\rhobold),\mubold;\rhobold,\thetabold,\tau)^T - \pder{\tilde\rbm_\Phibold}{\mubold}(\tilde\ybm_\Phibold^\star(\mubold;\rhobold),\mubold;\rhobold)^T\tilde\lambdabold_\Phibold^\star(\mubold;\rhobold,\thetabold,\tau),  
\end{aligned}
\end{equation}
where the operator that reconstructs the hyperreduced AL gradient from the hyperreduced adjoint solution,
$\tilde\gbm_\Phibold^\lambda : \Rbb^n\times\Rbb^n\times\Dcal\times\Rcal \times \Rbb^{N_\cbm} \times \Rbb \rightarrow \Rbb^{1\times N_\mubold}$, is
\begin{equation}
 \tilde\gbm_\Phibold^\lambda : (\tilde\zbm, \tilde\ybm,\mubold;\rhobold,\thetabold,\tau) \mapsto
 \pder{\tilde{\ell}_\Phibold}{\mubold}(\tilde\ybm,\mubold;\rhobold,\thetabold,\tau) - \tilde\zbm^T\pder{\tilde\rbm_\Phibold}{\mubold}(\tilde\ybm,\mubold; \rhobold),
\end{equation}
or, in unassembled form,
\begin{equation}\label{eqn:eqp_elem_glam}
\tilde\gbm_\Phibold^\lambda(\tilde\zbm, \tilde\ybm, \mubold; \rhobold,\thetabold,\tau) 
= \tilde\gbm_\Phibold^{\mathtt{L},\lambda}(\tilde\zbm, \tilde\ybm, \mubold; \rhobold,\thetabold) + \tau\bracket{\sum^{N_\mathtt{e}}_{e=1}\rho_e\cbm_e(\Phibold_e\tilde\ybm,\mubold)}^T\bracket{\sum^{N_\mathtt{e}}_{e=1}\rho_e\sens{\cbm_e}(\Phibold_e\tilde\ybm,\mubold)},
\end{equation}
where $\tilde\gbm_\Phibold^{\mathtt{L},\lambda}: \Rbb^n\times\Rbb^n\times\Dcal\times\Rcal \times \Rbb^{N_\cbm}\rightarrow \Rbb^{1\times N_\mubold}$ is defined as
\begin{equation}
\tilde\gbm_\Phibold^{\mathtt{L},\lambda}:(\tilde\zbm, \tilde\ybm, \mubold; \rhobold,\thetabold)
\mapsto \sum^{N_\mathtt{e}}_{e=1} \rho_e\bracket{
\sens{j_e}(\Phibold_e\tilde\ybm,\mubold)-\thetabold^T\sens{\cbm_e}(\Phibold_e\tilde\ybm,\mubold)-\tilde\zbm^T\Phibold_e^T\sens{\rbm_e}(\Phibold_e\tilde\ybm, \Phibold_e\tilde\ybm', \mubold)}.
\end{equation}

Finally, the hyperreduced sensitivity residual, $\tilde\rbm^\partial_\Phibold : \Rbb^{n\times N_\mubold}\times\Rbb^n\times\Dcal\times\Rcal \rightarrow \Rbb^{n\times N_\mubold}$, is defined as
\begin{equation} \label{eqn:eqp:res_sens}
 \tilde\rbm^\partial_\Phibold : (\tilde\wbm,\tilde\ybm,\mubold;\rhobold) \mapsto \pder{\tilde\rbm_\Phibold}{\tilde\ybm}(\tilde\ybm,\mubold;\rhobold)\tilde\wbm + \pder{\tilde\rbm_\Phibold}{\mubold}(\tilde\ybm,\mubold;\rhobold),
\end{equation}
or, in unassembled form,
\begin{equation} \label{eqn:eqp:res_sens_elem}
  \tilde\rbm^\partial_\Phibold(\tilde\wbm,\tilde\ybm,\mubold;\rhobold) =
  \sum^{N_\mathtt{e}}_{e=1} \rho_e \bracket{
       \Phibold_e^T\pder{\rbm_e}{\ubm_e}(\Phibold_e\tilde\ybm, \Phibold_e'\tilde\ybm, \mubold) \Phibold_e\tilde\wbm +
       \Phibold_e^T\pder{\rbm_e}{\ubm_e'}(\Phibold_e\tilde\ybm, \Phibold_e'\tilde\ybm, \mubold) \Phibold_e'\tilde\wbm +
       \Phibold_e^T\pder{\rbm_e}{\mubold}(\Phibold_e\tilde\ybm, \Phibold_e'\tilde\ybm, \mubold)
     }.
\end{equation}
The hyperreduced sensitivity problem is: given $\mubold \in \Dcal$ and the corresponding reduced primal solution $\tilde\ybm_\Phibold^\star$ find the hyperreduced sensitivity solution $\partial_\mubold\tilde\ybm_\Phibold^\star \in \Rbb^{n \times N_\mubold}$ such that
\begin{equation}
\tilde\rbm^\partial_\Phibold(\partial_\mubold\tilde\ybm_\Phibold^\star,\tilde\ybm_\Phibold^\star,\mubold;\rhobold) = \zerobold
\end{equation}
\begin{remark}\label{rem:rho_one}
It can be observed that the hyperreduced form of all quantities, e.g., (\ref{eqn:eqp_elem_f}), (\ref{eqn:eqp:res_adj}), (\ref{eqn:eqp_elem_glam}), and (\ref{eqn:eqp:res_sens}),
agree with the corresponding reduced quantity in the case where $\rhobold = \onebold$.
\end{remark}
\begin{remark}
Following \cite{wen_globally_2023}, we use a single reduced basis to approximate
the HDM state, sensitivity, and adjoint, as opposed to a separate basis for each.
Additionally, we use a single set of EQP weights for all terms (e.g.,
the primal residual, adjoint residual, sensitivity residual, optimization functionals,
and gradient reconstruction).  This allows us to directly derive gradients at the reduced
and hyperreduced levels, which guarantees consistency with the true gradient of the
(hyper)reduced model, which is helpful for convergence analysis and solving
trust-region subproblems in the next section.
\end{remark}

\subsubsection{Training}
\label{sec:hyperreduction:eqp_training}
The success of EQP is inherently linked to the construction of a sparse weight vector that ensures the hyperreduced
quantities introduced in Section~\ref{sec:hyperreduction:eqp_form} accurately approximate the corresponding reduced quantity in Section~\ref{sec:hyperreduction:rom}.
In \cite{yano_lp_2019,yano_discontinuous_2019}, the EQP weights are chosen to be the solution of an $\ell_1$ minimization problem (to promote sparsity) that
includes several constraints, most important of which is the \textit{manifold accuracy} constraint that requires the reduced ($\hat\rbm$) and hyperreduced ($\tilde\rbm$) primal residuals are sufficiently close on some training set.
Manifold constraints on the quantities of interest and adjoint residual, among others, are included when EQP is used to accelerate dual-weighted residual error estimation \cite{yano2020goal,du_adaptive_2021}. In the adjoint-based optimization setting, we require the hyperreduced primal residual, adjoint residual, AL function, and gradient reconstruction operator accurately approximate the corresponding reduced quantity \cite{wen_globally_2023}. However, the presence of the quadratic penalty would lead to a nonlinear optimization problem (as opposed to a linear program) if we directly control the error in the AL function. Instead, we control the error in the Lagrangian and constraint functions separately to yield a linear program.

To this end, we let $\Phibold$ be a given reduced basis and $\Xibold\subset \Dcal$ be a collection
of EQP training parameters, and define the EQP weights for residual terms, $\rhobold^\star$,
as the solution of the following linear program
\begin{equation}\label{eqn:linprog}
 \rhobold^\star = \argoptunc{\rhobold\in\Ccal^\mathtt{nn}\cap\Ccal_{\Phibold,\Xibold,\deltabold}}{\sum_{e=1}^{N_\mathtt{e}} \rho_e},
\end{equation}
where $\Ccal^\mathtt{nn}$ is the set of nonnegative weights
\begin{equation}
 \Ccal^\mathtt{nn} \coloneqq \left\{ \rhobold\in\Rbb^{N_\mathtt{e}} \suchthat \rho_e \geq 0, ~ e = 1,\dots,N_\mathtt{e} \right\},
\end{equation}
$ \deltabold = (\delta_\mathtt{dv},\delta_\mathtt{rp},\delta_\mathtt{lra},\delta_\mathtt{lga},\delta_\mathtt{c}, \delta_\mathtt{dc\mu}, \delta_\mathtt{dcy}, \delta_\mathtt{rs}, \delta_\mathtt{lq}) \in\Rbb^9$ is a collection of tolerances, and $\Ccal_{\Phibold,\Xibold,\deltabold_\mathtt{r}}\subset\Rbb^{N_\mathtt{e}}$ is the intersection of some subset
of the following accuracy constraints\footnote{Superscript legend: 
$\mathtt{nn}$ = \underline{n}on\underline{n}egativity, 
$\mathtt{dv}$ = \underline{d}omain \underline{v}olume, 
$\mathtt{rp}$ = \underline{p}rimal \underline{r}esidual, 
$\mathtt{lra}$ = \underline{L}agrangian \underline{a}djoint \underline{r}esidual, 
$\mathtt{lga}$ = \underline{L}agrangian \underline{a}djoint \underline{g}radient reconstruction,
$\mathtt{c}$ = \underline{c}onstraints, 
$\mathtt{dc\mu}$ = \underline{d}erivatives of the \underline{c}onstraints with respect to $\underline{\mubold}$,
$\mathtt{dcy}$ = \underline{d}erivatives of the \underline{c}onstraints with respect to $\underline{\ybm}$,
$\mathtt{lq}$ = \underline{L}agrangian \underline{q}uantity of interest, 
$\mathtt{rs}$ = \underline{s}ensitivity \underline{r}esidual.}, i.e.,
\begin{equation}
  \Ccal_{\Phibold,\Xibold,\deltabold} \coloneqq
  \Ccal_{\delta_\mathtt{dv}}^\mathtt{dv} \cap
  \Ccal_{\Phibold,\Xibold,\delta_\mathtt{rp}}^\mathtt{rp} \cap
  \Ccal_{\Phibold,\Xibold,\delta_\mathtt{lra}}^\mathtt{lra} \cap
  \Ccal_{\Phibold,\Xibold,\delta_\mathtt{lga}}^\mathtt{lga} \cap
  \Ccal_{\Phibold,\Xibold,\delta_\mathtt{c}}^\mathtt{c} \cap
  \Ccal_{\Phibold,\Xibold,\delta_\mathtt{dc\mu}}^\mathtt{dc\mu} \cap
  \Ccal_{\Phibold,\Xibold,\delta_\mathtt{dcy}}^\mathtt{dcy} \cap
  \Ccal_{\Phibold,\Xibold,\delta_\mathtt{rs}}^\mathtt{rs} \cap
  \Ccal_{\Phibold,\Xibold,\delta_\mathtt{lq}}^\mathtt{lq}, 
\end{equation}
where
\begin{equation}\label{eqn:eqp:rescon}
\begin{aligned}
 \Ccal_{\delta}^\mathtt{dv} &\coloneqq \left\{ \rhobold\in\Rbb^{N_\mathtt{e}}\suchthat \left| |\Omega| - \sum_{e=1}^{N_\mathtt{e}} \rho_e|\Omega_e|\right| \leq \delta\right\}, \\
\Ccal_{\Phibold,\Xibold,\delta}^\mathtt{rp} &\coloneqq \left\{ \rhobold\in\Rbb^{N_\mathtt{e}}\suchthat  \norm{\hat\rbm_\Phibold(\hat\ybm_\Phibold^\star(\mubold),\mubold)-\tilde\rbm_\Phibold(\hat\ybm_\Phibold^\star(\mubold),\mubold;\rhobold)}_\infty \leq \delta, \forall \mubold\in\Xibold\right\}, \\
\Ccal_{\Phibold,\Xibold,\delta}^\mathtt{lra} &\coloneqq \left\{ \rhobold\in\Rbb^{N_\mathtt{e}} \suchthat  \norm{\hat\rbm_\Phibold^{\mathtt{L},\lambda}(\hat\lambdabold_\Phibold^\star(\mubold;\thetabold,\tau),\hat\ybm_\Phibold^\star(\mubold),\mubold;\thetabold)-\tilde\rbm_\Phibold^{\mathtt{L},\lambda}(\hat\lambdabold_\Phibold^\star(\mubold;\thetabold,\tau),\hat\ybm_\Phibold^\star(\mubold),\mubold;\rhobold,\thetabold)}_\infty \leq \delta, \forall \mubold\in\Xibold\right\}, \\
\Ccal_{\Phibold,\Xibold,\delta}^\mathtt{lga} &\coloneqq \left\{ \rhobold\in\Rbb^{N_\mathtt{e}}\suchthat  \norm{\hat\gbm_\Phibold^{\mathtt{L},\lambda}(\hat\lambdabold_\Phibold^\star(\mubold;\thetabold,\tau),\hat\ybm_\Phibold^\star(\mubold),\mubold;\thetabold)-\tilde\gbm_\Phibold^{\mathtt{L},\lambda}(\hat\lambdabold_\Phibold^\star(\mubold;\thetabold,\tau),\hat\ybm_\Phibold^\star(\mubold),\mubold;\rhobold,\thetabold)}_\infty \leq \delta, \forall \mubold\in\Xibold\right\},\\
\Ccal_{\Phibold,\Xibold,\delta}^\mathtt{c} &\coloneqq \left\{ \rhobold\in\Rbb^{N_\mathtt{e}}\suchthat  \norm{\hat{\cbm}_\Phibold(\hat\ybm_\Phibold^\star(\mubold),\mubold)-\tilde{\cbm}_\Phibold(\hat\ybm_\Phibold^\star(\mubold),\mubold;\rhobold)}_\infty \leq \frac{\delta}{\tau}, \forall \mubold\in\Xibold \right\}, \\
\Ccal_{\Phibold,\Xibold,\delta}^\mathtt{dc\mu} &\coloneqq \left\{ \rhobold\in\Rbb^{N_\mathtt{e}}\suchthat  
\norm{
  \partial_\mubold \hat \cbm_\Phibold(\hat\ybm_\Phibold^\star(\mubold),\mubold)- \partial_\mubold \tilde \cbm_\Phibold(\hat\ybm_\Phibold^\star(\mubold),\mubold;\rhobold)
}_\infty \leq \frac{\delta}{\tau}, \forall \mubold\in\Xibold\right\}, \\
\Ccal_{\Phibold,\Xibold,\delta}^\mathtt{dcy} &\coloneqq \left\{ \rhobold\in\Rbb^{N_\mathtt{e}}\suchthat 
\norm{
  \partial_\ybm \hat \cbm_\Phibold(\hat\ybm_\Phibold^\star(\mubold),\mubold)-
  \partial_\ybm \tilde \cbm_\Phibold(\hat\ybm_\Phibold^\star(\mubold),\mubold;\rhobold)
  }_\infty \leq \frac{\delta}{\tau}, \forall \mubold\in\Xibold\right\},\\
\Ccal_{\Phibold,\Xibold,\delta}^\mathtt{rs} &\coloneqq \left\{ \rhobold\in\Rbb^{N_\mathtt{e}} \suchthat  \norm{\hat\rbm_\Phibold^\partial(\partial_\mubold\hat\ybm_\Phibold^\star(\mubold),\hat\ybm_\Phibold^\star(\mubold),\mubold)-\tilde\rbm_\Phibold^\partial(\partial_\mubold\hat\ybm_\Phibold^\star(\mubold),\hat\ybm_\Phibold^\star(\mubold),\mubold;\rhobold)}_\infty \leq \delta, \forall \mubold\in\Xibold\right\},\\
\Ccal_{\Phibold,\Xibold,\delta}^\mathtt{lq} &\coloneqq \left\{ \rhobold\in\Rbb^{N_\mathtt{e}}\suchthat  \abs{\hat{\ell}^\mathtt{L}_\Phibold(\hat\ybm_\Phibold^\star(\mubold),\mubold;\thetabold)-\tilde{\ell}^\mathtt{L}_\Phibold(\hat\ybm_\Phibold^\star(\mubold),\mubold;\rhobold,\thetabold)} \leq \delta, \forall \mubold\in\Xibold\right\},
\end{aligned}
\end{equation}
where $|\Scal|$ is the volume of the set $\Scal$, $\partial_\ybm \hat \cbm_\Phibold = \pder{\hat\cbm_\Phibold}{\ybm}$ and $\partial_\mubold \hat \cbm_\Phibold = \pder{\hat\cbm_\Phibold}{\mubold}$.

Each constraint in (\ref{eqn:eqp:rescon}) ensures a selected reduced quantity introduced in  Section~\ref{sec:hyperreduction:rom} is sufficiently approximated by the corresponding hyperreduced quantity in Section~\ref{sec:hyperreduction:eqp_form}. Finally, we define $\rhobold^\star : (\Phibold,\Xibold,\deltabold_\mathtt{r}) \mapsto \rhobold^\star(\Phibold,\Xibold,\deltabold_\mathtt{r})$
as the implicit map from a given reduced basis $\Phibold$, EQP training set $\Xibold$, and tolerances $\deltabold$ to the solution of the linear program in (\ref{eqn:linprog}). 
For now, we leave the tolerances and EQP training set unspecified; in Section~\ref{sec:eqp_trammo}, these will be chosen to guarantee global convergence of the trust-region method to a local minimum of the unreduced optimization problem in (\ref{eqn:hdm_full_space_auglag}).
\begin{remark} \label{rem:eqp_props}
The optimization problem in (\ref{eqn:linprog}) is guaranteed to have a feasible solution (regardless of which subset of the constraints is used)
because all hyperreduced quantities are equivalent to the corresponding reduced quantity when $\rhobold=\onebold$. 
\end{remark}
\begin{remark}
The constraints on the Lagrangian part of the augmented Lagrangian function $\Ccal_{\Phibold,\Xibold,\delta}^\mathrm{lq}$ and sensitivity residual $\Ccal_{\Phibold,\Xibold,\delta}^\mathrm{rs}$ are not required for global convergence of the trust-region method in Section~\ref{sec:eqp_trammo}; however, similar constraints were shown to accelerate convergence in the unconstrained setting \cite{wen_globally_2023}.
\end{remark}

\subsection{Error estimation}
\label{sec:hyperreduction:eqp_err_est}
We now introduce residual-based error estimates for the hyperreduced quantity of interest and its gradient. We begin with
a residual-based error bound from \cite{wen_globally_2023} that applies under regularity assumptions (Assumptions~\ref{assum:hdm}-\ref{assum:eqp} in~\ref{sec:appendix_A}) independent
of the training procedure for the reduced basis $\Phibold$ and weight vector $\rhobold$ for fixed Lagrange multiplier estimate $\thetabold$ and penalty parameter $\tau$.
\begin{theorem}\label{the:qoi_grad_errbnd}
Under Assumptions~\ref{assum:hdm}-\ref{assum:eqp}, there exist constants $c_1,c_2>0$ such that for any $\mubold\in\Dcal$, $\rhobold\in\Rcal$, $\thetabold\in\Rbb^{N_\cbm}$, and $\tau > 0$
\begin{equation}\label{eqn:eqp:qoi_errbnd}
\begin{aligned}
 \abs{f(\mubold;\thetabold,\tau) - \tilde{f}_\Phibold(\mubold;\rhobold,\thetabold,\tau)} 
 &\leq c_1\norm{\rbm(\Phibold\hat\ybm_\Phibold^\star(\mubold),\mubold)} + c_2\norm{\tilde\rbm_\Phibold(\hat\ybm_\Phibold^\star(\mubold),\mubold;\rhobold)} + \\
 &\quad \abs{\hat\ell_\Phibold(\hat\ybm_\Phibold^\star(\mubold),\mubold;\thetabold,\tau)-\tilde\ell_\Phibold(\hat\ybm_\Phibold^\star(\mubold),\mubold;\rhobold,\thetabold,\tau)}. 
\end{aligned}
\end{equation}
Furthermore, there exist constants $c_1',c_2',c_3',c_4'>0$ such that
\begin{equation}\label{eqn:eqp:qoi_grad_errbnd}
\begin{aligned}
 \norm{\nabla f(\mubold;\thetabold,\tau) - \nabla\tilde{f}_\Phibold(\mubold;\rhobold,\thetabold,\tau)} \leq
 &c_1'\norm{\rbm(\Phibold\hat\ybm_\Phibold^\star(\mubold),\mubold)} +
 c_2'\norm{\rbm^\lambda(\Phibold \hat\lambdabold_\Phibold^\star(\mubold;\thetabold,\tau),\Phibold\hat\ybm_\Phibold^\star(\mubold),\mubold;\thetabold,\tau)} + \\
 & c_3'\norm{\tilde\rbm_\Phibold(\hat\ybm_\Phibold^\star(\mubold),\mubold;\rhobold)} + 
 c_4'\norm{\tilde\rbm^\lambda_\Phibold(\hat\lambdabold_\Phibold^\star(\mubold;\thetabold,\tau),\hat\ybm_\Phibold^\star(\mubold),\mubold;\rhobold,\thetabold,\tau)} + \\
 & \norm{\hat\gbm_\Phibold^\lambda(\hat\lambdabold_\Phibold^\star(\mubold;\thetabold,\tau),\hat\ybm_\Phibold^\star(\mubold),\mubold;\thetabold,\tau) - \tilde\gbm_\Phibold^\lambda(\hat\lambdabold_\Phibold^\star(\mubold;\thetabold,\tau),\hat\ybm_\Phibold^\star(\mubold),\mubold;\rhobold,\thetabold,\tau)}.
\end{aligned}
\end{equation}
\begin{proof}
Follows from Theorem 1 in \cite{wen_globally_2023}.
\end{proof}
\end{theorem}


Now we specialize the results in Theorem~\ref{the:qoi_grad_errbnd} with requirements on the training procedure for $\Phibold$ and $\rhobold$.
In particular, if the HDM primal and adjoint solution at a given $\mubold\in\Dcal$ is included in the column space of the
reduced basis, then the first term in (\ref{eqn:eqp:qoi_errbnd}) and the first two terms in (\ref{eqn:eqp:qoi_grad_errbnd}) are zero. All remaining terms are the difference
between the reduced and hyperreduced quantity evaluated at the reduced state $\hat\ybm_\Phibold^\star$, which are
exactly controlled by the EQP constraints in Section~\ref{sec:hyperreduction:eqp_training}, which leads to the following corollaries.

\begin{corollary}\label{cor:qoi_errbnd}
Suppose Assumptions~\ref{assum:hdm}-\ref{assum:eqp} hold with $\Rcal\supset\Ccal_{\Phibold,\Xibold,\delta_\mathtt{rp}}^\mathtt{rp} \cap \Ccal_{\Phibold,\Xibold,\delta_\mathtt{lq}}^\mathtt{lq} \cap
\Ccal_{\Phibold,\Xibold,\delta_\mathtt{c}}^\mathtt{c}$ and consider any $\mubold\in\Dcal$. Then, if the reduced basis
$\Phibold\in\Rbb^{N_\ubm\times n}$ satisfies
\begin{equation}\label{eqn:prim_dual_basis}
 \ubm^\star(\mubold) \in \mathrm{Ran}~\Phibold, \qquad
 \lambdabold^\star(\mubold) \in \mathrm{Ran}~\Phibold,
\end{equation}
where $\mathrm{Ran}(\Phibold)$ indicates the column space of $\Phibold$, and the weight vector $\rhobold$ is the solution of (\ref{eqn:linprog}) with constraint set
$\Ccal_{\Phibold,\Xibold,\deltabold} \subseteq \Ccal_{\Phibold,\Xibold,\delta_\mathtt{rp}}^\mathtt{rp} \cap \Ccal_{\Phibold,\Xibold,\delta_\mathtt{lq}}^\mathtt{lq} \cap
\Ccal_{\Phibold,\Xibold,\delta_\mathtt{c}}^\mathtt{c}$ and EQP training set $\Xibold\subset\Dcal$ with $\mubold\in\Xibold$, there
exist constants $c_2,c_3 > 0$ (independent of $\mubold$) such that
\begin{equation} \label{eqn:eqp:qoi_errbnd2}
 \abs{f(\mubold;\thetabold,\tau) - \tilde{f}_\Phibold(\mubold;\rhobold,\thetabold,\tau)} \leq c_2\delta_\mathtt{rp} + \delta_\mathtt{lq} + \tau c_3\delta_\mathtt{c}.
\end{equation}
\begin{proof}
See~\ref{sec:appendix_B}.
\end{proof}
\end{corollary}
 
\begin{corollary}\label{cor:qoi_grad_errbnd}
Suppose Assumptions~\ref{assum:hdm}-\ref{assum:eqp} hold with $\Rcal\supset\Ccal_{\Phibold,\Xibold,\delta_\mathtt{rp}}^\mathtt{rp} \cap   \Ccal_{\Phibold,\Xibold,\delta_\mathtt{lra}}^\mathtt{lra} \cap
\Ccal_{\Phibold,\Xibold,\delta_\mathtt{lga}}^\mathtt{lga} \cap
\Ccal_{\Phibold,\Xibold,\delta_\mathtt{c}}^\mathtt{c} \cap
\Ccal_{\Phibold,\Xibold,\delta_\mathtt{dc\mu}}^\mathtt{dc\mu} \cap
\Ccal_{\Phibold,\Xibold,\delta_\mathtt{dcy}}^\mathtt{dcy}$ and consider any $\mubold\in\Dcal$.
Then, if the reduced basis
$\Phibold\in\Rbb^{N_\ubm\times n}$ satisfies (\ref{eqn:prim_dual_basis}) and the  weight vector is the solution of (\ref{eqn:linprog}) with constraint set
$\Ccal_{\Phibold,\Xibold,\deltabold} \subseteq \Ccal_{\Phibold,\Xibold,\delta_\mathtt{rp}}^\mathtt{rp} \cap \Ccal_{\Phibold,\Xibold,\delta_\mathtt{lra}}^\mathtt{lra} \cap
\Ccal_{\Phibold,\Xibold,\delta_\mathtt{lga}}^\mathtt{lga} \cap
\Ccal_{\Phibold,\Xibold,\delta_\mathtt{c}}^\mathtt{c} \cap
\Ccal_{\Phibold,\Xibold,\delta_\mathtt{dc\mu}}^\mathtt{dc\mu} \cap
\Ccal_{\Phibold,\Xibold,\delta_\mathtt{dcy}}^\mathtt{dcy}$ and EQP training set $\Xibold\subset\Dcal$ with $\mubold\in\Xibold$, there exist constants
$c_3',c_4',\dots,c_7'>0$ (independent of $\mubold$) such that
\begin{equation} \label{eqn:eqp:qoi_grad_errbnd2}
\norm{\nabla f(\mubold;\thetabold,\tau) - \nabla\tilde{f}_\Phibold(\mubold;\rhobold,\thetabold,\tau)} \leq
c_3'\delta_\mathtt{rp}+c_4'\delta_\mathtt{lra}+\delta_\mathtt{lga}+c_5'\tau\delta_\mathtt{c}+c_6'\tau\delta_\mathtt{dcy}+c_7'\tau\delta_\mathtt{dc\mu}.
\end{equation}
\begin{proof}
See~\ref{sec:appendix_B}.
\end{proof}
\end{corollary}

\begin{corollary} \label{cor:both_errbnd}
Suppose Assumptions~\ref{assum:hdm}-\ref{assum:eqp} hold with $\Rcal\supset\Ccal_{\Phibold,\Xibold,\delta_\mathtt{rp}}^\mathtt{rp} \cap \Ccal_{\Phibold,\Xibold,\delta_\mathtt{lra}}^\mathtt{lra} \cap
\Ccal_{\Phibold,\Xibold,\delta_\mathtt{lga}}^\mathtt{lga} \cap
\Ccal_{\Phibold,\Xibold,\delta_\mathtt{c}}^\mathtt{c} \cap
\Ccal_{\Phibold,\Xibold,\delta_\mathtt{dc\mu}}^\mathtt{dc\mu} \cap
\Ccal_{\Phibold,\Xibold,\delta_\mathtt{dcy}}^\mathtt{dcy} \cap
\Ccal_{\Phibold,\Xibold,\delta_\mathtt{lq}}^\mathtt{lq}$
and consider any $\mubold\in\Dcal$. Then, if the reduced basis
$\Phibold\in\Rbb^{N_\ubm\times n}$ satisfies (\ref{eqn:prim_dual_basis}) and the weight vector is the solution of (\ref{eqn:linprog}) with constraint set
$\Ccal_{\Phibold,\Xibold,\deltabold} \subseteq \Ccal_{\Phibold,\Xibold,\delta_\mathtt{rp}}^\mathtt{rp} \cap \Ccal_{\Phibold,\Xibold,\delta_\mathtt{lra}}^\mathtt{lra} \cap
\Ccal_{\Phibold,\Xibold,\delta_\mathtt{lga}}^\mathtt{lga} \cap
\Ccal_{\Phibold,\Xibold,\delta_\mathtt{c}}^\mathtt{c} \cap
\Ccal_{\Phibold,\Xibold,\delta_\mathtt{dc\mu}}^\mathtt{dc\mu} \cap
\Ccal_{\Phibold,\Xibold,\delta_\mathtt{dcy}}^\mathtt{dcy} \cap
\Ccal_{\Phibold,\Xibold,\delta_\mathtt{lq}}^\mathtt{lq}$ and EQP training set $\Xibold\subset\Dcal$ with $\mubold\in\Xibold$, then both (\ref{eqn:eqp:qoi_errbnd2}) and (\ref{eqn:eqp:qoi_grad_errbnd2}) hold.
\begin{proof}
Follows directly from Corollaries~\ref{cor:qoi_errbnd}~and~\ref{cor:qoi_grad_errbnd}.
\end{proof}
\end{corollary}

\section{Augmented Lagrangian trust-region method based on hyperreduced models}
\label{sec:trammo}
In this section, we use the hyperreduced models introduced in Section~\ref{sec:hyperreduction} to accelerate the optimization problem of interest (\ref{eqn:auglag_reduced}) and embed it in an augmented Lagrangian framework to solve constrained optimization in (\ref{eqn:hdm_full_space}). To this end, we introduce a trust-region framework equipped with a new asymptotic error bound adapted from \cite{kouri_trust-region_2013} for bound-constrained problems (Section~\ref{sec:trammo:tr_inexact}), the classical augmented Lagrangian framework (Section~\ref{sec:trammo:auglag}), and the proposed approach to train and leverage hyperreduced models in the trust-region framework to ensure global convergence of each augmented Lagrangian subproblem(Section~\ref{sec:eqp_trammo}).

\subsection{An inexact trust-region method to solve bound-constrained problems}
\label{sec:trammo:tr_inexact}
We first consider a generic optimization problem with bound constraints,
\begin{equation}
\label{eqn:generic_bndopt}
\optunc{\mubold \in \Ccal}{F(\mubold)},
\end{equation}
where $F : \Ccal \rightarrow \Rbb$ satisfies Assumptions~\ref{assum:at:tr_inexact}. The problem in (\ref{eqn:generic_bndopt}) can be embedded in the augmented Lagrangian framework as a subproblem to solve the original constrained problem in (\ref{eqn:hdm_full_space}) (Section~\ref{sec:trammo:auglag}). To solve (\ref{eqn:generic_bndopt}), we seek a critical point $\mubold^\star \in \Ccal$ such that $\norm{\chi(\mubold^\star)} = 0$ \cite{conn_trust-region_2000}, where $\chi : \Dcal \mapsto \Rbb^{N_\mubold}$ is the projected gradient
\begin{equation}\label{eqn:chi_f}
\chi(\mubold) \coloneqq \Pbm_\Ccal(\mubold-\nabla F(\mubold), \mubold_l, \mubold_u) - \mubold,
\end{equation}
and $\Pbm_\Ccal(\xbm, \xbm_l, \xbm_u)$ is the Euclidean projection of the vector $\xbm$ onto rectangular box $[\xbm_l, \xbm_u]$, i.e.,
\begin{equation}\label{eqn:proj}
    \bracket{\Pbm_\Ccal(\xbm, \xbm_l, \xbm_u)}_i = 
    \left\{
        \begin{aligned}
            [\xbm_l]_i &\quad \mathtt{if} &\quad &  [\xbm]_i \leq [\xbm_l]_i, \\
            [\xbm]_i &\quad \mathtt{if} &\quad &[\xbm_l]_i < [\xbm]_i < [\xbm_u]_i, \\
            [\xbm_u]_i &\quad \mathtt{if} &\quad &[\xbm]_i \geq [\xbm_u]_i.      
        \end{aligned}
    \right.
\end{equation}

We introduce a trust-region method that constructs a sequence of trust-region centers $\{\mubold_k\}_{k=1}^\infty$ whose limit will be a local solution of (\ref{eqn:generic_bndopt}). At each trust-region center, a smooth approximation model $m_k : \Dcal \rightarrow \Rbb$ is built such that $m_k(\mubold) \approx F(\mubold)$ for all $\mubold\in\{\mubold\in\Ccal | \norm{\mubold-\mubold_k} \leq \Delta_k\}$,
where $\Delta_k>0$ is the trust-region radius. A candidate step $\check\mubold_k$ is
produced by approximately solving the generalized trust-region subproblem
\begin{equation}\label{eqn:tr:subprob}
 \optunc{\mubold \in \Ccal_{\Delta_k}}{m_k(\mubold)},
\end{equation}
where the set of trust-region constraints is defined by
\begin{equation}\label{eqn:tr_con}
\Ccal_{\Delta_k} \coloneqq \curlyb{\mubold \in \Ccal \suchthat \norm{\mubold-\mubold_k}\leq \Delta_k}.
\end{equation}

To ensure the proposed trust-region method is globally convergent, we require the objective models $F$ and $m_k$ to satisfy the following assumptions:
\begin{assume}\label{assum:at:tr_inexact}
    Assumptions on the trust-region method with an inexact gradient condition
    \begin{enumerate}[label=\textbf{(AT\arabic*)}]
        \item \label{assum:at:f} $F$ is twice continuously differentiable and bounded below
        \item \label{assum:at:m} $m_k$ is twice continuously differentiable for $k=1,2,\ldots.$
        \item \label{assum:at:hession}There exists $\zeta_1 >0, \zeta_2>1$ 
        such that $\norm{\nabla^2 F(\mubold)} \leq \zeta_1$ and $\norm{\nabla^2 m_k(\mubold)} \leq \zeta_2 - 1$
        \item \label{assum:at:grad} There exists $\xi >0 $ such that
        \begin{equation} \label{eqn:tr:asym_err}
            \norm{\nabla m_k (\mubold_k)-\nabla F(\mubold_k)} \leq \xi \min\curlyb{\norm{\chi_{m}(\mubold_k)}, \Delta_k},
        \end{equation}
        where $\chi_{m}(\mubold)$ is the reduced version of $\chi(\mubold)$, i.e., $\chi_{m}(\mubold)$ is defined as
        \begin{equation}\label{eqn:chi_m}
            \chi_m(\mubold) \coloneqq \Pbm_\Ccal(\mubold-\nabla m_k(\mubold), \mubold_l, \mubold_u) - \mubold
        \end{equation}
        \item \label{assum:at:curvature} There exists a positive constant $\beta$ such that $\beta_k \leq \beta$ for all $k \in \Nbb$, 
        where the upper bound of the curvature, $\beta_k$, is defined as  
        \begin{equation}
        \label{eqn:beta_k}
        \beta_k \coloneqq 1 + \max_{\xbm\in\Ccal_{\Delta_k}} \norm{\nabla^2 m_k(\xbm)}_2       
        \end{equation}
        \item \label{assum:at:cauchy} There exists $\kappa \in (0,1)$ such that solution of the trust-region subproblem satisfies the fraction of Cauchy decrease
        \begin{equation}\label{eqn:cauchy_decrease}
        m_k(\mubold_k) - m_k(\check{\mubold}_k) \geq
        \kappa \norm{\chi_m(\mubold_k)}
        \min\bracket{\frac{\norm{\chi_m(\mubold_k)}}{\beta_k}, \Delta_k},
        \end{equation}
        where $\check{\mubold}_k$ is the candidate step at the $k$th trust-region iteration.
    \end{enumerate}
    \label{trammo:assump}
\end{assume}

We use the trust-region method described in \cite{yano_globally_2021} adapted
to the projected gradient criticality measure \cite{kouri_inexact_2021} and
the asymptotic error condition, \ref{assum:at:grad}, where $\xi > 0$ is
an arbitrary constant independent of $k$. Similar to its original form
in~\cite{kouri_trust-region_2013}, the error bound requires that the model $m_k$
must be asymptotically accurate as $\norm{\chi_{m}(\mubold_k)} \rightarrow 0$ or
$\Delta_k \rightarrow 0 $. However, due to the arbitrariness of $\xi$, it may be
less effective for a fixed $k$. To ensure the error condition is practical, we seek
a \textit{computable} error indicator $\varphi_k: \Dcal \rightarrow \Rbb$ such that
\begin{equation}\label{eqn:tr:grad_cond_varphi}
 \norm{\nabla m_k(\mubold_k) - \nabla F(\mubold_k)} \leq \xi \varphi_k(\mubold_k),
\end{equation}
which reduces the gradient condition to
\begin{equation}\label{eqn:tr:varphi}
\varphi_k(\mubold_k) \leq \kappa_\varphi \min\{\norm{\chi_{m}(\mubold_k)},\Delta_k\},
\end{equation}
where $\kappa_\varphi > 0$ is a chosen constant. This provides a computable criterion in selecting model $m_k$ at the $k\mathrm{th}$ trust-region step. The complete trust-region algorithm is summarized in Algorithm~\ref{alg:tr_inexact}.
\begin{algorithm}[H]
    \begin{algorithmic}[1]
        \REQUIRE Current iterate $\mubold_k$ and radius $\Delta_k$, and parameters $0<\gamma_1\leq\gamma_2<1$, $\Delta_\mathrm{max}>0$, $0< \eta_1<\eta_2<1$ \\
        \ENSURE Next iterate $\mubold_{k+1}$
        \STATE 
            {\bf Model update:} Construct the approximation model, $m_{k}(\mubold)$, that satisfies \ref{assum:at:m} and \ref{assum:at:grad}.
        \STATE 
            {\bf Step computation:} Solve the trust-region subproblem to get the candidate center $\check{\mubold}_{k}$
            \begin{equation*}
             \optunc{\mubold \in \Ccal_{\Delta_k}}{m_k(\mubold)},
            \end{equation*}
            such that $\check{\mubold}_k$ satisfies the fraction of \textit{Cauchy decrease condition} in (\ref{eqn:cauchy_decrease}).
        \STATE
            {\bf Actual-to-predicted reduction ratio:} Evaluate the actual-to-predicted reduction ratio
            \begin{equation}\label{eqn:tr:redu_ratio}
                \varrho_{k}=\frac{F(\mubold_{k})-F(\check{\mubold}_{k})}{m_k(\mubold_{k})-m_k(\check{\mubold}_{k})}
            \end{equation}

        \STATE {\bf Step acceptance}: \\
            \begin{tabbing}
                \qquad \qquad \qquad 
                \=\textbf{if} \qquad \=$\varrho_{k} \geq \eta_1$  \qquad \= \textbf{then}  \qquad \=$\mubold_{k+1}=\check\mubold_k$ \qquad \= \textbf{else} \qquad \=$\mubold_{k+1}=\mubold_k$ \qquad \= \textbf{end if}
            \end{tabbing}
            
        \STATE
            {\bf Trust-region radius update:} \\
            \begin{tabbing}
                \qquad \qquad \qquad 
                \=\textbf{if} \qquad \=$\varrho_{k} < \eta_1$ \qquad \qquad \= \textbf{then}  \qquad \=$\Delta_{k+1} \in [\gamma_1 \Delta_k, \gamma_2 \Delta_k]$ \qquad \= \textbf{end if} \\

                \>\textbf{if}  \>$\varrho_{k} \in [\eta_1, \eta_2)$ \>\textbf{then} \>$\Delta_{k+1} \in [\gamma_2 \Delta_k, \Delta_k]$ \>\textbf{end if}\\

                \>\textbf{if}  \>$\varrho_{k} \geq \eta_2$ \>\textbf{then} \>$\Delta_{k+1} \in [\Delta_k, \Delta_{\text{max}}]$ \>\textbf{end if} 
            \end{tabbing}
\end{algorithmic} 
\caption{Trust-region method with inexact gradient condition}
\label{alg:tr_inexact}
\end{algorithm}

Theorem~\ref{thm:xi_m-zero} establishes the global convergence for the trust-region method described in Algorithm~\ref{alg:tr_inexact}. Similar proofs can be found in \cite{toint_global_1988,kouri_trust-region_2013,yano_globally_2021} for different accuracy assumptions.
\begin{theorem}\label{thm:xi_m-zero}
Suppose Assumptions~\ref{assum:at:tr_inexact} hold. Then
\begin{equation}
    \liminf_{k \rightarrow \infty} \norm{\chi_m(\mubold_k)} = \liminf_{k \rightarrow \infty} \norm{\chi(\mubold_k)} = 0
\end{equation}
\end{theorem}
\begin{proof}
See \ref{sec:tr-proofs}.
\end{proof}

\subsection{Augmented Lagrangian framework}\label{sec:trammo:auglag}
Next, we introduce the augmented Lagrangian framework of \cite{nocedal_numerical_2006} to solve (\ref{eqn:auglag_reduced}), which uses the augmented Lagrangian function (frozen Lagrange multiplier estimates $\thetabold$ and penalty parameter $\tau$) as the objective function $F$ in (\ref{eqn:generic_bndopt}) and the inexact trust-region method in Algorithm~\ref{alg:tr_inexact} to solve the subproblem until $\mubold_i^\star \in \Ccal$ that satisfies
\begin{equation}\label{eqn:proj_optim}
    \norm{\mubold_i^\star - \Pbm_\Ccal(\mubold_i^\star - \nabla f(\mubold_i^\star;\thetabold_i,\tau_i), \mubold_l, \mubold_u)}_\infty \leq \omega_i,
\end{equation}
where $i$ denotes the major (outer) iteration number, and $\omega_i$ is the convergence tolerance for the current major iteration. In this setting, the values of $\thetabold_i$ and $\tau_i$ are fixed at the iteration $i$, i.e., they will not change throughout all trust-region iterations when solving (\ref{eqn:proj_optim}). The augmented Lagrangian framework can become computationally expensive for small $\omega_i$ due to a large number of subproblems iterations. In practice, we limit the maximum number of iterations to 50 to ensure that the subproblem is solved sufficiently while avoiding excessive iterations.

We initialize the Lagrange multipliers to be $\thetabold_0 = \zerobold$ to begin the optimization process and use an iterative method derived from first-order optimality to update Lagrangian multipliers as
\begin{equation}\label{eqn:tr:lag_multp_update}
    \thetabold_{i+1} = \thetabold_{i} - \tau \cbm(\mubold^\star_i).
\end{equation}
When the current iterate $\mubold_i^\star$ is located in the infeasible region, we enlarge the penalty parameter by a scaling factor $a > 1$ to force the optimization trajectory toward the feasible region. The complete algorithm is summarized in Algorithm~\ref{alg:auglag}.
\begin{algorithm}[H]
    \begin{algorithmic}[1]
        \REQUIRE Initial iterate $\mubold_0$, Lagrangian multipliers $\thetabold_0$, penalty parameter $\tau_0$, scaling factor $a$, constraint violation tolerances $\pi_0$ and $\pi_\star$, and optimality tolerances $\omega_0$ and $\omega_\star$
        \ENSURE Next iterate $\mubold_{i+1}$
        \STATE {\bf Step computation:} Approximately solve for an iterate, $\mubold^\star_i$, using a bound-constrained optimization solver (e.g., Algorithm~\ref{alg:tr_inexact}) such that (\ref{eqn:proj_optim}) is satisfied.
        \STATE {\bf Termination criteria:} Terminate if $\mubold_i^\star$ satisfies optimality criteria
\begin{equation*}
\norm{\mubold_i^\star - \Pbm_\Ccal(\mubold_i^\star - \nabla f(\mubold_i^\star;\thetabold_i,\tau_i), \mubold_l, \mubold_u)}_\infty \leq \omega_\star, \qquad
\norm{\cbm(\ubm^\star(\mubold^\star_i), \mubold^\star_i)} \leq \pi_\star.
\end{equation*}
        \STATE {\bf Step is feasible, i.e., $\norm{\cbm(\ubm^\star(\mubold^\star_i), \mubold^\star_i)} \leq \pi_i$:} Update Lagrangian multipliers
            \begin{equation}
                    \thetabold_{i+1} = \thetabold_{i} - \tau \cbm(\ubm^\star(\mubold^\star_i), \mubold^\star_i), \quad
                    \tau_{i+1} = \tau_i, \quad
                    \pi_{i+1} = \pi_i / \tau^{0.9}_{i+1}, \quad
                    \omega_{i+1} = \omega_i/\tau_{i+1}
            \end{equation}
        \STATE {\bf Step is infeasible, i.e., $\norm{\cbm(\ubm^\star(\mubold^\star_i), \mubold^\star_i)} > \pi_i$:} Update penalty parameter
        \begin{equation}
                \thetabold_{i+1} = \thetabold_{i}, \quad
                \tau_{i+1} = a\tau_i, \quad
                \pi_{i+1} = 1/\tau^{0.1}_{i+1}, \quad
                \omega_{i+1} = 1/\tau_{i+1}
        \end{equation}
        \STATE {\bf $i = i + 1$, return to Step 1}
\end{algorithmic}
\caption{The augmented Lagrangian framework}
\label{alg:auglag}
\end{algorithm}
In this work, we take $\pi_0 = 1 / \tau_0$ and $\omega_0 = 1 / \tau_0^{0.1}$ \cite{nocedal_numerical_2006}, and study the impact of $\tau_0$ and $a$ (Section~\ref{sec:numexp}).

\subsection{Accelerated bound-constrained optimization using trust regions and on-the-fly model hyperreduction}
\label{sec:eqp_trammo}
Finally, we introduce the proposed EQP/BTR method, as an extension of the trust-region approach in Section~\ref{sec:trammo:tr_inexact}, to solve the bound-constrained subproblem in Step 2 of Algorithm~\ref{alg:auglag} by merging the developments of Sections~\ref{sec:hyperreduction}~and~\ref{sec:trammo}, namely, the hyperreduced models and the globally convergent trust-region method that relies on asymptotic error bounds.
In the remainder of this section, we freeze Algorithm~\ref{alg:auglag} at the $i$th major iteration and drop all major iteration indices. Subscripts will denote the trust-region step of major iteration $i$, i.e., $\Abm_k$ denotes the value of the quantity $\Abm$ at the $kth$ trust-region step of the $i$th major iteration in Algorithm~\ref{alg:auglag}.

Now, we take the trust-region model ($m_k$) at the $k$th trust-region center $\mubold_k$ as the quadratic approximation to the hyperreduced AL function
\begin{equation}\label{eqn:tr:mk}
 m_k(\mubold;\thetabold,\tau) = \tilde{f}_{\Phibold_k}(\mubold_k; \rhobold_k, \thetabold,\tau) + \nabla\tilde{f}_{\Phibold_k}(\mubold_k; \rhobold_k,\thetabold,\tau) + \frac{1}{2} (\mubold-\mubold_k)^T\tilde\Hbm_k(\mubold-\mubold_k),
\end{equation}
where $\Phibold_k\in\Rbb^{N_\ubm \times n_k}$ is the reduced basis, $\rhobold_k\in\Ccal_{\Phibold_k,\Xibold_k,\deltabold_k}$ is the weight
vector, and $\tilde\Hbm_k\in\Rbb^{N_\mubold\times N_\mubold}$ is the Hessian of $\tilde{f}_{\Phibold_k}$ defined by
\begin{equation}
 \tilde\Hbm_k \coloneqq \nabla^2\tilde{f}_{\Phibold_k}(\mubold_k; \rhobold_k,\thetabold,\tau),
\end{equation}
all at the $k$th trust-region center. Hessian-vector products are approximated using a first-order finite difference approach given by
\begin{equation}
 \tilde\Hbm_k \vbm \approx \frac{1}{\varepsilon}\bracket{\nabla\tilde{f}_{\Phibold_k}(\mubold_k+\varepsilon \vbm; \rhobold_k,\thetabold,\tau) - \nabla\tilde{f}_{\Phibold_k}(\mubold_k; \rhobold_k,\thetabold,\tau) },
\end{equation}
where $\varepsilon\in\Rbb_{>0}$ is the finite difference step size ($\varepsilon = 10^{-6}$ in this work). As we assume that the size of the reduced basis $n_k$ is small and the EQP weight vector $\rhobold_k$ is sparse, the evaluations of the model $m_k$ and its gradient $\nabla m_k$ are much cheaper than the HDM model evaluations $f$.

We follow the same procedure as in \cite{wen_globally_2023} to construct the reduced basis that includes the HDM sensitivities at the starting point. The reduced basis is chosen to  as
\begin{equation}\label{eqn:reduc_basis}
 \Phibold_k = \mathtt{GramSchmidt}\left(\begin{bmatrix} \ubm^\star(\mubold_k) & \lambdabold^\star(\mubold_k;\thetabold,\tau) & \partial_\mubold\ubm^\star(\mubold_0) & \Phibold_k^\mathtt{p} & \Phibold_k^\mathtt{a} \end{bmatrix}\right),
\end{equation}
to guarantee $\ubm^\star(\mubold_k), \lambdabold^\star(\mubold_k;\thetabold,\tau) \in\mathrm{Ran}~\Phibold_k$, where
\begin{equation}\label{eqn:tr:pod_snapshots}
 \Phibold_k^\mathtt{p} \coloneqq \mathtt{POD}_{p_k}(\Ubm_{k-1})\in\Rbb^{N_\ubm\times p_k}, \qquad
 \Phibold_k^\mathtt{a} \coloneqq \mathtt{POD}_{q_k}(\Vbm_{k-1})\in\Rbb^{N_\ubm\times q_k}
\end{equation}
are optimal compressions of state and adjoint snapshots from \textit{all previous iterations}
\begin{equation}\label{eqn:tr:snapshots}
\begin{aligned}
 \Ubm_k &= \begin{bmatrix} \ubm^\star(\mubold_0) & \cdots & \ubm^\star(\mubold_k)\end{bmatrix} \in \Rbb^{N_\ubm\times(k+1)}, \\
 \Vbm_k &= \begin{bmatrix} \lambdabold^\star(\mubold_0;\thetabold,\tau) & \cdots & \lambdabold^\star(\mubold_k;\thetabold,\tau) \end{bmatrix}  \in \Rbb^{N_\ubm\times(k+1)}.
\end{aligned}
\end{equation}
This construction of the reduced basis leads to a basis of size $n_k = 2+N_\mubold+p_k+q_k$ (bounded by $n_k \leq 2(k+1)+N_\mubold$), and $0\leq p_k\leq k$ and $0\leq q_k\leq k$
are user-defined parameters. Fast low-rank singular value decomposition updates \cite{brand_fast_2006} can be used to construct $\Phibold_k^\mathtt{p}$ and $\Phibold_k^\mathtt{a}$ efficiently.

Next, we choose the constraint set $\Ccal_{\Phibold_k,\Xibold_k,\deltabold_k}$ such that
\begin{equation}\label{eqn:tr:conv_constr_set}
\Ccal_{\Phibold_k,\Xibold_k,\deltabold_k} \subseteq
\Ccal_{\Phibold_k,\Xibold_k,\delta_{\mathtt{rp},k}}^\mathtt{rp} \cap \Ccal_{\Phibold_k,\Xibold_k,\delta_{\mathtt{lra},k}}^\mathtt{lra} \cap \Ccal_{\Phibold_k,\Xibold_k,\delta_{\mathtt{lga},k}}^\mathtt{lga}  \cap \Ccal_{\Phibold_k,\Xibold_k,\delta_{\mathtt{dcy},k}}^\mathtt{dcy} \cap \Ccal_{\Phibold_k,\Xibold_k,\delta_{\mathtt{dc\mu},k}}^\mathtt{dc\mu} \cap \Ccal_{\Phibold_k,\Xibold_k,\delta_{\mathtt{c},k}}^\mathtt{c}
\end{equation}
and the EQP training set $\Xibold_k\subset\Dcal$ such that $\mubold_k \in \Xibold_k$ to
ensure Corollary~\ref{cor:qoi_grad_errbnd} holds, which leads to the result
\begin{equation}
\begin{aligned}
\norm{\nabla f(\mubold_k;\thetabold, \tau) - \nabla m_k(\mubold_k;\thetabold,\tau)}
&=\norm{\nabla f(\mubold_k;\thetabold,\tau) - \nabla \tilde{f}_{\Phibold_k}(\mubold_k;\rhobold_k,\thetabold,\tau)}\\
&\leq c_3'\delta_{\mathtt{rp},k}+c_4'\delta_{\mathtt{lra},k}+\delta_{\mathtt{lga},k}+
c_5'\tau\delta_{\mathtt{c},k}+c_6'\tau\delta_{\mathtt{dcy},k}+c_7'\tau\delta_{\mathtt{dc\mu},k}
\end{aligned}
\end{equation}
Therefore, we take $\varphi_k$ in (\ref{eqn:tr:varphi}) to be
\begin{equation}\label{eqn:varphi_def}
 \varphi_k \coloneqq \kappa_1 \delta_{\mathtt{rp},k} + \kappa_2 \delta_{\mathtt{lra},k} + \kappa_3 \delta_{\mathtt{lga},k} + \kappa_4 \tau \delta_{\mathtt{c},k} + \kappa_5 \tau \delta_{\mathtt{dcy},k} + \kappa_6 \tau \delta_{\mathtt{dc\mu},k},
\end{equation}
where $\kappa_1,\kappa_2,\dots,\kappa_6>0$ are user-defined parameters. Note that (\ref{eqn:varphi_def}) differs from the similar definition in~\cite{wen_globally_2023} where we keep the penalty parameter, $\tau$, to enforce the accuracy on the penalty term in the case that $\tau$ approaches a large number.
Then, condition (\ref{eqn:tr:varphi}) leads to the following bound on the tolerances
\begin{equation}
\kappa_1 \delta_{\mathtt{rp},k} + \kappa_2 \delta_{\mathtt{lra},k} + \kappa_3 \delta_{\mathtt{lga},k} + \kappa_4 \tau \delta_{\mathtt{c},k}+ \kappa_5 \tau \delta_{\mathtt{dcy},k} + \kappa_6 \tau \delta_{\mathtt{dc\mu},k}  \leq \hat\kappa~\mathrm{min}\left\{\norm{\chi_{m}(\mubold_k; \thetabold, \tau)},\Delta_k\right\},
\end{equation}
where $\chi_m(\mubold; \thetabold, \tau) = \Pbm_\Ccal(\mubold - \nabla m_k(\mubold;\thetabold,\tau),\mubold_l,\mubold_u)-\mubold$.
For simplicity, we impose the slightly stronger condition that equally splits the bound among the six tolerances as
\begin{equation}\label{eqn:eqp_tol_cond}
\begin{aligned}
 \delta_{\mathtt{rp},k} &\leq \frac{\hat\kappa}{6\kappa_1}\mathrm{min}\left\{\norm{\chi_{m}(\mubold_k;\thetabold,\tau)},\Delta_k\right\}, \\
 \delta_{\mathtt{lra},k} &\leq \frac{\hat\kappa}{6\kappa_2}\mathrm{min}\left\{\norm{\chi_{m}(\mubold_k; \thetabold,\tau)},\Delta_k\right\}, \\
 \delta_{\mathtt{lga},k} &\leq \frac{\hat\kappa}{6\kappa_3}\mathrm{min}\left\{\norm{\chi_{m}(\mubold_k;\thetabold,\tau)},\Delta_k\right\}, \\
 \delta_{\mathtt{c},k} &\leq \frac{\hat\kappa}{6\tau\kappa_4}\mathrm{min}\left\{\norm{\chi_{m}(\mubold_k;\thetabold,\tau)},\Delta_k\right\}, \\
 \delta_{\mathtt{dcy},k} &\leq \frac{\hat\kappa}{6\tau\kappa_5}\mathrm{min}\left\{\norm{\chi_{m}(\mubold_k;\thetabold,\tau)},\Delta_k\right\}, \\
 \delta_{\mathtt{dc\mu},k} &\leq \frac{\hat\kappa}{6\tau\kappa_6}\mathrm{min}\left\{\norm{\chi_{m}(\mubold_k;\thetabold,\tau)},\Delta_k\right\}.
\end{aligned}
\end{equation}
The weighting of the tolerances can be achieved through the choice of $\kappa_1,\kappa_2,\dots,\kappa_6$, although we take $\kappa_1=\kappa_2=\dots=\kappa_6$ in this work. As shown in (\ref{eqn:eqp_tol_cond}), the penalty parameter $\tau$ is accountd for in the tolerances related to the constraint set $\cbm$, which enforces the accuracy on the constraint terms when $\tau$ increase. Finally, with these choices, we chose the weight vector $\rhobold_k$
to the solution of (\ref{eqn:linprog}), i.e.,
\begin{equation}\label{eqn:tr:compute_eqp_weights}
\rhobold_k = \rhobold^\star(\Phibold_k, \Xibold_k,\deltabold_k; \thetabold, \tau),
\end{equation}
where the tolerances
\begin{equation}\label{eqn:tr:all_constr}
 \deltabold_k = (\delta_\mathtt{dv},\delta_{\mathtt{rp},k},\delta_{\mathtt{lra},k},\delta_{\mathtt{lga},k}, \delta_{\mathtt{c},k}, \delta_{\mathtt{dcy},k}, \delta_{\mathtt{dc\mu},k}, \delta_\mathtt{lq}, \delta_\mathtt{rs})
\end{equation}
are chosen according to (\ref{eqn:eqp_tol_cond}) with arbitrary $\delta_\mathtt{dv},\delta_\mathtt{lq},\delta_\mathtt{rs} > 0$.

The trust-region subproblem does not need to be solved exactly; rather, it must only satisfy the fraction of the Cauchy decrease condition. We assume the fraction of Cauchy decrease, \ref{assum:at:cauchy}, based on the flexibility of the choice of the first-order optimality condition in \cite{conn_trust-region_2000}. For bound-constrained problems, an $l^\infty$ norm trust-region constraint is natural because it is also a bound constraint. To approximately solve $l^\infty$-norm trust-region subproblems, we use the algorithm in~\cite{conn_testing_1988}: (1) compute the generalized Cauchy point, $\mubold^c$, along the projected path and (2) use an active-set truncated conjugate gradient (TCG) method is used to solve reduced quadratic program. 
The step candidate is evaluated using the actual-to-predicted reduction ratio, $\varrho_k$ in (\ref{eqn:tr:redu_ratio}), to determine whether to accept the step and how to adjust the trust-region radius.

The complete EQP/BTR algorithm is summarized in Algorithm~\ref{alg:tr_hyp}. The specific construction of the reduced basis $\Phibold_k$ and weight vector $\rhobold_k$ outlined above are sufficient to ensure the complete algorithm is globally convergent.
\begin{algorithm}[H]
    \begin{algorithmic}[1]
        \REQUIRE Current iterate $\mubold_k$ and radius $\Delta_k$, trust-region parameters $0<\gamma_1\leq\gamma_2<1, \Delta_\mathrm{max}>0, 0<\eta_1<\eta_2<1$, snapshot matrices $\Ubm_k$ and $\Vbm_k$, frozen Lagrange multiplier estimate $\thetabold$ and penalty parameter $\tau$
       \ENSURE Next iterate $\mubold_{k+1}$, updated snapshot matrices $\Ubm_{k+1}$ and $\Vbm_{k+1}$
        \STATE {\bf Model update:} Build approximation model $m_k(\mubold;\thetabold,\tau)$ in (\ref{eqn:tr:mk})
        \begin{itemize}[topsep=0pt, itemsep=0pt]
          \setlength\itemsep{0pt}
          \item Solve primal and adjoint HDM: $\ubm^\star(\mubold_k)$, $\lambdabold^\star(\mubold_k; \thetabold, \tau)$
          \item Construct reduced basis $\Phibold_k$ according to (\ref{eqn:reduc_basis})
          \item Compute EQP weights $\rhobold_k$ according to (\ref{eqn:tr:compute_eqp_weights}) with tolerances given by (\ref{eqn:eqp_tol_cond}) and (\ref{eqn:tr:all_constr})
          \item Update snapshot matrices
          \begin{equation}
           \Ubm_{k+1} = \begin{bmatrix} \ubm^\star(\mubold_k) & \Ubm_k \end{bmatrix}, \qquad
           \Vbm_{k+1} = \begin{bmatrix} \lambdabold^\star(\mubold_k; \thetabold, \tau) & \Vbm_k \end{bmatrix}
          \end{equation}
        \end{itemize}
        \STATE {\bf Step computation:} Compute the generalized Cauchy point, $\mubold^c$ and solve the active-set trust-region subproblem to get the candidate center, $\check\mubold_k$
        \STATE {\bf Actual-to-predicted reduction ratio:} Compute the actual-to-predicted reduction ratio $\varrho_k$
            \begin{equation}
                \varrho_{k}=\frac{f(\mubold_{k};\thetabold,\tau)-f(\check{\mubold}_{k};\thetabold,\tau)}{m_k(\mubold_{k};\thetabold,\tau)-m_k(\check{\mubold}_{k};\thetabold,\tau)}
            \end{equation}
        \STATE {\bf Step acceptance}: \\
            \begin{tabbing}
                \qquad \qquad \qquad 
                \=\textbf{if} \qquad \=$\varrho_{k} \geq \eta_1$  \qquad \= \textbf{then}  \qquad \=$\mubold_{k+1}=\check\mubold_k$ \qquad \= \textbf{else} \qquad \=$\mubold_{k+1}=\mubold_k$ \qquad \= \textbf{end if}
            \end{tabbing}
            
        \STATE
            {\bf Trust-region radius update:} \\
            \begin{tabbing}
                \qquad \qquad \qquad 
                \=\textbf{if} \qquad \=$\varrho_{k} < \eta_1$ \qquad \qquad \= \textbf{then}  \qquad \=$\Delta_{k+1} \in [\gamma_1 \Delta_k, \gamma_2 \Delta_k]$ \qquad \= \textbf{end if} \\

                \>\textbf{if}  \>$\varrho_{k} \in [\eta_1, \eta_2)$ \>\textbf{then} \>$\Delta_{k+1} \in [\gamma_2 \Delta_k, \Delta_k]$ \>\textbf{end if}\\

                \>\textbf{if}  \>$\varrho_{k} \geq \eta_2$ \>\textbf{then} \>$\Delta_{k+1} \in [\Delta_k, \Delta_{\text{max}}]$ \>\textbf{end if} 
            \end{tabbing}
\end{algorithmic}
\caption{Trust-region method with hyperreduced approximation models}
\label{alg:tr_hyp}
\end{algorithm}

\begin{theorem} \label{thm:globconv}
Suppose Assumptions~\ref{assum:at:f}-\ref{assum:at:hession},~\ref{assum:at:curvature},~and~\ref{assum:hdm}-\ref{assum:eqp} hold
with $\Rcal\supset\cup_{k=1}^\infty\Ccal_{\Phibold_k,\Xibold_k,\deltabold_k}$. The sequence of trust-region centers $\{\mubold_k\}$ generated by Algorithm~\ref{alg:tr_hyp} satisfy
\begin{equation}\label{eqn:tr:inf_mk}
 \liminf_{k\rightarrow\infty}\,\norm{\chi_m(\mubold_k;\thetabold,\tau)} = \liminf_{k\rightarrow\infty}\,\norm{\chi(\mubold_k;\thetabold,\tau)} = 0
\end{equation}
independent of the choice of $p_k$ and $q_k$.
\begin{proof}
From Theorem~\ref{thm:globconv}, the result holds if (\ref{eqn:tr:asym_err}) and (\ref{eqn:cauchy_decrease}) hold.  With the choice of $\Phibold_k$ in (\ref{eqn:reduc_basis}), we have $\ubm^\star(\mubold_k),\lambdabold^\star(\mubold_k)\in\mathrm{Ran}~\Phibold_k$ independent of the choice of $p_k$ or $q_k$. Furthermore, from the choice of $\Xibold_k$ such that $\mubold_k\in\Xibold_k$ (e.g., $\Xibold_k=\{\mubold_k\}$)) and the constraint set (\ref{eqn:tr:conv_constr_set}), the assumptions of Corollary~\ref{cor:qoi_grad_errbnd} are satisfied, which implies the existence of constants $c_3',c_4', c_5', c_6', c_7'>0$ such that
\begin{equation}
\begin{aligned}
 \norm{\nabla f(\mubold_k;\thetabold,\tau) - \nabla m_k(\mubold_k;\thetabold,\tau)}
 &= \norm{\nabla f(\mubold_k;\thetabold,\tau) - \nabla \tilde{f}_\Phibold(\mubold_k;\rhobold_k,\thetabold,\tau)}\\
 &\leq c_3'\delta_{\mathtt{rp},k} + c_4'\delta_{\mathtt{lra},k} + \delta_{\mathtt{lga},k}  + c_5' \tau \delta_{\mathtt{c},k} + c_6' \tau \delta_{\mathtt{dcy},k} + c_7' \tau \delta_{\mathtt{dc\mu},k}.
\end{aligned}
\end{equation}
From the condition in (\ref{eqn:eqp_tol_cond}), this reduces to
\begin{equation}
\norm{\nabla f(\mubold_k;\thetabold,\tau) - \nabla m_k(\mubold_k;\thetabold,\tau)} \leq \hat\kappa \paren{\frac{c_3'}{6\kappa_1} + \frac{c_4'}{6\kappa_2} + \frac{1}{6\kappa_3} + \frac{c_5'}{6\kappa_4} + \frac{c_6'}{6\kappa_5} + \frac{c_7'}{6\kappa_6}} \mathrm{min}\left\{\norm{\chi_m(\mubold_k;\thetabold,\tau)},\Delta_k\right\},
\end{equation} 
which is identical to (\ref{eqn:tr:asym_err}) with $\xi = \hat\kappa \paren{\frac{c_3'}{6\kappa_1} + \frac{c_4'}{6\kappa_2} + \frac{1}{6\kappa_3} + \frac{c_5'}{6\kappa_4} + \frac{c_6'}{6\kappa_5} + \frac{c_7'}{6\kappa_6}}$. Furthermore, the TCG trust-region subproblem solver \cite{conn_testing_1988} delivers a candidate step satisfying (\ref{eqn:cauchy_decrease}), which leads to the desired result.
\end{proof} 
\end{theorem}

\begin{remark}
If there are no bound constraints present in (\ref{eqn:hdm_full_space}), i.e., the problem only possess equality constraints, the inexact trust-region framework for unconstrained problems detailed in~\cite{wen_globally_2023} can be used to solve (\ref{eqn:tr:subprob}).
\end{remark}

\begin{remark}
If the optimization problem (\ref{eqn:hdm_full_space}) without constraints $\cbm$ is equivalent to setting $\thetabold = \zerobold$ and $\tau = 0$ in the AL framework, which makes the outer loop (major iterations) unnecessary. The approach introduced in \cite{wen_globally_2023} is sufficient for such problems.
\end{remark}

\begin{remark}
We choose the quadratic model in (\ref{eqn:tr:mk}) rather than the more obvious choice of directly using the hyperreduced model,
i.e., $m_k(\mubold;\thetabold,\tau) = \tilde{f}_{\Phibold_k}(\mubold; \rhobold_k,\thetabold,\tau)$, as done in other work \cite{zahr_adaptive_2016,qian_certified_2017,yano_globally_2021} because the trust-region subproblems are less expensive to solve. The reader is referred to \cite{wen_globally_2023} for a detailed discussion.
\end{remark}

\begin{remark}
In this work, we take $\Xibold_k = \{\mubold_k\}$ since it is the simplest option that satisfies the requirement on the EQP training set $\Xibold_k$ although other options possible \cite{wen_globally_2023}.
\end{remark}

\begin{remark}
We lag the right-hand side conditions in (\ref{eqn:eqp_tol_cond}) to $m_{k-1}(\mubold_{k-1};\thetabold,\tau)$ to simplify the implementation and improve efficiency \cite{wen_globally_2023}.
\end{remark} 

\begin{remark}
In some cases, an affine subspace approximation $\ubm^\star \approx \bar\ubm + \Phibold\hat\ybm^\star$ is preferred to the linear
approximation in (\ref{eqn:rom_ansatz}). We choose the affine offset at the $k$th iteration to be the trust-region center, i.e.,
$\bar\ubm_k = \ubm^\star(\mubold_k)$, which no longer requires $\ubm^\star(\mubold_k)$ to be explicitly included
in the basis construction in (\ref{eqn:reduc_basis}). Lastly, we modify the primal snapshot matrix to be
\begin{equation}
 \Ubm_{k-1} = \begin{bmatrix} \ubm^\star(\mubold_0)-\ubm^\star(\mubold_k) & \cdots & \ubm^\star(\mubold_{k-1})-\ubm^\star(\mubold_k) \end{bmatrix}
\end{equation}
because the reduced basis $\Phibold_k$ now only represents \textit{deviations} of the primal state from the offset (rather than the primal state itself).
\end{remark}

\begin{remark}
    In practice, we reuse $\ubm^\star(\mubold_{j-1}^\star)$ and $\lambdabold^\star(\mubold_{j-1}^\star; \thetabold_{j-1}, \taubold_{j-1})$, i.e., the primal and adjoint solution from the critical point of major iteration $j-1$, to construct the reduced basis in (\ref{eqn:reduc_basis}) at the beginning of $j$th iteration.
\end{remark}

\section{Numerical experiments}
\label{sec:numexp}
In this section, we demonstrate the performance of the proposed method with two aerodynamic shape optimization problems with state-based constraints, e.g., drag, lift. To provide a baseline for comparison, we solve the optimization problem in (\ref{eqn:hdm_full_space}) with only HDM evaluations (no model reduction) with a standard interior-point method. We compare the HDM performance with ROM/BTR and EQP/BTR methods embedded in the augmented Lagrangian framework.

To assess the performance of the proposed method, we study the computational cost required to achieve a given value of the objective function, i.e., given $\epsilon>0$, we investigate the costs for each algorithm to satisfy $S_i < \epsilon$, where $S_i$ is the normalized distance from the optimal objective value at major iteration $i$
\begin{equation}\label{eqn:numexp:si}
S_i \coloneqq \frac{\abs{j(\ubm^\star(\mubold_i^\star),\mubold_i^\star)-j_\star}}{j(\ubm^\star(\mubold_0),\mubold_0)},
\end{equation}
$\mubold_0$ is the initial guess for the optimal parameter configuration, and $j_\star$ is the optimal objective value determined by HDM optimization (interior-point) with optimality tolerance $\omega_\star = 10^{-10}$. Furthermore, we introduce the following notation for the objective function, constraint function, and optimality measure at the end of major iteration $i$
\begin{equation}
 j_i \coloneqq j(\ubm^\star(\mubold_i^\star),\mubold_i^\star), \qquad
 \cbm_i \coloneqq \cbm(\ubm^\star(\mubold_i^\star),\mubold_i^\star), \qquad
 \chi_i \coloneqq \chi(\mubold_i^\star).
\end{equation}
We quantify computational cost with the CPU time required to reach $S_i < \epsilon$ for a given $\epsilon$ and constraint violation tolerance.

In this work, there are only a few user-defined parameters because most of the EQP tolerances are prescribed in Section~\ref{sec:eqp_trammo} directly from the current optimization state (\ref{eqn:eqp_tol_cond}). Among the remaining user-defined parameters are those related to the trust-region algorithm itself. We fix these parameters are reasonable values determined from the thorough study in \cite{wen_globally_2023}: $\Delta_0=0.1$, $\eta_1 = 0.1$, $\eta_2 = 0.75$, $\gamma_1=0.5$, $\gamma_2 = 1$, and $\hat\kappa = 10^{-6}$. The remaining parameters are the EQP tolerances that do not appear in the convergence criteria but empirically accelerate the optimization, which we set as: $\delta_\mathtt{dv}=\delta_\mathtt{lq}=10^{-6}$, and $\delta_\mathtt{rs}= 5 \times 10^{-4}$. Finally, we set the constraint tolerance $\pi_\star = 10^{-6}$ and optimality tolerance $\omega_\star = 10^{-5}$. 

\subsection{Inverse design of an airfoil in inviscid, subsonic flow with one side constraint}
\label{sec:numexp1}
We consider an aerodynamic shape design problem that aims to recover an airfoil with a similar flow profile as the RAE2822 with lower drag starting from a NACA0012 at Mach number $M_\infty = 0.5$ and $2^\circ$ angle of attack. Let $\Omega\subset\Rbb^d$ ($d=2$) be the circular region around the NACA0012 airfoil with cord length $L=1$ that extends $10L$ from the leading edge, and consider steady, inviscid, compressible flow governed by the compressible Euler equations 
\begin{equation} \label{eqn:numexp:euler}
\pder{}{x_j}\left(\rho v_j\right) = 0, \quad
\pder{}{x_j}\left(\rho v_iv_j+P\delta_{ij}\right) = 0, \quad
\pder{}{x_j}\left(\left[\rho E+P\right]v_j\right) = 0,
\end{equation}
where $i = 1,\dots, d$ and the repeated index $j=1,\dots,d$ indicates the Einstein summation. The density of the fluid $\func{\rho}{\Omega}{\Rbb_{>0}}$, its velocity $\func{v_i}{\Omega}{\Rbb}$ in the $x_i$ direction for $i=1,\dots,d$, and its total energy $\func{E}{\Omega}{\Rbb_{>0}}$ are implicitly defined as the solution of (\ref{eqn:numexp:euler}).
For a fluid with calorically ideal properties, the pressure of the fluid, $\func{P}{\Omega}{\Rbb_{>0}}$, is related to the energy via the ideal gas law 
\begin{equation}
P = (\gamma-1)\left(\rho E - \frac{\rho v_i v_i}{2}\right),
\end{equation}
where $\gamma\in\Rbb_{>0}$ is the ratio of specific heats. The conservative variables are collected into a state vector $U : \Omega \rightarrow \Rbb^{d+2}$ with $U : x \mapsto (\rho(x),\rho(x)v(x), \rho(x)E(x))$. In this problem, a farfield boundary condition is applied to the cutoff surface of the domain and a slip wall condition ($v\cdot n = 0$) is applied to the airfoil surface.

The goal of this problem is to reconstruct a flow field similar to that of the RAE2822 airfoil, denoted as $U_\mathrm{RAE2822} : \Omega \rightarrow \Rbb^{d+2}$, with lower drag by adjusting the airfoil surface, which is parametrized using a Bezier curve with 18 control points ($\mubold$). The deformation is propagated throughout the mesh domain using linear elasticity. This leads to the following PDE-constrained optimization problem 
\begin{equation}\label{eqn:contin_inv}
    \optconTwo{U,\mubold}
    {\frac{1}{2}\int_{\Omega(\mubold)} \norm{U - U_\mathrm{RAE2822}}^2 ~ dV}
    {\Lcal(U;\mubold)=0}
    {D(U,\mubold) \leq 0.98 D_{\mathrm{RAE2822}}}
\end{equation}
where $D(U,\mubold)$ is the drag of the airfoil with shape defined by the control points $\mubold$ and flow field $U$, $D_{\mathrm{RAE2822}} = D(U_\mathrm{RAE2822},\mubold_\mathrm{RAE2822})$ is the drag of the RAE2822 airfoil, $\mubold_\mathrm{RAE2822}$ are the Bezier control points defining the RAE2822 airfoil surface, and $\Lcal$ is the differential operator that includes the Euler equations (\ref{eqn:numexp:euler}) and appropriate boundary conditions. The discrete optimization problem in (\ref{eqn:hdm_full_space}) is formulated by discretizing the governing equations and quantity of interest with a nodal discontinuous Galerkin (DG) method on a mesh consisting of $1018$ $\Pcal^2$ triangular elements (generated using DistMesh \cite{persson_simple_2004}).


\subsubsection{Influence of the initial penalty parameter $\tau_0$ and its scaling factor $a$}
\label{sec:param}
In the augmented Lagrangian framework, the initial penalty parameter $\tau_0$ and its scaling factor $a$ are user-defined parameters and play an important role in the efficiency of the algorithm, influencing the overall convergence speed and feasibility. Given that these two parameter jointly define $\tau_i = \tau_0 a^{k_i}$ for some $k_i$ throughout the optimization process, we examine their combined effect by exploring nine scenarios: $(\tau_0, a) \in [10, 25, 50] \times [10, 50, 100]$. Figure~\ref{fig:parm_fact_study} shows the convergence history of the objective value in (\ref{eqn:contin_inv}) and constraint violation for each case. The choice of $(\tau_0=10, a=50)$ outperforms other settings by achieving the fastest convergence to the feasible region that satisfies the stopping optimality tolerance. This observation underscores that even though higher values of $\tau_0$ and $a$ can strictly enforce feasibility trajectories throughout the optimization process, they can lead to an increase in the number of trust-region iterations required for step computation in Algorithm~\ref{alg:auglag}. Hence, a balanced approach in choosing $\tau_0$ and $a$ emerges as crucial for optimizing algorithmic efficiency and performance.
\begin{figure}
    \centering
    \begin{minipage}[t]{0.49\textwidth}
        \centering
        \tikzset{every picture/.style={scale=0.98}}
        \input{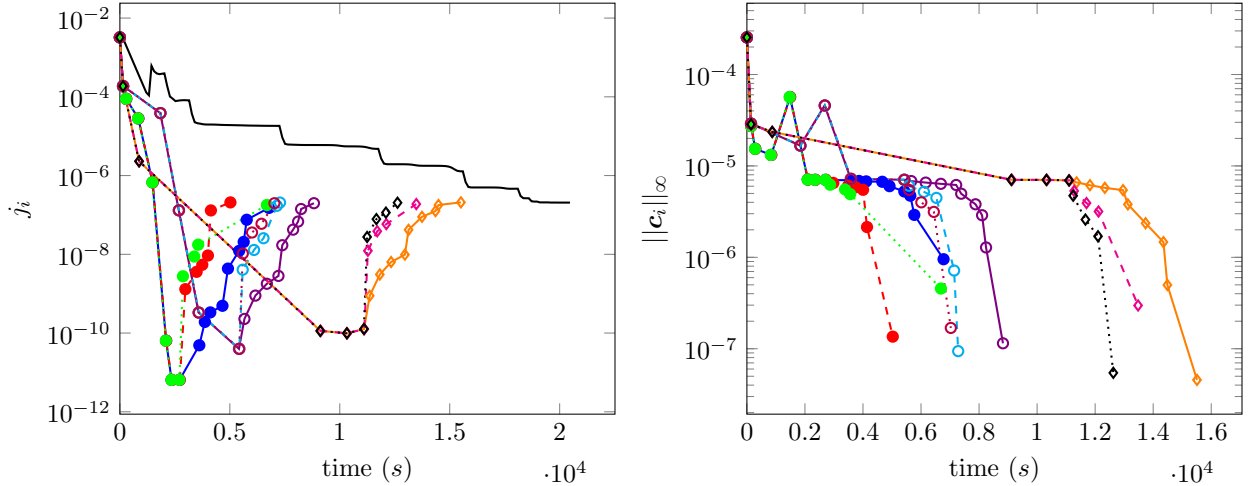}
    \end{minipage}
    \hfill
    \begin{minipage}[t]{0.49\textwidth}
        \centering
        \tikzset{every picture/.style={scale=0.98}}
        \begin{tikzpicture}
\begin{axis}[
xmin=0,
ymode=log,
xlabel={time $(s)$},
ylabel={$||\cbm_i||_\infty$}]
\addplot [mark=*, mark size=2, mark options={solid}, thick, solid, blue, mark repeat=1]
coordinates {
( 0.00000000e+00,  2.54616700e-04)
( 1.46714107e+02,  2.73712524e-05)
( 2.86712636e+02,  1.53848250e-05)
( 8.37328408e+02,  1.31854998e-05)
( 1.48689723e+03,  5.67032935e-05)
( 2.10009673e+03,  7.09083663e-06)
( 2.35143923e+03,  7.04327490e-06)
( 2.71593949e+03,  7.04244130e-06)
( 3.61208128e+03,  6.95997978e-06)
( 3.86323702e+03,  6.84775994e-06)
( 4.11138794e+03,  6.77247666e-06)
( 4.66517533e+03,  6.70836728e-06)
( 4.91444585e+03,  6.00666705e-06)
( 5.42422193e+03,  5.25398008e-06)
( 5.63071480e+03,  4.72363180e-06)
( 5.76825923e+03,  2.91323975e-06)
( 6.77684992e+03,  9.53474320e-07)};\label{line:cp10f10}

\addplot [mark=*, mark size=2, mark options={solid}, thick, dashed, red, mark repeat=1]
coordinates {
( 0.00000000e+00,  2.54616700e-04)
( 1.47369189e+02,  2.73712524e-05)
( 2.87925775e+02,  1.53848250e-05)
( 8.38970946e+02,  1.31854998e-05)
( 1.48994787e+03,  5.67032935e-05)
( 2.10484250e+03,  7.09083663e-06)
( 2.35755089e+03,  7.04327490e-06)
( 2.72453808e+03,  7.04244130e-06)
( 2.97720871e+03,  6.51871365e-06)
( 3.49007434e+03,  6.08827942e-06)
( 3.74687223e+03,  5.85416924e-06)
( 4.00308884e+03,  5.47921035e-06)
( 4.13878587e+03,  2.14456124e-06)
( 5.02667745e+03,  1.35636628e-07)};\label{line:cp10f50}

\addplot [mark=*, mark size=2, mark options={solid}, thick, dotted, green, mark repeat=1]
coordinates {
( 0.00000000e+00,  2.54616700e-04)
( 1.42828634e+02,  2.73712524e-05)
( 2.79765620e+02,  1.53848250e-05)
( 8.25034401e+02,  1.31854998e-05)
( 1.47270534e+03,  5.67032935e-05)
( 2.08946114e+03,  7.09083663e-06)
( 2.34230054e+03,  7.04327490e-06)
( 2.70551556e+03,  7.04244130e-06)
( 2.86776991e+03,  6.22843768e-06)
( 3.38326011e+03,  5.53315053e-06)
( 3.56017337e+03,  4.90783001e-06)
( 6.68591659e+03,  4.55532977e-07)};\label{line:cp10f100}

\addplot [mark=o, mark size=2, mark options={solid}, thick, solid, violet, mark repeat=1]
coordinates {
( 0.00000000e+00,  2.54616700e-04)
( 1.41572842e+02,  2.90695596e-05)
( 1.83818775e+03,  1.66885503e-05)
( 2.67781491e+03,  4.58074502e-05)
( 3.57844457e+03,  7.21979122e-06)
( 5.41504202e+03,  7.05866858e-06)
( 5.66426050e+03,  6.82524837e-06)
( 6.16231518e+03,  6.57386815e-06)
( 6.68366787e+03,  6.37269869e-06)
( 7.21102459e+03,  6.18598896e-06)
( 7.38495113e+03,  4.96466404e-06)
( 7.87791952e+03,  3.80537251e-06)
( 8.10400183e+03,  2.88865043e-06)
( 8.24109074e+03,  1.28203428e-06)
( 8.82422869e+03,  1.14674853e-07)};\label{line:cp25f10}

\addplot [mark=o, mark size=2, mark options={solid}, thick, dashed, cyan, mark repeat=1]
coordinates {
( 0.00000000e+00,  2.54616700e-04)
( 1.42825755e+02,  2.90695596e-05)
( 1.84723594e+03,  1.66885503e-05)
( 2.68786462e+03,  4.58074502e-05)
( 3.59113612e+03,  7.21979122e-06)
( 5.43022353e+03,  7.05866858e-06)
( 5.58519279e+03,  6.06996049e-06)
( 6.09957146e+03,  5.20580351e-06)
( 6.53062721e+03,  4.46960902e-06)
( 7.14016746e+03,  7.16576586e-07)
( 7.28079319e+03,  9.41399180e-08)};\label{line:cp25f50}

\addplot [mark=o, mark size=2, mark options={solid}, thick, dotted, purple, mark repeat=1]
coordinates {
( 0.00000000e+00,  2.54616700e-04)
( 1.42487317e+02,  2.90695596e-05)
( 1.84681471e+03,  1.66885503e-05)
( 2.68683970e+03,  4.58074502e-05)
( 3.58936366e+03,  7.21979122e-06)
( 5.43001620e+03,  7.05866858e-06)
( 5.58957634e+03,  5.46685493e-06)
( 6.01056160e+03,  3.99653419e-06)
( 6.44145278e+03,  3.14113655e-06)
( 7.02723558e+03,  1.68890161e-07)};\label{line:cp25f100}

\addplot [mark=diamond, mark size=2, mark options={solid}, thick, solid, orange, mark repeat=1]
coordinates {
( 0.00000000e+00,  2.54616700e-04)
( 1.42967969e+02,  2.85334267e-05)
( 8.73039104e+02,  2.35632431e-05)
( 9.12233386e+03,  7.04043256e-06)
( 1.03282738e+04,  7.03661031e-06)
( 1.11045941e+04,  6.93976985e-06)
( 1.13602848e+04,  6.60806286e-06)
( 1.18212512e+04,  6.15312722e-06)
( 1.23344962e+04,  5.75615798e-06)
( 1.29626104e+04,  5.43357780e-06)
( 1.31273082e+04,  3.79890006e-06)
( 1.37385700e+04,  2.35972937e-06)
( 1.43533638e+04,  1.46805038e-06)
( 1.44893622e+04,  4.98014999e-07)
( 1.55097409e+04,  4.55950894e-08)};\label{line:cp50f10}

\addplot [mark=diamond, mark size=2, mark options={solid}, thick, dashed, magenta, mark repeat=1]
coordinates {
( 0.00000000e+00,  2.54616700e-04)
( 1.43394710e+02,  2.85334267e-05)
( 8.76764153e+02,  2.35632431e-05)
( 9.13357709e+03,  7.04043256e-06)
( 1.03384998e+04,  7.03661031e-06)
( 1.11152419e+04,  6.93976985e-06)
( 1.12778003e+04,  5.32590141e-06)
( 1.17028287e+04,  3.92259433e-06)
( 1.21157427e+04,  3.17072404e-06)
( 1.34795779e+04,  2.97627530e-07)};\label{line:cp50f50}

\addplot [mark=diamond, mark size=2, mark options={solid}, thick, dotted, black, mark repeat=1]
coordinates {
( 0.00000000e+00,  2.54616700e-04)
( 1.42910475e+02,  2.85334267e-05)
( 8.71232093e+02,  2.35632431e-05)
( 9.10954899e+03,  7.04043256e-06)
( 1.03192687e+04,  7.03661031e-06)
( 1.10974223e+04,  6.93976985e-06)
( 1.12433373e+04,  4.72166088e-06)
( 1.16651695e+04,  2.58922673e-06)
( 1.20989925e+04,  1.70091554e-06)
( 1.26261195e+04,  5.44778429e-08)};\label{line:cp50f100}

\end{axis}
\end{tikzpicture}
    \end{minipage}
    \caption{Convergence history (only major iterations shown) of the objective value (\textit{left}) and constraint violation (\textit{right}) of the EQP/TR method applied the inverse design problem (\ref{eqn:contin_inv}) for combinations of $(\tau_0, a)$. Legend: 
    $(\tau_0=10,a=10)$ (\ref{line:fp10f10}), 
    $(\tau_0=10,a=50)$ (\ref{line:fp10f50}),
    $(\tau_0=10,a=100)$ (\ref{line:fp10f100}),
    $(\tau_0=25,a=10)$ (\ref{line:fp25f10}),
    $(\tau_0=25,a=50)$ (\ref{line:fp25f50}),
    $(\tau_0=25,a=100)$ (\ref{line:fp25f100}),
    $(\tau_0=50,a=10)$ (\ref{line:fp50f10}),
    $(\tau_0=50,a=50)$ (\ref{line:fp50f50}),
    $(\tau_0=50,a=100)$ (\ref{line:fp50f100}),
    and HDM solution (\ref{line:fhdm}).
    }
    \label{fig:parm_fact_study}
\end{figure}
\subsubsection{Influence of inheriting the snapshots from the previous AL iteration}
\label{sec:snap}
Next, we explore the influence of inheriting the snapshots from the previous AL iteration. Algorithm~\ref{alg:auglag} is designed to reset the trust-region solver with each new major iteration, denoted as $i$. This structure offers a unique opportunity to expedite the convergence process by integrating snapshots from the previous iteration, $i-1$, into the current trust-region cycle because the reduced basis $\Phibold_k$ will be enriched. In this study, we choose $[0, 5, 15, 20]$ snapshots, for each primal and dual solution collections, from $(i-1)$th iteration.
Figure~\ref{fig:snapsz_study} shows unexpected results that the strategy of enriching $\Phibold_k$ with snapshots from previous iterations does not yield the anticipated benefits for problem because the additional snapshots increase the size of $\Phibold_k$, which increases the cost of the EQP weight calculation and each ROM query without substantially reducing the number of iterations required for convergence.
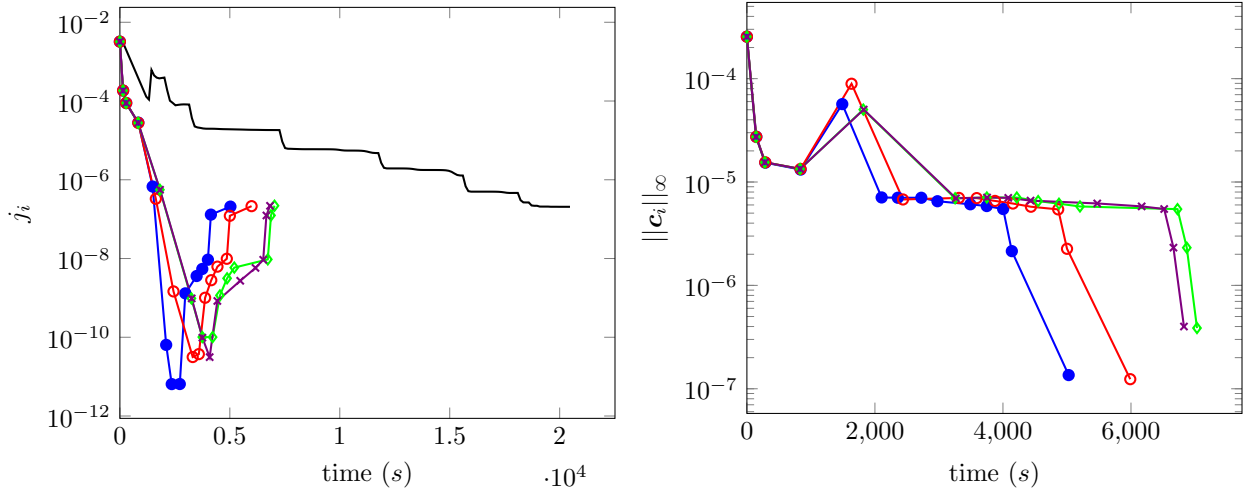
\begin{figure}
    \centering
    \begin{minipage}[t]{0.49\textwidth}
        \centering
        \tikzset{every picture/.style={scale=0.98}}
        \begin{tikzpicture}
\begin{axis}[
xmin=0,
ymode=log,
xlabel={time $(s)$},
ylabel={$j_i$}]
\addplot [mark=*, mark size=2, mark options={solid}, thick, solid, blue, mark repeat=1]
coordinates {
( 0.00000000e+00,  3.22557827e-03)
( 1.47369189e+02,  1.84513119e-04)
( 2.87925775e+02,  8.89305541e-05)
( 8.38970946e+02,  2.82500267e-05)
( 1.48994787e+03,  6.72525521e-07)
( 2.10484250e+03,  6.40094442e-11)
( 2.35755089e+03,  6.45858760e-12)
( 2.72453808e+03,  6.48917813e-12)
( 2.97720871e+03,  1.30282149e-09)
( 3.49007434e+03,  3.59256693e-09)
( 3.74687223e+03,  5.42489470e-09)
( 4.00308884e+03,  9.32367131e-09)
( 4.13878587e+03,  1.30193105e-07)
( 5.02667745e+03,  2.08142929e-07)};\label{line:fsz0}

\addplot [mark=o, mark size=2, mark options={solid}, thick, solid, red, mark repeat=1]
coordinates {
( 0.00000000e+00,  3.22557827e-03)
( 1.46321237e+02,  1.84513119e-04)
( 2.86339219e+02,  8.91544299e-05)
( 8.38277865e+02,  2.78085456e-05)
( 1.63471072e+03,  3.28727645e-07)
( 2.43823247e+03,  1.45990710e-09)
( 3.31212104e+03,  3.14662121e-11)
( 3.59033168e+03,  3.72860164e-11)
( 3.87532244e+03,  1.00790500e-09)
( 4.15532890e+03,  2.81509625e-09)
( 4.43730906e+03,  6.21632826e-09)
( 4.86277361e+03,  9.84815653e-09)
( 5.00096138e+03,  1.21616627e-07)
( 5.98888312e+03,  2.13623697e-07)};\label{line:fsz5}

\addplot [mark=diamond, mark size=2, mark options={solid}, thick, solid, green, mark repeat=1]
coordinates {
( 0.00000000e+00,  3.22557827e-03)
( 1.42879364e+02,  1.84513119e-04)
( 2.80377146e+02,  8.91544299e-05)
( 8.30673215e+02,  2.78085456e-05)
( 1.81849819e+03,  5.57476470e-07)
( 3.25254826e+03,  9.70123416e-10)
( 3.74853558e+03,  1.01131216e-10)
( 4.21654705e+03,  1.00784551e-10)
( 4.54995675e+03,  1.15549023e-09)
( 4.87509364e+03,  3.13995159e-09)
( 5.20331692e+03,  5.87990890e-09)
( 6.72671101e+03,  9.49313260e-09)
( 6.87181225e+03,  1.24377485e-07)
( 7.03128677e+03,  2.19291026e-07)};\label{line:fsz15}

\addplot [mark=x, mark size=2, mark options={solid}, thick, solid, violet, mark repeat=1]
coordinates {
( 0.00000000e+00,  3.22557827e-03)
( 1.42859614e+02,  1.84513119e-04)
( 2.81271223e+02,  8.91544299e-05)
( 8.32085241e+02,  2.78085456e-05)
( 1.82496697e+03,  5.57476470e-07)
( 3.26602134e+03,  9.70123416e-10)
( 3.75169423e+03,  9.77983684e-11)
( 4.08279123e+03,  3.12665622e-11)
( 4.43127525e+03,  8.35270851e-10)
( 5.47287503e+03,  2.72020070e-09)
( 6.16688517e+03,  5.77678583e-09)
( 6.51848638e+03,  9.29258627e-09)
( 6.66332658e+03,  1.23916272e-07)
( 6.82749195e+03,  2.17821308e-07)};\label{line:fsz20}

\addplot [thick, solid, black]
coordinates {
( 1.22896869e+02,  3.22557827e-03)
( 1.06799565e+03,  1.90546969e-04)
( 1.19932398e+03,  1.28867972e-04)
( 1.32098372e+03,  1.09043901e-04)
( 1.43812967e+03,  6.12303064e-04)
( 1.55579032e+03,  4.47296750e-04)
( 1.67342566e+03,  3.92956542e-04)
( 1.79059177e+03,  3.76058934e-04)
( 1.90739485e+03,  3.84816272e-04)
( 2.02423345e+03,  3.94022032e-04)
( 2.15015503e+03,  1.89018700e-04)
( 2.27590696e+03,  1.01839033e-04)
( 2.40070099e+03,  8.96532613e-05)
( 2.52550812e+03,  7.79420657e-05)
( 2.65012963e+03,  7.98990001e-05)
( 2.77488475e+03,  8.14942580e-05)
( 2.89962137e+03,  8.21747456e-05)
( 3.02436618e+03,  8.18450639e-05)
( 3.14893007e+03,  8.13326522e-05)
( 3.27364043e+03,  3.70506303e-05)
( 3.40191371e+03,  2.26126294e-05)
( 3.53029760e+03,  2.13233091e-05)
( 3.65853409e+03,  2.07631477e-05)
( 3.78714573e+03,  2.00824867e-05)
( 3.91624603e+03,  1.98057092e-05)
( 4.04485389e+03,  1.97751158e-05)
( 4.17345000e+03,  1.97761299e-05)
( 4.30196766e+03,  1.97188580e-05)
( 4.43051719e+03,  1.95842981e-05)
( 4.55902992e+03,  1.94203818e-05)
( 4.68753402e+03,  1.93061270e-05)
( 4.81598306e+03,  1.92423801e-05)
( 4.94449532e+03,  1.91806371e-05)
( 5.07304974e+03,  1.91092850e-05)
( 5.20145196e+03,  1.90348185e-05)
( 5.32982366e+03,  1.89653850e-05)
( 5.45817850e+03,  1.88984172e-05)
( 5.58659537e+03,  1.88221660e-05)
( 5.71509484e+03,  1.87292359e-05)
( 5.84340581e+03,  1.86381189e-05)
( 5.97170617e+03,  1.85763509e-05)
( 6.10027403e+03,  1.85407756e-05)
( 6.22872834e+03,  1.85039044e-05)
( 6.35721110e+03,  1.84536106e-05)
( 6.48555724e+03,  1.83975264e-05)
( 6.61402312e+03,  1.83540647e-05)
( 6.74229497e+03,  1.83318501e-05)
( 6.87075389e+03,  1.83187152e-05)
( 6.99912469e+03,  1.83040206e-05)
( 7.12754256e+03,  1.82898208e-05)
( 7.25607015e+03,  1.82818175e-05)
( 7.38446190e+03,  9.32098228e-06)
( 7.51286242e+03,  6.37991159e-06)
( 7.64128943e+03,  6.13775475e-06)
( 7.76971719e+03,  6.11444171e-06)
( 7.89819098e+03,  6.07066532e-06)
( 8.02663548e+03,  6.04111352e-06)
( 8.15517042e+03,  6.00590301e-06)
( 8.28362833e+03,  5.98041099e-06)
( 8.41209695e+03,  5.97177719e-06)
( 8.54052767e+03,  5.97139362e-06)
( 8.66903643e+03,  5.97127575e-06)
( 8.79734534e+03,  5.96994879e-06)
( 8.92585742e+03,  5.96648286e-06)
( 9.05433894e+03,  5.96250257e-06)
( 9.18277367e+03,  5.95947355e-06)
( 9.31124429e+03,  5.95669307e-06)
( 9.43972234e+03,  5.95185346e-06)
( 9.56822342e+03,  5.94173936e-06)
( 9.69669575e+03,  5.91824248e-06)
( 9.82517140e+03,  5.86526970e-06)
( 9.95378089e+03,  5.76293262e-06)
( 1.00821964e+04,  5.62778111e-06)
( 1.02107705e+04,  5.53847387e-06)
( 1.03392187e+04,  5.51585165e-06)
( 1.04677542e+04,  5.51291692e-06)
( 1.05962014e+04,  5.51108248e-06)
( 1.07247322e+04,  5.50463262e-06)
( 1.08532869e+04,  5.48991815e-06)
( 1.09818725e+04,  5.45197619e-06)
( 1.11103119e+04,  5.36725677e-06)
( 1.12387809e+04,  5.20558513e-06)
( 1.13672436e+04,  4.99203282e-06)
( 1.14957869e+04,  4.83205107e-06)
( 1.16243141e+04,  4.77476314e-06)
( 1.17528924e+04,  4.76584050e-06)
( 1.18813664e+04,  2.70149687e-06)
( 1.20097777e+04,  2.01025378e-06)
( 1.21383105e+04,  1.94281879e-06)
( 1.22667609e+04,  1.94105764e-06)
( 1.23951525e+04,  1.94064423e-06)
( 1.25235539e+04,  1.94060670e-06)
( 1.26520451e+04,  1.94045836e-06)
( 1.27803694e+04,  1.93980548e-06)
( 1.29090158e+04,  1.93765860e-06)
( 1.30374128e+04,  1.93182884e-06)
( 1.31658461e+04,  1.91743699e-06)
( 1.32944803e+04,  1.88736905e-06)
( 1.34184993e+04,  1.84127143e-06)
( 1.35422898e+04,  1.79908794e-06)
( 1.36709578e+04,  1.77843070e-06)
( 1.37993906e+04,  1.77267107e-06)
( 1.39278250e+04,  1.77157791e-06)
( 1.40561505e+04,  1.77096941e-06)
( 1.41845574e+04,  1.77028186e-06)
( 1.43129085e+04,  1.76949190e-06)
( 1.44412862e+04,  1.76805255e-06)
( 1.45695984e+04,  1.76480416e-06)
( 1.46980009e+04,  1.75672924e-06)
( 1.48264080e+04,  1.73666511e-06)
( 1.49549191e+04,  1.68920763e-06)
( 1.50833869e+04,  1.59143414e-06)
( 1.52121715e+04,  1.44350067e-06)
( 1.53405770e+04,  1.32232732e-06)
( 1.54693184e+04,  1.28537112e-06)
( 1.55975756e+04,  1.28088636e-06)
( 1.57221539e+04,  6.83353405e-07)
( 1.58469055e+04,  5.13488366e-07)
( 1.59717817e+04,  4.99008537e-07)
( 1.60965048e+04,  4.98828404e-07)
( 1.62212168e+04,  4.98851118e-07)
( 1.63460100e+04,  4.98860162e-07)
( 1.64707656e+04,  4.98858233e-07)
( 1.65955451e+04,  4.98838987e-07)
( 1.67202874e+04,  4.98780515e-07)
( 1.68451763e+04,  4.98627699e-07)
( 1.69723014e+04,  4.98227827e-07)
( 1.70972224e+04,  4.97206058e-07)
( 1.72221524e+04,  4.94681163e-07)
( 1.73470224e+04,  4.88993219e-07)
( 1.74719131e+04,  4.78703221e-07)
( 1.75967376e+04,  4.67170930e-07)
( 1.77215688e+04,  4.62137871e-07)
( 1.78463554e+04,  4.62507994e-07)
( 1.79712058e+04,  4.63122300e-07)
( 1.80961199e+04,  4.63251576e-07)
( 1.82209204e+04,  3.02992888e-07)
( 1.83457442e+04,  2.66665103e-07)
( 1.84706531e+04,  2.65321851e-07)
( 1.85954944e+04,  2.65318132e-07)
( 1.87203265e+04,  2.22583020e-07)
( 1.88452312e+04,  2.18314940e-07)
( 1.89702388e+04,  2.18300854e-07)
( 1.90951339e+04,  2.08514943e-07)
( 1.92200679e+04,  2.08265910e-07)
( 1.93448846e+04,  2.08265929e-07)
( 1.94697814e+04,  2.08265943e-07)
( 1.95946793e+04,  2.06237591e-07)
( 1.97195555e+04,  2.06226558e-07)
( 1.98443776e+04,  2.06226560e-07)
( 1.99691652e+04,  2.05817745e-07)
( 2.00940032e+04,  2.05817308e-07)
( 2.02188574e+04,  2.05817308e-07)
( 2.03439381e+04,  2.05735419e-07)
( 2.04688459e+04,  2.05735402e-07)};\label{line:fhdm}

\end{axis}
\end{tikzpicture}
    \end{minipage}
    \hfill
    \begin{minipage}[t]{0.49\textwidth}
        \centering
        \tikzset{every picture/.style={scale=0.98}}
        \begin{tikzpicture}
\begin{axis}[
xmin=0,
ymode=log,
xlabel={time $(s)$},
ylabel={$||\cbm_i||_\infty$}]
\addplot [mark=*, mark size=2, mark options={solid}, thick, solid, blue, mark repeat=1]
coordinates {
( 0.00000000e+00,  2.54616700e-04)
( 1.47369189e+02,  2.73712524e-05)
( 2.87925775e+02,  1.53848250e-05)
( 8.38970946e+02,  1.31854998e-05)
( 1.48994787e+03,  5.67032935e-05)
( 2.10484250e+03,  7.09083663e-06)
( 2.35755089e+03,  7.04327490e-06)
( 2.72453808e+03,  7.04244130e-06)
( 2.97720871e+03,  6.51871365e-06)
( 3.49007434e+03,  6.08827942e-06)
( 3.74687223e+03,  5.85416924e-06)
( 4.00308884e+03,  5.47921035e-06)
( 4.13878587e+03,  2.14456124e-06)
( 5.02667745e+03,  1.35636628e-07)};\label{line:csz0}

\addplot [mark=o, mark size=2, mark options={solid}, thick, solid, red, mark repeat=1]
coordinates {
( 0.00000000e+00,  2.54616700e-04)
( 1.46321237e+02,  2.73712524e-05)
( 2.86339219e+02,  1.55050204e-05)
( 8.38277865e+02,  1.33763652e-05)
( 1.63471072e+03,  8.92497861e-05)
( 2.43823247e+03,  6.78041435e-06)
( 3.31212104e+03,  7.01617786e-06)
( 3.59033168e+03,  6.98668207e-06)
( 3.87532244e+03,  6.56295602e-06)
( 4.15532890e+03,  6.19295829e-06)
( 4.43730906e+03,  5.76732365e-06)
( 4.86277361e+03,  5.43447343e-06)
( 5.00096138e+03,  2.25908295e-06)
( 5.98888312e+03,  1.23794880e-07)};\label{line:csz5}

\addplot [mark=diamond, mark size=2, mark options={solid}, thick, solid, green, mark repeat=1]
coordinates {
( 0.00000000e+00,  2.54616700e-04)
( 1.42879364e+02,  2.73712524e-05)
( 2.80377146e+02,  1.55050204e-05)
( 8.30673215e+02,  1.33763652e-05)
( 1.81849819e+03,  5.02012712e-05)
( 3.25254826e+03,  6.95953205e-06)
( 3.74853558e+03,  7.07908915e-06)
( 4.21654705e+03,  7.07506959e-06)
( 4.54995675e+03,  6.53969654e-06)
( 4.87509364e+03,  6.16472952e-06)
( 5.20331692e+03,  5.80199107e-06)
( 6.72671101e+03,  5.46468806e-06)
( 6.87181225e+03,  2.31499173e-06)
( 7.03128677e+03,  3.86697042e-07)};\label{line:csz15}

\addplot [mark=x, mark size=2, mark options={solid}, thick, solid, violet, mark repeat=1]
coordinates {
( 0.00000000e+00,  2.54616700e-04)
( 1.42859614e+02,  2.73712524e-05)
( 2.81271223e+02,  1.55050204e-05)
( 8.32085241e+02,  1.33763652e-05)
( 1.82496697e+03,  5.02012712e-05)
( 3.26602134e+03,  6.95953205e-06)
( 3.75169423e+03,  7.08511773e-06)
( 4.08279123e+03,  7.02532184e-06)
( 4.43127525e+03,  6.60881589e-06)
( 5.47287503e+03,  6.20293554e-06)
( 6.16688517e+03,  5.81060419e-06)
( 6.51848638e+03,  5.47981300e-06)
( 6.66332658e+03,  2.31831061e-06)
( 6.82749195e+03,  3.99974380e-07)};\label{line:csz20}

\end{axis}
\end{tikzpicture}
    \end{minipage}
    \caption{Convergence history (only major iterations shown) of the objective value (\textit{left}) and constraint violation (\textit{right}) of the EQP/TR method applied the inverse design problem (\ref{eqn:contin_inv}) when inheriting a different number of snapshots from the previous AL iteration. Legend: 
    0 snapshots (\ref{line:fsz0}), 
    5 snapshots (\ref{line:fsz5}), 
    15 snapshots (\ref{line:fsz15}),
    20 snapshots (\ref{line:fsz20}),
    and HDM solution (\ref{line:fhdm}).
    }
    \label{fig:snapsz_study}
\end{figure}
\subsubsection{Performance}
Next, we study the overall performance of the EQP/TR method with frozen parameters based on the studies in Sections~\ref{sec:param}-\ref{sec:snap}, i.e., we set $\tau_0 = 10$, $a=50$ and do not use snapshot from previous AL iterations. Table~\ref{tab:inv1_hist} presents the convergence history of the proposed method at selected iterations. The method is converging to a feasible point with $\norm{\chi_i}_\infty = 5.7\times 10^{-6}$ in 13 major (AL) iterations. At each major iteration, the value of objective function $j_i$ is not necessarily monotonically decreasing due to the trade-off between objective value reduction and constraint violation. In early iterations, we observe rapid reduction in $j_i$ with only minor reduction in the constraint violation. At the beginning of each major iteration, the algorithm reconstructs the reduced basis $\Phibold_0$ and resets the element usage, which leads to low element usage (approx. $15\%$ of entire mesh) in early trust-region iterations. Table~\ref{tab:inv1_hist} also shows reasonable maximum element usage below $21\%$. Due to the small reduced basis and relatively low element usage, the trust-region subproblems account for a small portion of the overall cost; the HDM evaluations at the end of each trust-region iteration (\ref{eqn:tr:redu_ratio}) dominate the computational costs. The element usage in the side-constrained setting is larger than the element usage in the unconstrained setting of \cite{wen_globally_2023} because of the additional EQP constraints required by the AL framework.
\begin{table}
\centering
\begin{tabular}{c c c c c c c} 
    \hline
    Iteration & $j_i$ & $||\cbm_i||_\infty$ & $||\chi_i||_\infty$ & $S_i$ & $||\rhobold_i||_{0,\mathrm{min}}$ $(\%)$& $||\rhobold_i||_{0,\mathrm{max}}$ $(\%)$\\ [0.5ex]
    \hline
    1 & 1.8451e-04 & 2.7371e-05 & 5.6559e-03 & 5.7139e-02 & 15.82 & 15.82 \\
2 & 8.8931e-05 & 1.5385e-05 & 3.5909e-03 & 2.7507e-02 & 16.40 & 16.40 \\
3 & 2.8250e-05 & 1.3185e-05 & 8.4954e-04 & 8.6943e-03 & 16.11 & 19.94 \\
4 & 6.7253e-07 & 5.6703e-05 & 3.6108e-05 & 1.4472e-04 & 15.82 & 20.63 \\
5 & 6.4009e-11 & 7.0908e-06 & 2.3744e-06 & 6.3763e-05 & 16.40 & 19.25 \\
10 & 5.4249e-09 & 5.8542e-06 & 5.6429e-06 & 6.2101e-05 & 14.54 & 14.54 \\
13 & 2.0814e-07 & 1.3564e-07 & 5.7409e-06 & 7.4639e-07 & 15.23 & 19.16 \\
\hline 

\end{tabular}
\caption{Convergence history of the EQP/TR method applied to the inverse design problem in (\ref{eqn:contin_inv}), where $||\rhobold_i||_{0,\mathrm{min}}$ ($||\rhobold_i||_{0,\mathrm{max}}$) is the fewest (most) nonzero EQP weights across all trust-region iterations at major iteration $i$.}
\label{tab:inv1_hist}
\end{table}

Finally, we compare the performance of the HDM, ROM, and EQP methods in terms of the total computational cost required to satisfy $S_i < \epsilon$ for $\epsilon \in \{10^{-3}, 10^{-4}, 10^{-6}\}$ and various constraint violation tolerances. Figure~\ref{fig:inv1_final_hist} shows the convergence history of the objective value $j_i$ and constraint violation $\norm{\cbm_i}_\infty$. Both the ROM/TR and EQP/TR methods converge much faster than the HDM-based method with the EQP/TR exhibiting the fastest convergence. The ROM/TR and EQP/TR methods prioritize objective reductions over constraint violation reduction in early iterations, which leads better objective values with loose constraint violation tolerances. Tables~\ref{tab:inv1:con_1e4} -~\ref{tab:inv1:con_1e6} present the performance of each method for three constraint violation tolerances ($10^{-4}$, $10^{-5}$, $10^{-6}$), respectively. Overall, we observe a speedup of the ROM/TR method (relative to the HDM-based method) between $1.3$ and $9.8$, and a speedup of the EQP/TR method between $2.4$ and $8.0$. Tighter constraint violation tolerances reduce the overall speedup of the ROM/TR and EQP/TR methods because additional AL iterations with larger penalty parameter are required. Larger penalty parameter leads to additional trust-region iterations, which in turn leads to a larger reduced basis and more expensive EQP training. The initial and optimal shape produced by each method considered are shown in Figure~\ref{fig:inv1:optimal_soln}.
\begin{figure}
    \centering
    \begin{minipage}[t]{0.49\textwidth}
        \centering
        \tikzset{every picture/.style={scale=0.98}}
        \begin{tikzpicture}
\begin{axis}[
xmin=0,
ymode=log,
xlabel={time $(s)$},
ylabel={$j_i$}]
\addplot [mark=*, mark size=2, mark options={solid}, thick, solid, blue, mark repeat=1]
coordinates {
( 0.00000000e+00,  3.22557827e-03)
( 1.47369189e+02,  1.84513119e-04)
( 2.87925775e+02,  8.89305541e-05)
( 8.38970946e+02,  2.82500267e-05)
( 1.48994787e+03,  6.72525521e-07)
( 2.10484250e+03,  6.40094442e-11)
( 2.35755089e+03,  6.45858760e-12)
( 2.72453808e+03,  6.48917813e-12)
( 2.97720871e+03,  1.30282149e-09)
( 3.49007434e+03,  3.59256693e-09)
( 3.74687223e+03,  5.42489470e-09)
( 4.00308884e+03,  9.32367131e-09)
( 4.13878587e+03,  1.30193105e-07)
( 5.02667745e+03,  2.08142929e-07)};\label{line:feqp}

\addplot [mark=o, mark size=2, mark options={solid}, thick, dashed, red, mark repeat=1]
coordinates {
( 0.00000000e+00,  3.22557827e-03)
( 1.55085087e+02,  1.81478846e-04)
( 3.01007928e+02,  3.35389161e-05)
( 4.46587823e+02,  2.65888646e-05)
( 1.62712662e+03,  1.42694492e-07)
( 3.46455895e+03,  2.89171265e-11)
( 4.11507004e+03,  8.30252730e-12)
( 4.76437250e+03,  6.19468777e-12)
( 5.04230921e+03,  8.39195064e-10)
( 5.69235436e+03,  2.94533122e-09)
( 6.34200166e+03,  5.85110830e-09)
( 7.36254789e+03,  9.81966315e-09)
( 7.52111548e+03,  1.32457320e-07)
( 9.04514515e+03,  2.04686549e-07)};\label{line:from}

\addplot [thick, solid, black]
coordinates {
( 1.22896869e+02,  3.22557827e-03)
( 1.06799565e+03,  1.90546969e-04)
( 1.19932398e+03,  1.28867972e-04)
( 1.32098372e+03,  1.09043901e-04)
( 1.43812967e+03,  6.12303064e-04)
( 1.55579032e+03,  4.47296750e-04)
( 1.67342566e+03,  3.92956542e-04)
( 1.79059177e+03,  3.76058934e-04)
( 1.90739485e+03,  3.84816272e-04)
( 2.02423345e+03,  3.94022032e-04)
( 2.15015503e+03,  1.89018700e-04)
( 2.27590696e+03,  1.01839033e-04)
( 2.40070099e+03,  8.96532613e-05)
( 2.52550812e+03,  7.79420657e-05)
( 2.65012963e+03,  7.98990001e-05)
( 2.77488475e+03,  8.14942580e-05)
( 2.89962137e+03,  8.21747456e-05)
( 3.02436618e+03,  8.18450639e-05)
( 3.14893007e+03,  8.13326522e-05)
( 3.27364043e+03,  3.70506303e-05)
( 3.40191371e+03,  2.26126294e-05)
( 3.53029760e+03,  2.13233091e-05)
( 3.65853409e+03,  2.07631477e-05)
( 3.78714573e+03,  2.00824867e-05)
( 3.91624603e+03,  1.98057092e-05)
( 4.04485389e+03,  1.97751158e-05)
( 4.17345000e+03,  1.97761299e-05)
( 4.30196766e+03,  1.97188580e-05)
( 4.43051719e+03,  1.95842981e-05)
( 4.55902992e+03,  1.94203818e-05)
( 4.68753402e+03,  1.93061270e-05)
( 4.81598306e+03,  1.92423801e-05)
( 4.94449532e+03,  1.91806371e-05)
( 5.07304974e+03,  1.91092850e-05)
( 5.20145196e+03,  1.90348185e-05)
( 5.32982366e+03,  1.89653850e-05)
( 5.45817850e+03,  1.88984172e-05)
( 5.58659537e+03,  1.88221660e-05)
( 5.71509484e+03,  1.87292359e-05)
( 5.84340581e+03,  1.86381189e-05)
( 5.97170617e+03,  1.85763509e-05)
( 6.10027403e+03,  1.85407756e-05)
( 6.22872834e+03,  1.85039044e-05)
( 6.35721110e+03,  1.84536106e-05)
( 6.48555724e+03,  1.83975264e-05)
( 6.61402312e+03,  1.83540647e-05)
( 6.74229497e+03,  1.83318501e-05)
( 6.87075389e+03,  1.83187152e-05)
( 6.99912469e+03,  1.83040206e-05)
( 7.12754256e+03,  1.82898208e-05)
( 7.25607015e+03,  1.82818175e-05)
( 7.38446190e+03,  9.32098228e-06)
( 7.51286242e+03,  6.37991159e-06)
( 7.64128943e+03,  6.13775475e-06)
( 7.76971719e+03,  6.11444171e-06)
( 7.89819098e+03,  6.07066532e-06)
( 8.02663548e+03,  6.04111352e-06)
( 8.15517042e+03,  6.00590301e-06)
( 8.28362833e+03,  5.98041099e-06)
( 8.41209695e+03,  5.97177719e-06)
( 8.54052767e+03,  5.97139362e-06)
( 8.66903643e+03,  5.97127575e-06)
( 8.79734534e+03,  5.96994879e-06)
( 8.92585742e+03,  5.96648286e-06)
( 9.05433894e+03,  5.96250257e-06)
( 9.18277367e+03,  5.95947355e-06)
( 9.31124429e+03,  5.95669307e-06)
( 9.43972234e+03,  5.95185346e-06)
( 9.56822342e+03,  5.94173936e-06)
( 9.69669575e+03,  5.91824248e-06)
( 9.82517140e+03,  5.86526970e-06)
( 9.95378089e+03,  5.76293262e-06)
( 1.00821964e+04,  5.62778111e-06)
( 1.02107705e+04,  5.53847387e-06)
( 1.03392187e+04,  5.51585165e-06)
( 1.04677542e+04,  5.51291692e-06)
( 1.05962014e+04,  5.51108248e-06)
( 1.07247322e+04,  5.50463262e-06)
( 1.08532869e+04,  5.48991815e-06)
( 1.09818725e+04,  5.45197619e-06)
( 1.11103119e+04,  5.36725677e-06)
( 1.12387809e+04,  5.20558513e-06)
( 1.13672436e+04,  4.99203282e-06)
( 1.14957869e+04,  4.83205107e-06)
( 1.16243141e+04,  4.77476314e-06)
( 1.17528924e+04,  4.76584050e-06)
( 1.18813664e+04,  2.70149687e-06)
( 1.20097777e+04,  2.01025378e-06)
( 1.21383105e+04,  1.94281879e-06)
( 1.22667609e+04,  1.94105764e-06)
( 1.23951525e+04,  1.94064423e-06)
( 1.25235539e+04,  1.94060670e-06)
( 1.26520451e+04,  1.94045836e-06)
( 1.27803694e+04,  1.93980548e-06)
( 1.29090158e+04,  1.93765860e-06)
( 1.30374128e+04,  1.93182884e-06)
( 1.31658461e+04,  1.91743699e-06)
( 1.32944803e+04,  1.88736905e-06)
( 1.34184993e+04,  1.84127143e-06)
( 1.35422898e+04,  1.79908794e-06)
( 1.36709578e+04,  1.77843070e-06)
( 1.37993906e+04,  1.77267107e-06)
( 1.39278250e+04,  1.77157791e-06)
( 1.40561505e+04,  1.77096941e-06)
( 1.41845574e+04,  1.77028186e-06)
( 1.43129085e+04,  1.76949190e-06)
( 1.44412862e+04,  1.76805255e-06)
( 1.45695984e+04,  1.76480416e-06)
( 1.46980009e+04,  1.75672924e-06)
( 1.48264080e+04,  1.73666511e-06)
( 1.49549191e+04,  1.68920763e-06)
( 1.50833869e+04,  1.59143414e-06)
( 1.52121715e+04,  1.44350067e-06)
( 1.53405770e+04,  1.32232732e-06)
( 1.54693184e+04,  1.28537112e-06)
( 1.55975756e+04,  1.28088636e-06)
( 1.57221539e+04,  6.83353405e-07)
( 1.58469055e+04,  5.13488366e-07)
( 1.59717817e+04,  4.99008537e-07)
( 1.60965048e+04,  4.98828404e-07)
( 1.62212168e+04,  4.98851118e-07)
( 1.63460100e+04,  4.98860162e-07)
( 1.64707656e+04,  4.98858233e-07)
( 1.65955451e+04,  4.98838987e-07)
( 1.67202874e+04,  4.98780515e-07)
( 1.68451763e+04,  4.98627699e-07)
( 1.69723014e+04,  4.98227827e-07)
( 1.70972224e+04,  4.97206058e-07)
( 1.72221524e+04,  4.94681163e-07)
( 1.73470224e+04,  4.88993219e-07)
( 1.74719131e+04,  4.78703221e-07)
( 1.75967376e+04,  4.67170930e-07)
( 1.77215688e+04,  4.62137871e-07)
( 1.78463554e+04,  4.62507994e-07)
( 1.79712058e+04,  4.63122300e-07)
( 1.80961199e+04,  4.63251576e-07)
( 1.82209204e+04,  3.02992888e-07)
( 1.83457442e+04,  2.66665103e-07)
( 1.84706531e+04,  2.65321851e-07)
( 1.85954944e+04,  2.65318132e-07)
( 1.87203265e+04,  2.22583020e-07)
( 1.88452312e+04,  2.18314940e-07)
( 1.89702388e+04,  2.18300854e-07)
( 1.90951339e+04,  2.08514943e-07)
( 1.92200679e+04,  2.08265910e-07)
( 1.93448846e+04,  2.08265929e-07)
( 1.94697814e+04,  2.08265943e-07)
( 1.95946793e+04,  2.06237591e-07)
( 1.97195555e+04,  2.06226558e-07)
( 1.98443776e+04,  2.06226560e-07)
( 1.99691652e+04,  2.05817745e-07)
( 2.00940032e+04,  2.05817308e-07)
( 2.02188574e+04,  2.05817308e-07)
( 2.03439381e+04,  2.05735419e-07)
( 2.04688459e+04,  2.05735402e-07)};\label{line:fhdm}

\end{axis}
\end{tikzpicture}
    \end{minipage}
    \hfill
    \begin{minipage}[t]{0.49\textwidth}
        \centering
        \tikzset{every picture/.style={scale=0.98}}
        \begin{tikzpicture}
\begin{axis}[
xmin=0,
ymode=log,
xlabel={time $(s)$},
ylabel={$||\cbm_i||_\infty$}]
\addplot [mark=*, mark size=2, mark options={solid}, thick, solid, blue, mark repeat=1]
coordinates {
( 0.00000000e+00,  2.54616700e-04)
( 1.47369189e+02,  2.73712524e-05)
( 2.87925775e+02,  1.53848250e-05)
( 8.38970946e+02,  1.31854998e-05)
( 1.48994787e+03,  5.67032935e-05)
( 2.10484250e+03,  7.09083663e-06)
( 2.35755089e+03,  7.04327490e-06)
( 2.72453808e+03,  7.04244130e-06)
( 2.97720871e+03,  6.51871365e-06)
( 3.49007434e+03,  6.08827942e-06)
( 3.74687223e+03,  5.85416924e-06)
( 4.00308884e+03,  5.47921035e-06)
( 4.13878587e+03,  2.14456124e-06)
( 5.02667745e+03,  1.35636628e-07)};\label{line:ceqp}

\addplot [mark=o, mark size=2, mark options={solid}, thick, dashed, red, mark repeat=1]
coordinates {
( 0.00000000e+00,  2.54616700e-04)
( 1.55085087e+02,  2.45679873e-05)
( 3.01007928e+02,  9.27315261e-06)
( 4.46587823e+02,  1.07449936e-05)
( 1.62712662e+03,  4.93546237e-05)
( 3.46455895e+03,  7.08130954e-06)
( 4.11507004e+03,  7.04993098e-06)
( 4.76437250e+03,  7.04222540e-06)
( 5.04230921e+03,  6.60381779e-06)
( 5.69235436e+03,  6.16920708e-06)
( 6.34200166e+03,  5.80245090e-06)
( 7.36254789e+03,  5.43369718e-06)
( 7.52111548e+03,  2.08445188e-06)
( 9.04514515e+03,  1.66811241e-08)};\label{line:crom}

\addplot [thick, solid, black]
coordinates {
( 1.22896869e+02,  0.00000000e+00)
( 1.06799565e+03,  0.00000000e+00)
( 1.19932398e+03,  0.00000000e+00)
( 1.32098372e+03,  0.00000000e+00)
( 1.43812967e+03,  0.00000000e+00)
( 1.55579032e+03,  0.00000000e+00)
( 1.67342566e+03,  0.00000000e+00)
( 1.79059177e+03,  0.00000000e+00)
( 1.90739485e+03,  0.00000000e+00)
( 2.02423345e+03,  0.00000000e+00)
( 2.15015503e+03,  0.00000000e+00)
( 2.27590696e+03,  0.00000000e+00)
( 2.40070099e+03,  0.00000000e+00)
( 2.52550812e+03,  0.00000000e+00)
( 2.65012963e+03,  0.00000000e+00)
( 2.77488475e+03,  0.00000000e+00)
( 2.89962137e+03,  0.00000000e+00)
( 3.02436618e+03,  0.00000000e+00)
( 3.14893007e+03,  0.00000000e+00)
( 3.27364043e+03,  0.00000000e+00)
( 3.40191371e+03,  0.00000000e+00)
( 3.53029760e+03,  0.00000000e+00)
( 3.65853409e+03,  0.00000000e+00)
( 3.78714573e+03,  0.00000000e+00)
( 3.91624603e+03,  0.00000000e+00)
( 4.04485389e+03,  0.00000000e+00)
( 4.17345000e+03,  0.00000000e+00)
( 4.30196766e+03,  0.00000000e+00)
( 4.43051719e+03,  0.00000000e+00)
( 4.55902992e+03,  0.00000000e+00)
( 4.68753402e+03,  0.00000000e+00)
( 4.81598306e+03,  0.00000000e+00)
( 4.94449532e+03,  0.00000000e+00)
( 5.07304974e+03,  0.00000000e+00)
( 5.20145196e+03,  0.00000000e+00)
( 5.32982366e+03,  0.00000000e+00)
( 5.45817850e+03,  0.00000000e+00)
( 5.58659537e+03,  0.00000000e+00)
( 5.71509484e+03,  0.00000000e+00)
( 5.84340581e+03,  0.00000000e+00)
( 5.97170617e+03,  0.00000000e+00)
( 6.10027403e+03,  0.00000000e+00)
( 6.22872834e+03,  0.00000000e+00)
( 6.35721110e+03,  0.00000000e+00)
( 6.48555724e+03,  0.00000000e+00)
( 6.61402312e+03,  0.00000000e+00)
( 6.74229497e+03,  0.00000000e+00)
( 6.87075389e+03,  0.00000000e+00)
( 6.99912469e+03,  0.00000000e+00)
( 7.12754256e+03,  0.00000000e+00)
( 7.25607015e+03,  0.00000000e+00)
( 7.38446190e+03,  0.00000000e+00)
( 7.51286242e+03,  0.00000000e+00)
( 7.64128943e+03,  0.00000000e+00)
( 7.76971719e+03,  0.00000000e+00)
( 7.89819098e+03,  0.00000000e+00)
( 8.02663548e+03,  0.00000000e+00)
( 8.15517042e+03,  0.00000000e+00)
( 8.28362833e+03,  0.00000000e+00)
( 8.41209695e+03,  0.00000000e+00)
( 8.54052767e+03,  0.00000000e+00)
( 8.66903643e+03,  0.00000000e+00)
( 8.79734534e+03,  0.00000000e+00)
( 8.92585742e+03,  0.00000000e+00)
( 9.05433894e+03,  0.00000000e+00)
( 9.18277367e+03,  0.00000000e+00)
( 9.31124429e+03,  0.00000000e+00)
( 9.43972234e+03,  0.00000000e+00)
( 9.56822342e+03,  0.00000000e+00)
( 9.69669575e+03,  0.00000000e+00)
( 9.82517140e+03,  0.00000000e+00)
( 9.95378089e+03,  0.00000000e+00)
( 1.00821964e+04,  0.00000000e+00)
( 1.02107705e+04,  0.00000000e+00)
( 1.03392187e+04,  0.00000000e+00)
( 1.04677542e+04,  0.00000000e+00)
( 1.05962014e+04,  0.00000000e+00)
( 1.07247322e+04,  0.00000000e+00)
( 1.08532869e+04,  0.00000000e+00)
( 1.09818725e+04,  0.00000000e+00)
( 1.11103119e+04,  0.00000000e+00)
( 1.12387809e+04,  0.00000000e+00)
( 1.13672436e+04,  0.00000000e+00)
( 1.14957869e+04,  0.00000000e+00)
( 1.16243141e+04,  0.00000000e+00)
( 1.17528924e+04,  0.00000000e+00)
( 1.18813664e+04,  0.00000000e+00)
( 1.20097777e+04,  0.00000000e+00)
( 1.21383105e+04,  0.00000000e+00)
( 1.22667609e+04,  0.00000000e+00)
( 1.23951525e+04,  0.00000000e+00)
( 1.25235539e+04,  0.00000000e+00)
( 1.26520451e+04,  0.00000000e+00)
( 1.27803694e+04,  0.00000000e+00)
( 1.29090158e+04,  0.00000000e+00)
( 1.30374128e+04,  0.00000000e+00)
( 1.31658461e+04,  0.00000000e+00)
( 1.32944803e+04,  0.00000000e+00)
( 1.34184993e+04,  0.00000000e+00)
( 1.35422898e+04,  0.00000000e+00)
( 1.36709578e+04,  0.00000000e+00)
( 1.37993906e+04,  0.00000000e+00)
( 1.39278250e+04,  0.00000000e+00)
( 1.40561505e+04,  0.00000000e+00)
( 1.41845574e+04,  0.00000000e+00)
( 1.43129085e+04,  0.00000000e+00)
( 1.44412862e+04,  0.00000000e+00)
( 1.45695984e+04,  0.00000000e+00)
( 1.46980009e+04,  0.00000000e+00)
( 1.48264080e+04,  0.00000000e+00)
( 1.49549191e+04,  0.00000000e+00)
( 1.50833869e+04,  0.00000000e+00)
( 1.52121715e+04,  0.00000000e+00)
( 1.53405770e+04,  0.00000000e+00)
( 1.54693184e+04,  0.00000000e+00)
( 1.55975756e+04,  0.00000000e+00)
( 1.57221539e+04,  0.00000000e+00)
( 1.58469055e+04,  0.00000000e+00)
( 1.59717817e+04,  0.00000000e+00)
( 1.60965048e+04,  0.00000000e+00)
( 1.62212168e+04,  0.00000000e+00)
( 1.63460100e+04,  0.00000000e+00)
( 1.64707656e+04,  0.00000000e+00)
( 1.65955451e+04,  0.00000000e+00)
( 1.67202874e+04,  0.00000000e+00)
( 1.68451763e+04,  0.00000000e+00)
( 1.69723014e+04,  0.00000000e+00)
( 1.70972224e+04,  0.00000000e+00)
( 1.72221524e+04,  0.00000000e+00)
( 1.73470224e+04,  0.00000000e+00)
( 1.74719131e+04,  0.00000000e+00)
( 1.75967376e+04,  0.00000000e+00)
( 1.77215688e+04,  0.00000000e+00)
( 1.78463554e+04,  0.00000000e+00)
( 1.79712058e+04,  0.00000000e+00)
( 1.80961199e+04,  0.00000000e+00)
( 1.82209204e+04,  0.00000000e+00)
( 1.83457442e+04,  0.00000000e+00)
( 1.84706531e+04,  0.00000000e+00)
( 1.85954944e+04,  0.00000000e+00)
( 1.87203265e+04,  0.00000000e+00)
( 1.88452312e+04,  0.00000000e+00)
( 1.89702388e+04,  0.00000000e+00)
( 1.90951339e+04,  0.00000000e+00)
( 1.92200679e+04,  0.00000000e+00)
( 1.93448846e+04,  0.00000000e+00)
( 1.94697814e+04,  0.00000000e+00)
( 1.95946793e+04,  0.00000000e+00)
( 1.97195555e+04,  0.00000000e+00)
( 1.98443776e+04,  0.00000000e+00)
( 1.99691652e+04,  0.00000000e+00)
( 2.00940032e+04,  0.00000000e+00)
( 2.02188574e+04,  0.00000000e+00)
( 2.03439381e+04,  0.00000000e+00)
( 2.04688459e+04,  0.00000000e+00)};\label{line:chdm}

\end{axis}
\end{tikzpicture}
    \end{minipage}
    \caption{Convergence history (only major iterations shown) of the objective value (\textit{left}) and constraint violation (\textit{right}) of the HDM-based interior point solver (\ref{line:fhdm}), ROM/TR method (\ref{line:from}), and EQP/TR method (\ref{line:feqp}) applied to the inverse design problem in (\ref{eqn:contin_inv}).}
    \label{fig:inv1_final_hist}
\end{figure}
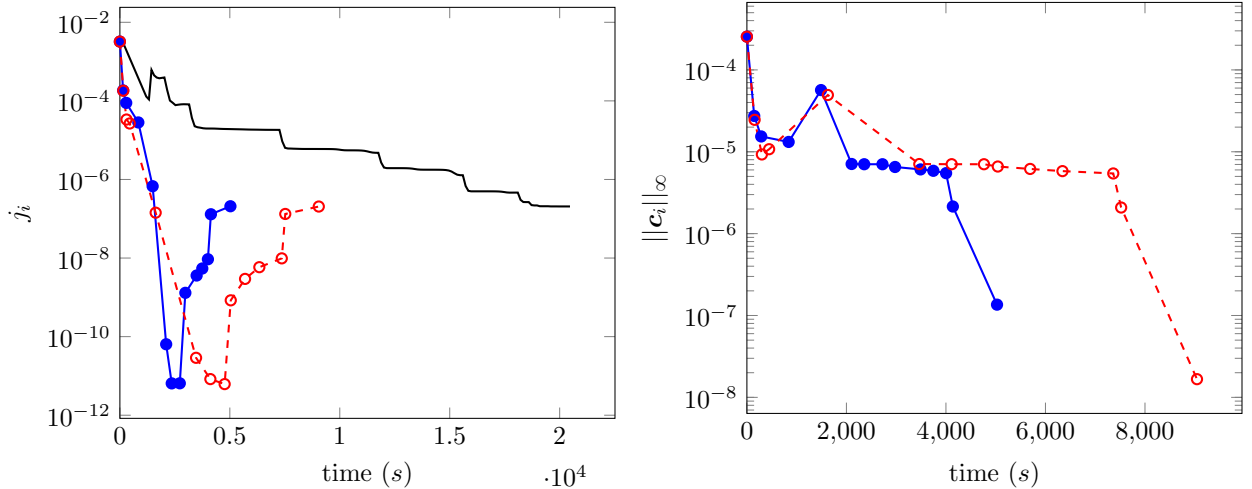

\begin{figure}
    \centering
    \begin{minipage}[t]{0.49\textwidth}
        \centering
        \includegraphics[width=\textwidth]{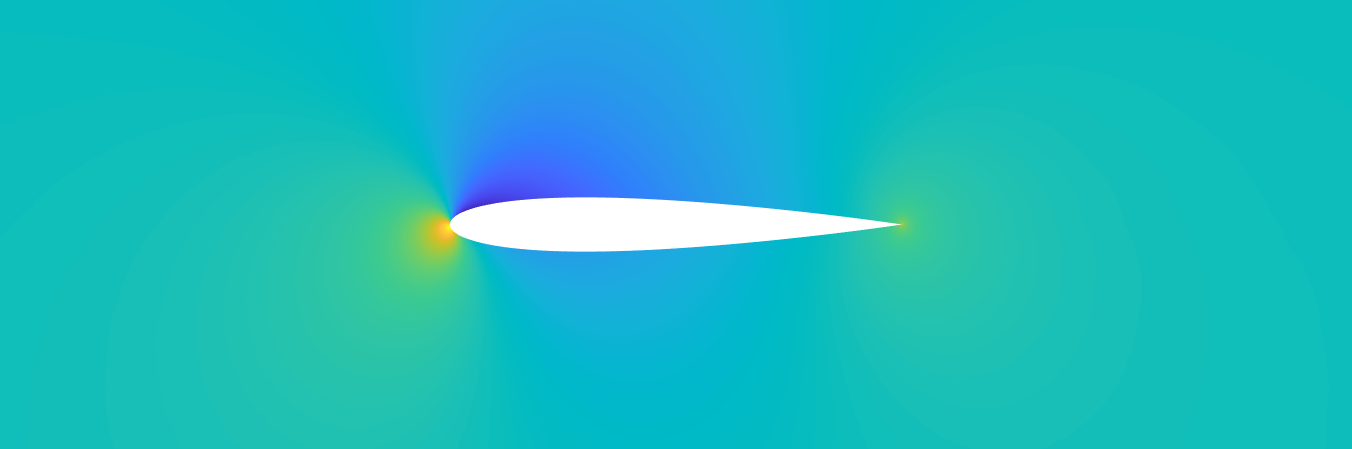}
    \end{minipage}
    \hfill
    \begin{minipage}[t]{0.49\textwidth}
        \centering
        \includegraphics[width=\textwidth]{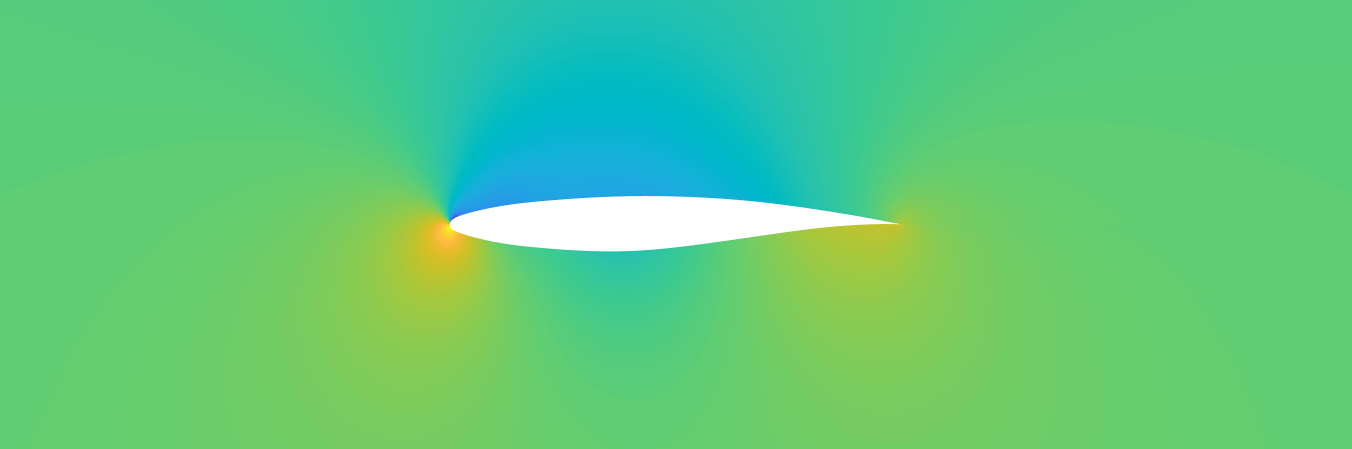}
    \end{minipage}
        \begin{minipage}[t]{0.49\textwidth}
        \centering
        \includegraphics[width=\textwidth]{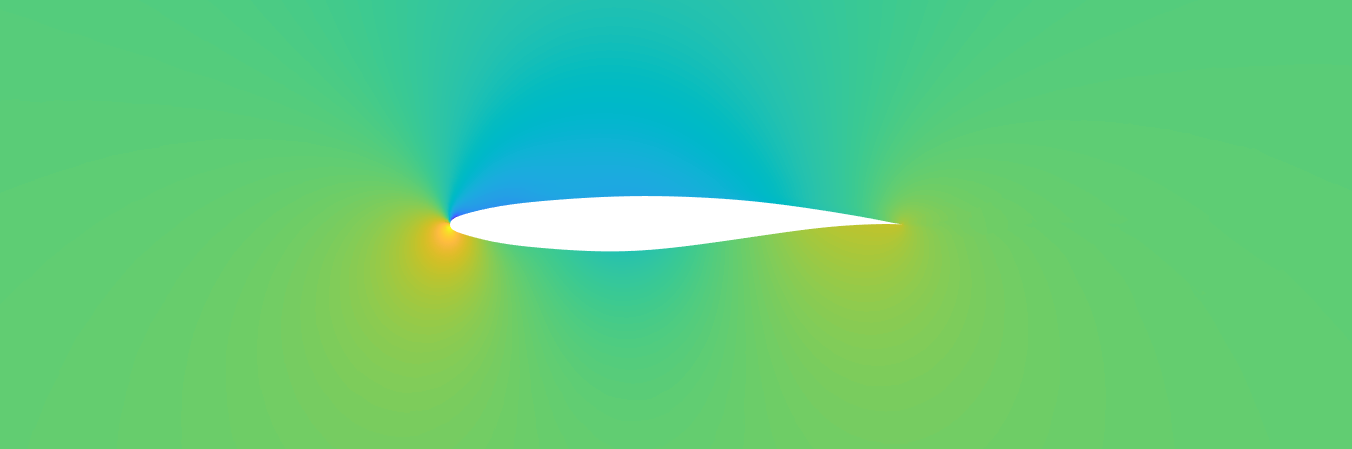}
    \end{minipage}
    \hfill
    \begin{minipage}[t]{0.49\textwidth}
        \centering
        \includegraphics[width=\textwidth]{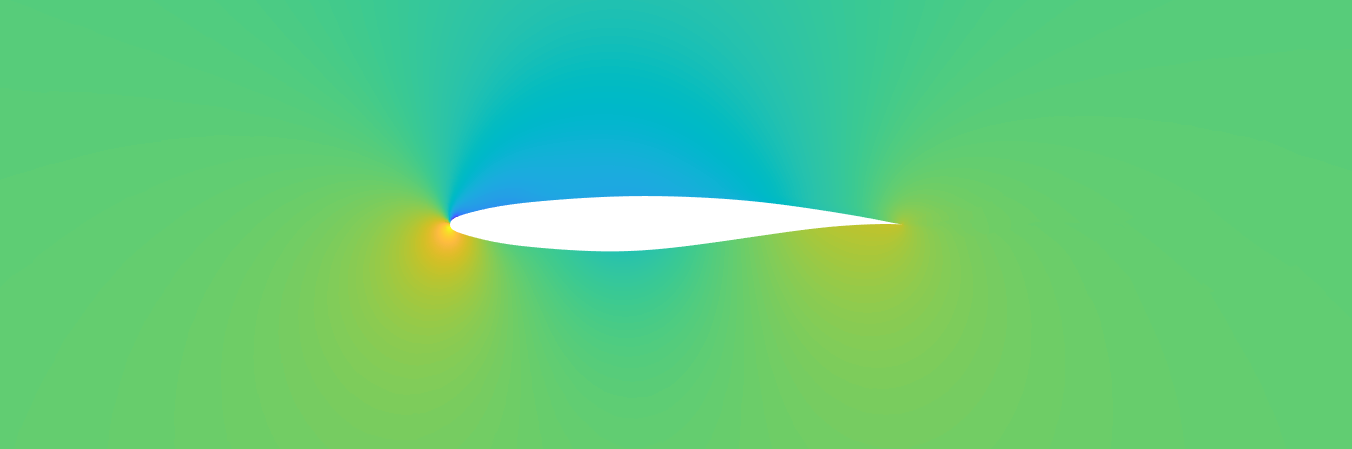}
    \end{minipage}
    \colorbarMatlabParula{0.79}{0.88}{0.96}{1.04}{1.12}
    \caption{The domain shape and density for the inverse design problem at the starting configuration (\textit{top-left}), HDM optimal solution (\textit{top-right}), ROM optimal solution (\textit{bottom-left}), and EQP optimal solution (\textit{bottom-right}).}
    \label{fig:inv1:optimal_soln}
\end{figure}
\begin{table}
\centering
\begin{tabular}{c c c c c c c} 
\hline
Method & $\epsilon$ & \# HDM & \# ROM &\# EQP & Cost(s) & Speedup\\ [0.5ex] 
\hline
HDM & 1.0000e-03 &     90 &      - &      - & 1.19e+04 & - \\
    & 1.0000e-04 &    121 &      - &      - & 1.58e+04 & - \\
    & 1.0000e-06 &    147 &      - &      - & 1.91e+04 & - \\
\hline 
ROM/TR & 1.0000e-03 &      9 &    216 &      - & 1.63e+03 & 7.30 \\
    & 1.0000e-04 &      9 &    216 &      - & 1.63e+03 & 9.74 \\
    & 1.0000e-06 &     23 &   2408 &      - & 9.05e+03 & 2.11 \\
\hline 
EQP/TR & 1.0000e-03 &     11 &     10 &    200 & 1.49e+03 & 7.97 \\
    & 1.0000e-04 &     14 &     13 &    534 & 2.10e+03 & 7.53 \\
    & 1.0000e-06 &     26 &     32 &   2691 & 5.03e+03 & 3.80 \\
\hline

\end{tabular}
\caption{Performance comparison for various methods applied to the inverse design problem in (\ref{eqn:contin_inv}) for a constraint violation tolerance of $\norm{\cbm_i}_\infty \leq 10^{-4}$. The speedup is defined as the cost of a particular method divided by the cost of the HDM method at the same cutoff tolerance $\epsilon$. For the ROM/TR method, the objective value has already reached a tolerance of $\epsilon = 10^{-4}$ when the constraint violation first drops below the specified tolerance, which explains the identical performance for the $\epsilon = 10^{-3}$ and $\epsilon = 10^{-4}$ cutoffs.}
\label{tab:inv1:con_1e4}
\end{table}
\begin{table}
\centering
\begin{tabular}{c c c c c c c} 
\hline
Method & $\epsilon$ & \# HDM & \# ROM &\# EQP & Cost(s) & Speedup\\ [0.5ex] 
\hline
HDM & 1.0000e-03 &     90 &      - &      - & 1.19e+04 & - \\
    & 1.0000e-04 &    121 &      - &      - & 1.58e+04 & - \\
    & 1.0000e-06 &    147 &      - &      - & 1.91e+04 & - \\
\hline 
ROM/TR & 1.0000e-03 &     12 &    783 &      - & 3.46e+03 & 3.43 \\
    & 1.0000e-04 &     12 &    783 &      - & 3.46e+03 & 4.57 \\
    & 1.0000e-06 &     23 &   2408 &      - & 9.05e+03 & 2.11 \\
\hline 
EQP/TR & 1.0000e-03 &     14 &     13 &    534 & 2.10e+03 & 5.64 \\
    & 1.0000e-04 &     14 &     13 &    534 & 2.10e+03 & 7.53 \\
    & 1.0000e-06 &     26 &     32 &   2691 & 5.03e+03 & 3.80 \\
\hline 

\end{tabular}
\caption{Performance comparison for various methods applied to the inverse design problem in (\ref{eqn:contin_inv}) for a constraint violation tolerance of $\norm{\cbm_i}_\infty \leq 10^{-5}$. Speedup is defined in Table~\ref{tab:inv1:con_1e4}. The objective value has already reached a tolerance of $\epsilon = 10^{-4}$ when the constraint violation first drops below the specified tolerance, which explains the identical performance for ROM/TR and EQP/TR methods with $\epsilon = 10^{-3}$ and $\epsilon = 10^{-4}$ cutoffs.}
\label{tab:inv1:con_1e5}
\end{table}
\begin{table}
\centering
\begin{tabular}{c c c c c c c} 
\hline
Method & $\epsilon$ & \# HDM & \# ROM &\# EQP & Cost(s) & Speedup\\ [0.5ex] 
\hline
HDM & 1.0000e-03 &     90 &      - &      - & 1.19e+04 & - \\
    & 1.0000e-04 &    121 &      - &      - & 1.58e+04 & - \\
    & 1.0000e-06 &    147 &      - &      - & 1.91e+04 & - \\
\hline 
ROM/TR & 1.0000e-03 &     23 &   2408 &      - & 9.05e+03 & 1.31 \\
    & 1.0000e-04 &     23 &   2408 &      - & 9.05e+03 & 1.75 \\
    & 1.0000e-06 &     23 &   2408 &      - & 9.05e+03 & 2.11 \\
\hline 
EQP/TR & 1.0000e-03 &     26 &     32 &   2691 & 5.03e+03 & 2.36 \\
    & 1.0000e-04 &     26 &     32 &   2691 & 5.03e+03 & 3.15 \\
    & 1.0000e-06 &     26 &     32 &   2691 & 5.03e+03 & 3.80 \\
\hline 

\end{tabular}
\caption{Performance comparison for various methods applied to the inverse design problem in (\ref{eqn:contin_inv}) for a constraint violation tolerance of $\norm{\cbm_i}_\infty \leq 10^{-6}$. Speedup is defined in Table~\ref{tab:inv1:con_1e4}. The objective value has already reached a tolerance of $\epsilon = 10^{-6}$ when the constraint violation first drops below the specified tolerance, which explains the identical performance for the ROM/TR and EQP/TR methods across the values of $\epsilon$ considered.}
\label{tab:inv1:con_1e6}
\end{table}
\subsection{Inverse design of an airfoil in inviscid, subsonic flow with two side constraints}
\label{sec:numexp2}
We revisit the aerodynamic shape optimization problem from Section~\ref{sec:numexp1} with an additional side constraint on the lift, that is, we aim to construct an airfoil with a flow field as close to the RAE2822 as possible with a lower drag and larger lift. We discretize the shape of the airfoil with 18 Bezier control points and employ the same governing equations and discretization detailed in Section~\ref{sec:numexp1}, which leads to the following shape optimization problem 
\begin{equation}\label{airfoil:eqn:prob}
    \optconThree{U,\mubold}
    {\frac{1}{2}\int_{\Omega(\mubold)} \norm{U - U_\mathrm{RAE2822}}^2 ~ dV}
    {\Lcal(U;\mubold)=0}
    {D(U,\mubold) \leq 0.98 D_{\mathrm{RAE2822}}}
    {L(U,\mubold) \geq 1.01 L_{\mathrm{RAE2822}}}
\end{equation}
where $L(U,\mubold)$ is the lift of the airfoil with shape defined by the control points $\mubold$ and flow field $U$, $L_{\mathrm{RAE2822}} = L(U_\mathrm{RAE2822},\mubold_\mathrm{RAE2822})$ is the lift of the RAE2822 airfoil, and $\Lcal$ is the differential operator that includes the Euler equations (\ref{eqn:numexp:euler}) and appropriate boundary conditions. 

Table~\ref{fig:inv2_final_hist} presents the convergence history of the EQP/TR method at selected iterations with the best practices determined in Sections~\ref{sec:param}-\ref{sec:snap}, i.e., $\tau_0 = 10$, $a = 50$ and do not use snapshot from previous AL iterations. The method is converging to a feasible point with $\norm{\chi_i}_\infty = 8.0\times 10^{-6}$ in 14 major iterations. In this problem, we observed a smaller element usage (approx. $13\%$) at the starting point of each major iteration and similar maximum element usage (approx. $21\%$) compared with the previous problem in Section~\ref{sec:numexp1}. This implies element usage is not necessarily directly proportional to the number of side constraints.
\begin{table}[H]
\centering
\begin{tabular}{c c c c c c c} 
    \hline
    Iteration & $j_i$ & $||\cbm_i||_\infty$ & $||\chi_i||_\infty$ & $S_i$ & $||\rhobold_i||_{0,\mathrm{min}}$ $(\%)$ & $||\rhobold_i||_{0,\mathrm{max}}$ $(\%)$\\ [0.5ex]
    \hline
    1 & 3.0747e-04 & 3.0329e-03 & 4.9006e-02 & 9.3958e-02 & 14.34 & 14.34 \\
2 & 2.8297e-04 & 2.8957e-03 & 7.2681e-03 & 8.6364e-02 & 13.26 & 13.26 \\
3 & 9.9921e-06 & 2.8855e-04 & 2.1097e-04 & 1.7339e-03 & 13.46 & 19.16 \\
4 & 4.1266e-06 & 7.7998e-05 & 5.1784e-06 & 8.4576e-05 & 13.75 & 20.92 \\
5 & 4.2723e-06 & 1.6064e-05 & 3.1604e-06 & 3.9392e-05 & 13.85 & 13.85 \\
10 & 4.3122e-06 & 2.3423e-06 & 1.3290e-06 & 2.7027e-05 & 14.05 & 21.12 \\
14 & 4.3694e-06 & 4.5063e-07 & 8.0079e-06 & 9.3098e-06 & 14.24 & 30.75 \\
\hline 

\end{tabular}
\caption{Convergence history of the EQP/TR method applied to the inverse design problem in (\ref{airfoil:eqn:prob}), where $||\rhobold_i||_{0,\mathrm{min}}$ ($||\rhobold_i||_{0,\mathrm{max}}$) is the fewest (most) nonzero EQP weights across all trust-region iterations at major iteration $i$.}
\label{tab:inv2_hist}
\end{table}

Finally, we compare the performance of the HDM, ROM, and EQP methods in terms of the total computational cost required to satisfy $S_i < \epsilon$ for $\epsilon \in \{10^{-3},10^{-4}, 5 \times 10^{-5}\}$ and various constraint violation tolerances. Figure~\ref{fig:inv2_final_hist} shows the convergence history of the objective value $j_i$ and constrain violation $\norm{\cbm_i}_\infty$. Similar to the previous problem, both the ROM/TR and EQP/TR methods converge much faster than the HDM-based method with the EQP/TR being the fastest. Unlike the previous problem, the initial configuration is not feasible with respect to the side constraints so the HDM (interior-point) method includes feasibility restoration. For this problem, the reduction in constraint violation is comparable between the three methods; however, the ROM/TR and EQP/TR methods reduce the objective value much more rapidly. Tables~\ref{tab:inv2:con_1e4}-\ref{tab:inv2:con_1e6} show the performance of each method for thre constraint violation tolerances ($10^{-4}$, $10^{-5}$, $10^{-6}$), respectively. Similar to the previous problem, the ROM/TR and EQP/TR methods acheive the larstest speedup for loose constraint tolerances, where the EQP/TR speedup can be as large as $12.7\times$. For tighter constraint violation tolerances, the speedup of both methods is more modest ($< 2\times$), and the ROM/TR method can actually be more expensive than the HDM-based method. Figure~\ref{fig:inv2:optimal_soln} shows the initial and optimal shapes for the three methods considered.
\begin{figure}
    \centering
    \begin{minipage}[t]{0.49\textwidth}
        \centering
        \tikzset{every picture/.style={scale=0.98}}
        \begin{tikzpicture}
\begin{axis}[
xmin=0,
ymode=log,
xlabel={time $(s)$},
ylabel={$j_i$}]
\addplot [mark=*, mark size=2, mark options={solid}, thick, solid, blue, mark repeat=1]
coordinates {
( 0.00000000e+00,  3.22557827e-03)
( 1.42366350e+02,  3.07467275e-04)
( 2.81821028e+02,  2.82973368e-04)
( 7.06166995e+02,  9.99214554e-06)
( 1.32864774e+03,  4.12659586e-06)
( 1.46526193e+03,  4.27234261e-06)
( 2.06359751e+03,  4.30220246e-06)
( 2.41466996e+03,  4.31054140e-06)
( 2.78828088e+03,  4.31091644e-06)
( 3.18071673e+03,  4.31125035e-06)
( 4.47353141e+03,  4.31222502e-06)
( 4.74727183e+03,  4.31203869e-06)
( 5.90594441e+03,  4.31212874e-06)
( 6.65548103e+03,  4.32765340e-06)
( 8.79925601e+03,  4.36937441e-06)};\label{line:feqp}

\addplot [mark=o, mark size=2, mark options={solid}, thick, dashed, red, mark repeat=1]
coordinates {
( 0.00000000e+00,  3.22557827e-03)
( 1.50092922e+02,  1.40748242e-03)
( 1.39502069e+03,  2.14045655e-05)
( 1.55368553e+03,  7.93831023e-06)
( 2.67303204e+03,  4.21295665e-06)
( 2.82341377e+03,  4.30745324e-06)
( 3.49581536e+03,  4.30937343e-06)
( 5.59314546e+03,  4.31153801e-06)
( 6.31028270e+03,  4.31208022e-06)
( 7.01971251e+03,  4.31270206e-06)
( 9.06850557e+03,  4.31347507e-06)
( 9.74981957e+03,  4.31358929e-06)
( 1.24664202e+04,  4.31415424e-06)
( 1.44490367e+04,  4.33689109e-06)
( 1.57909705e+04,  4.36158629e-06)};\label{line:from}

\addplot [thick, solid, black]
coordinates {
( 1.20522517e+02,  3.22557827e-03)
( 5.32199104e+02,  7.47485338e-03)
( 9.73416930e+02,  2.34082942e-02)
( 1.38225005e+03,  3.07701383e-02)
( 1.75708921e+03,  2.16454973e-02)
( 2.11279769e+03,  1.51769863e-02)
( 2.44486340e+03,  1.04513727e-02)
( 2.59900820e+03,  2.91350160e-03)
( 2.73424886e+03,  2.38447446e-04)
( 3.00486638e+03,  2.50821132e-04)
( 3.27557261e+03,  2.88527857e-04)
( 3.54627025e+03,  3.45732569e-04)
( 3.69150593e+03,  4.74785009e-04)
( 3.83665927e+03,  5.06390613e-04)
( 3.97194146e+03,  3.72354940e-04)
( 4.10720774e+03,  1.96308044e-04)
( 4.24255341e+03,  2.04203760e-04)
( 4.37790063e+03,  1.28497156e-04)
( 4.51311882e+03,  8.36542157e-05)
( 4.64839181e+03,  8.19406257e-05)
( 4.91933413e+03,  7.56229411e-05)
( 5.18967063e+03,  7.19637585e-05)
( 5.46052702e+03,  7.10998110e-05)
( 5.73120948e+03,  6.91488266e-05)
( 5.86651276e+03,  6.90573303e-05)
( 6.00188972e+03,  6.93159989e-05)
( 6.13729657e+03,  6.95934806e-05)
( 6.27635712e+03,  6.99277341e-05)
( 6.41148771e+03,  7.00313032e-05)
( 6.54667105e+03,  6.96945522e-05)
( 6.68244840e+03,  6.91024327e-05)
( 6.81774087e+03,  6.85268804e-05)
( 7.08827818e+03,  6.80701284e-05)
( 7.36647373e+03,  6.79130060e-05)
( 7.50176798e+03,  6.80409711e-05)
( 7.63726169e+03,  4.39193045e-05)
( 7.76689762e+03,  3.37956270e-05)
( 7.89657504e+03,  3.26332250e-05)
( 8.02628531e+03,  2.97198019e-05)
( 8.28563602e+03,  2.81721508e-05)
( 8.41528176e+03,  2.79734758e-05)
( 8.54507285e+03,  2.79751778e-05)
( 8.67469512e+03,  2.79625312e-05)
( 8.93420524e+03,  2.78972336e-05)
( 9.19347314e+03,  2.78251723e-05)
( 9.32689034e+03,  1.79033548e-05)
( 9.46028032e+03,  1.31796826e-05)
( 9.59366215e+03,  1.29718477e-05)
( 9.72730712e+03,  1.29003529e-05)
( 9.86074618e+03,  1.28340388e-05)
( 9.99424191e+03,  1.27374912e-05)
( 1.01278043e+04,  1.26221828e-05)
( 1.02612844e+04,  1.25600839e-05)
( 1.03948596e+04,  1.25403234e-05)
( 1.05281641e+04,  1.25289737e-05)
( 1.06615438e+04,  1.25062356e-05)
( 1.07948296e+04,  1.24635429e-05)
( 1.09282149e+04,  1.23952806e-05)
( 1.10619499e+04,  1.23293546e-05)
( 1.11957658e+04,  1.22988850e-05)
( 1.13292044e+04,  1.22921446e-05)
( 1.14628252e+04,  1.22899351e-05)
( 1.15963822e+04,  1.22858426e-05)
( 1.17300700e+04,  1.22781822e-05)
( 1.18634402e+04,  1.22598734e-05)
( 1.19967929e+04,  1.22158409e-05)
( 1.21302699e+04,  1.21118761e-05)
( 1.22636936e+04,  1.19015996e-05)
( 1.23970929e+04,  1.16026230e-05)
( 1.25307601e+04,  1.13847977e-05)
( 1.26642364e+04,  1.13258692e-05)
( 1.27975555e+04,  1.13220191e-05)
( 1.29308695e+04,  1.13214929e-05)
( 1.30642617e+04,  1.13189347e-05)
( 1.31977376e+04,  1.13118253e-05)
( 1.33311143e+04,  1.12928803e-05)
( 1.34644743e+04,  1.12482337e-05)
( 1.35979543e+04,  1.11558924e-05)
( 1.37313375e+04,  1.10172526e-05)
( 1.38647979e+04,  1.09061617e-05)
( 1.39982659e+04,  1.08708661e-05)
( 1.41316667e+04,  7.28639542e-06)
( 1.42650882e+04,  6.18547360e-06)
( 1.43986516e+04,  6.09399326e-06)
( 1.45320084e+04,  6.09114930e-06)
( 1.46654404e+04,  6.09303118e-06)
( 1.47950570e+04,  5.11819044e-06)
( 1.49247977e+04,  4.83685223e-06)
( 1.50543348e+04,  4.82110764e-06)
( 1.51839171e+04,  4.82095873e-06)
( 1.53135610e+04,  4.82101406e-06)
( 1.54431492e+04,  4.82102509e-06)
( 1.55728443e+04,  4.82103486e-06)
( 1.57025867e+04,  4.82102602e-06)
( 1.58323713e+04,  4.82097533e-06)
( 1.59620942e+04,  4.82083522e-06)
( 1.60917482e+04,  4.82049078e-06)
( 1.62214255e+04,  4.81976300e-06)
( 1.63511335e+04,  4.81856817e-06)
( 1.64809789e+04,  4.81729802e-06)
( 1.66106586e+04,  4.81654032e-06)
( 1.67403973e+04,  4.81628512e-06)
( 1.68700185e+04,  4.54856024e-06)
( 1.69998829e+04,  4.49524987e-06)
( 1.71296184e+04,  4.49345736e-06)
( 1.72595447e+04,  4.49346215e-06)
( 1.73894288e+04,  4.42537857e-06)
( 1.75194392e+04,  4.41872701e-06)
( 1.76489903e+04,  4.41865807e-06)
( 1.77786936e+04,  4.41865815e-06)
( 1.79084981e+04,  4.40313952e-06)
( 1.80381711e+04,  4.40266352e-06)
( 1.81683620e+04,  4.40266361e-06)
( 1.82979776e+04,  4.39942494e-06)
( 1.84276529e+04,  4.39940380e-06)
( 1.85573425e+04,  4.39940380e-06)
( 1.90765629e+04,  4.39940380e-06)};\label{line:fhdm}

\end{axis}
\end{tikzpicture}
    \end{minipage}
    \hfill
    \begin{minipage}[t]{0.49\textwidth}
        \centering
        \tikzset{every picture/.style={scale=0.98}}
        \begin{tikzpicture}
\begin{axis}[
xmin=0,
ymode=log,
xlabel={time $(s)$},
ylabel={$||\cbm_i||_\infty$}]
\addplot [mark=*, mark size=2, mark options={solid}, thick, solid, blue, mark repeat=1]
coordinates {
( 0.00000000e+00,  4.92245076e-02)
( 1.42366350e+02,  3.04597947e-03)
( 2.81821028e+02,  2.89583483e-03)
( 7.06166995e+02,  2.88776204e-04)
( 1.32864774e+03,  7.80240696e-05)
( 1.46526193e+03,  1.62422675e-05)
( 2.06359751e+03,  4.48659223e-06)
( 2.41466996e+03,  2.46293164e-06)
( 2.78828088e+03,  2.44318826e-06)
( 3.18071673e+03,  2.41352165e-06)
( 4.47353141e+03,  2.34567615e-06)
( 4.74727183e+03,  2.31712442e-06)
( 5.90594441e+03,  2.29265963e-06)
( 6.65548103e+03,  1.64183498e-06)
( 8.79925601e+03,  4.50645321e-07)};\label{line:ceqp}

\addplot [mark=o, mark size=2, mark options={solid}, thick, dashed, red, mark repeat=1]
coordinates {
( 0.00000000e+00,  4.92245076e-02)
( 1.50092922e+02,  1.38839321e-03)
( 1.39502069e+03,  8.04289270e-04)
( 1.55368553e+03,  5.22105240e-04)
( 2.67303204e+03,  5.26255641e-05)
( 2.82341377e+03,  1.36169883e-05)
( 3.49581536e+03,  2.81027489e-06)
( 5.59314546e+03,  2.26597827e-06)
( 6.31028270e+03,  2.23266138e-06)
( 7.01971251e+03,  2.20290178e-06)
( 9.06850557e+03,  2.13998826e-06)
( 9.74981957e+03,  2.11956857e-06)
( 1.24664202e+04,  2.06766069e-06)
( 1.44490367e+04,  1.14265008e-06)
( 1.57909705e+04,  5.90603490e-07)};\label{line:crom}

\addplot [thick, solid, black]
coordinates {
( 1.20522517e+02,  4.92238491e-02)
( 5.32199104e+02,  6.34129167e-04)
( 9.73416930e+02,  1.38986238e-03)
( 1.38225005e+03,  1.69086542e-03)
( 1.75708921e+03,  1.31641827e-03)
( 2.11279769e+03,  1.02803030e-03)
( 2.44486340e+03,  7.95435898e-04)
( 2.59900820e+03,  3.36044044e-04)
( 2.73424886e+03,  0.00000000e+00)
( 3.00486638e+03,  3.39079625e-06)
( 3.27557261e+03,  4.12988300e-06)
( 3.54627025e+03,  0.00000000e+00)
( 3.69150593e+03,  0.00000000e+00)
( 3.83665927e+03,  0.00000000e+00)
( 3.97194146e+03,  0.00000000e+00)
( 4.10720774e+03,  0.00000000e+00)
( 4.24255341e+03,  0.00000000e+00)
( 4.37790063e+03,  0.00000000e+00)
( 4.51311882e+03,  0.00000000e+00)
( 4.64839181e+03,  0.00000000e+00)
( 4.91933413e+03,  0.00000000e+00)
( 5.18967063e+03,  0.00000000e+00)
( 5.46052702e+03,  0.00000000e+00)
( 5.73120948e+03,  0.00000000e+00)
( 5.86651276e+03,  0.00000000e+00)
( 6.00188972e+03,  0.00000000e+00)
( 6.13729657e+03,  0.00000000e+00)
( 6.27635712e+03,  0.00000000e+00)
( 6.41148771e+03,  0.00000000e+00)
( 6.54667105e+03,  0.00000000e+00)
( 6.68244840e+03,  0.00000000e+00)
( 6.81774087e+03,  0.00000000e+00)
( 7.08827818e+03,  0.00000000e+00)
( 7.36647373e+03,  0.00000000e+00)
( 7.50176798e+03,  0.00000000e+00)
( 7.63726169e+03,  0.00000000e+00)
( 7.76689762e+03,  0.00000000e+00)
( 7.89657504e+03,  0.00000000e+00)
( 8.02628531e+03,  0.00000000e+00)
( 8.28563602e+03,  0.00000000e+00)
( 8.41528176e+03,  0.00000000e+00)
( 8.54507285e+03,  0.00000000e+00)
( 8.67469512e+03,  0.00000000e+00)
( 8.93420524e+03,  0.00000000e+00)
( 9.19347314e+03,  0.00000000e+00)
( 9.32689034e+03,  0.00000000e+00)
( 9.46028032e+03,  0.00000000e+00)
( 9.59366215e+03,  0.00000000e+00)
( 9.72730712e+03,  0.00000000e+00)
( 9.86074618e+03,  0.00000000e+00)
( 9.99424191e+03,  0.00000000e+00)
( 1.01278043e+04,  0.00000000e+00)
( 1.02612844e+04,  0.00000000e+00)
( 1.03948596e+04,  0.00000000e+00)
( 1.05281641e+04,  0.00000000e+00)
( 1.06615438e+04,  0.00000000e+00)
( 1.07948296e+04,  0.00000000e+00)
( 1.09282149e+04,  0.00000000e+00)
( 1.10619499e+04,  0.00000000e+00)
( 1.11957658e+04,  0.00000000e+00)
( 1.13292044e+04,  0.00000000e+00)
( 1.14628252e+04,  0.00000000e+00)
( 1.15963822e+04,  0.00000000e+00)
( 1.17300700e+04,  0.00000000e+00)
( 1.18634402e+04,  0.00000000e+00)
( 1.19967929e+04,  0.00000000e+00)
( 1.21302699e+04,  0.00000000e+00)
( 1.22636936e+04,  0.00000000e+00)
( 1.23970929e+04,  0.00000000e+00)
( 1.25307601e+04,  0.00000000e+00)
( 1.26642364e+04,  0.00000000e+00)
( 1.27975555e+04,  0.00000000e+00)
( 1.29308695e+04,  0.00000000e+00)
( 1.30642617e+04,  0.00000000e+00)
( 1.31977376e+04,  0.00000000e+00)
( 1.33311143e+04,  0.00000000e+00)
( 1.34644743e+04,  0.00000000e+00)
( 1.35979543e+04,  0.00000000e+00)
( 1.37313375e+04,  0.00000000e+00)
( 1.38647979e+04,  0.00000000e+00)
( 1.39982659e+04,  0.00000000e+00)
( 1.41316667e+04,  0.00000000e+00)
( 1.42650882e+04,  0.00000000e+00)
( 1.43986516e+04,  0.00000000e+00)
( 1.45320084e+04,  0.00000000e+00)
( 1.46654404e+04,  0.00000000e+00)
( 1.47950570e+04,  0.00000000e+00)
( 1.49247977e+04,  0.00000000e+00)
( 1.50543348e+04,  0.00000000e+00)
( 1.51839171e+04,  0.00000000e+00)
( 1.53135610e+04,  0.00000000e+00)
( 1.54431492e+04,  0.00000000e+00)
( 1.55728443e+04,  0.00000000e+00)
( 1.57025867e+04,  0.00000000e+00)
( 1.58323713e+04,  0.00000000e+00)
( 1.59620942e+04,  0.00000000e+00)
( 1.60917482e+04,  0.00000000e+00)
( 1.62214255e+04,  0.00000000e+00)
( 1.63511335e+04,  0.00000000e+00)
( 1.64809789e+04,  0.00000000e+00)
( 1.66106586e+04,  0.00000000e+00)
( 1.67403973e+04,  0.00000000e+00)
( 1.68700185e+04,  0.00000000e+00)
( 1.69998829e+04,  0.00000000e+00)
( 1.71296184e+04,  0.00000000e+00)
( 1.72595447e+04,  0.00000000e+00)
( 1.73894288e+04,  0.00000000e+00)
( 1.75194392e+04,  0.00000000e+00)
( 1.76489903e+04,  0.00000000e+00)
( 1.77786936e+04,  0.00000000e+00)
( 1.79084981e+04,  0.00000000e+00)
( 1.80381711e+04,  0.00000000e+00)
( 1.81683620e+04,  0.00000000e+00)
( 1.82979776e+04,  0.00000000e+00)
( 1.84276529e+04,  0.00000000e+00)
( 1.85573425e+04,  0.00000000e+00)
( 1.90765629e+04,  0.00000000e+00)};\label{line:chdm}

\end{axis}
\end{tikzpicture}
    \end{minipage}
    \caption{Convergence history (only major iterations shown) of the objective value (\textit{left}) and constraint violation (\textit{right}) of the HDM-based interior point solver (\ref{line:fhdm}), ROM/TR method (\ref{line:from}), and EQP/TR method (\ref{line:feqp}) applied to the inverse design problem in (\ref{airfoil:eqn:prob}).}
    \label{fig:inv2_final_hist}
\end{figure}
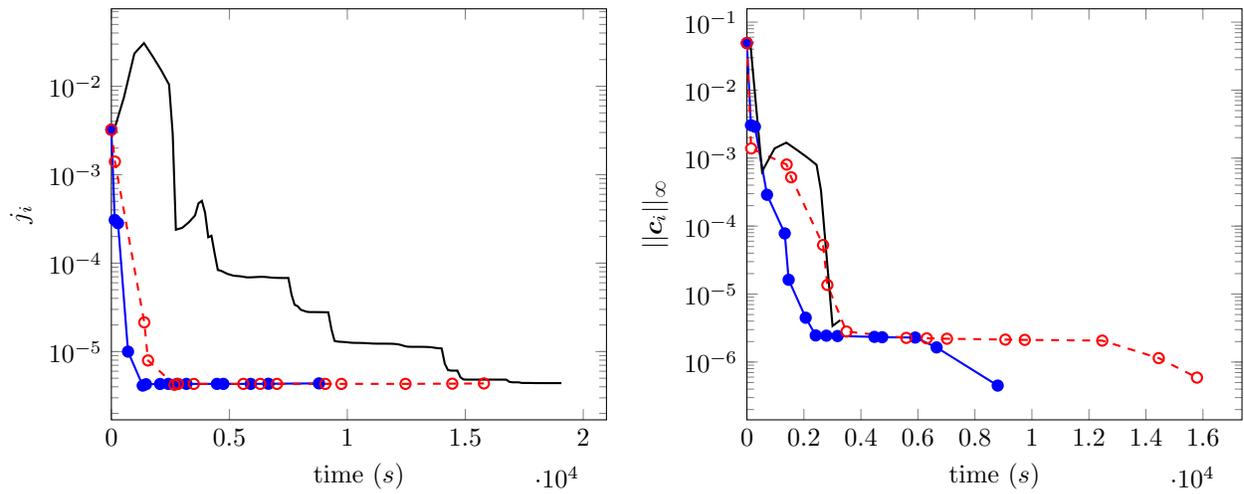
\begin{figure}
    \centering
    \begin{minipage}[t]{0.49\textwidth}
        \centering
        \includegraphics[width=\textwidth]{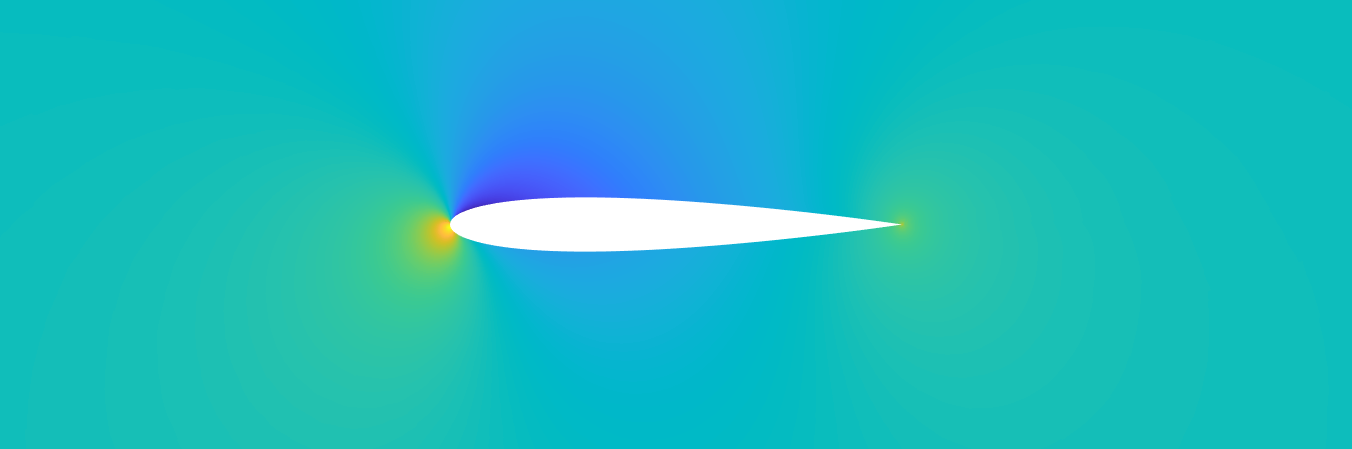}
    \end{minipage}
    \hfill
    \begin{minipage}[t]{0.49\textwidth}
        \centering
        \includegraphics[width=\textwidth]{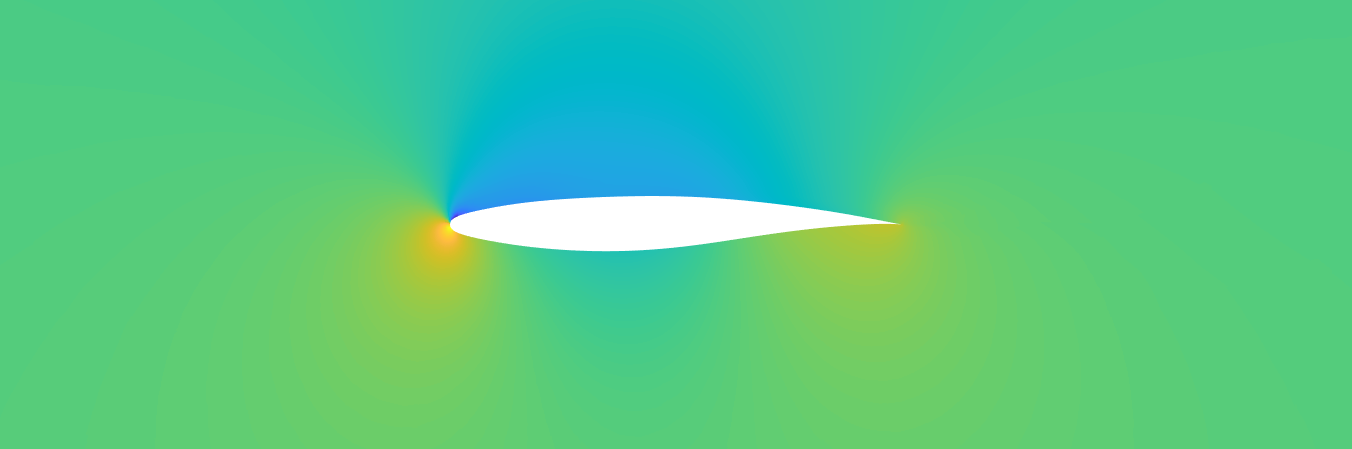}
    \end{minipage}
        \begin{minipage}[t]{0.49\textwidth}
        \centering
        \includegraphics[width=\textwidth]{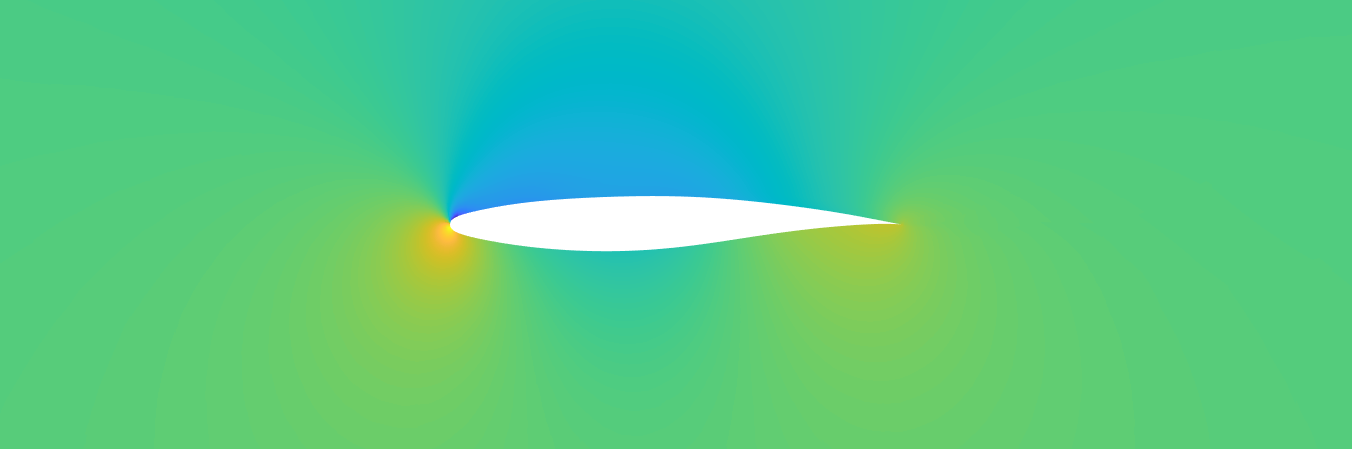}
    \end{minipage}
    \hfill
    \begin{minipage}[t]{0.49\textwidth}
        \centering
        \includegraphics[width=\textwidth]{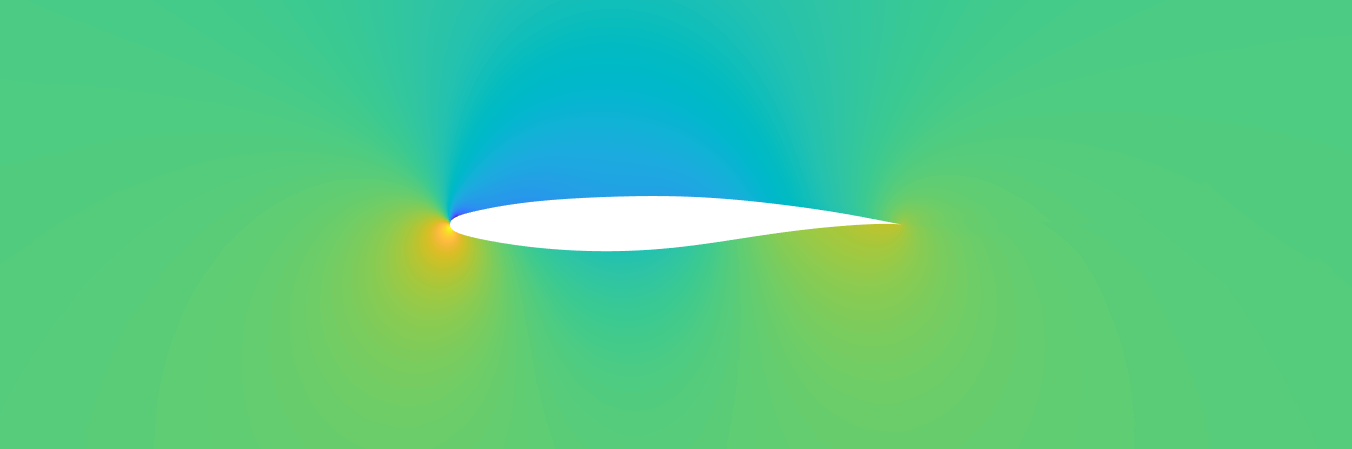}
    \end{minipage}
    \colorbarMatlabParula{0.81}{0.89}{0.97}{1.05}{1.12}
    \caption{The domain shape and density for the inverse design problem at the starting configuration (\textit{top-left}), HDM optimal solution (\textit{top-right}), ROM optimal solution (\textit{bottom-left}), and EQP optimal solution (\textit{bottom-right}).}
    \label{fig:inv2:optimal_soln}
\end{figure}
\begin{table}[H]
\centering
\begin{tabular}{c c c c c c c} 
\hline
Method & $\epsilon$ & \# HDM & \# ROM &\# EQP & Cost(s) & Speedup\\ [0.5ex] 
\hline
HDM & 1.0000e-03 &    100 &      - &      - & 1.41e+04 & - \\
    & 1.0000e-04 &    121 &      - &      - & 1.69e+04 & - \\
    & 5.0000e-05 &    121 &      - &      - & 1.69e+04 & - \\
\hline 
ROM/TR & 1.0000e-03 &     15 &    171 &      - & 2.67e+03 & 5.29 \\
    & 1.0000e-04 &     15 &    171 &      - & 2.67e+03 & 6.31 \\
    & 5.0000e-05 &     16 &    177 &      - & 2.82e+03 & 5.98 \\
\hline 
EQP/TR & 1.0000e-03 &     10 &      9 &    124 & 1.33e+03 & 10.64 \\
    & 1.0000e-04 &     10 &      9 &    124 & 1.33e+03 & 12.70 \\
    & 5.0000e-05 &     11 &     10 &    130 & 1.47e+03 & 11.51 \\
\hline 

\end{tabular}
\caption{Performance comparison for various methods applied to the inverse design problem in (\ref{airfoil:eqn:prob}) for a constraint violation tolerance of $\norm{\cbm_i}_\infty \leq 10^{-4}$. Speedup is defined in Table~\ref{tab:inv1:con_1e4}. The objective value has already reached a tolerance of $\epsilon = 10^{-4}$ when the constraint violation first drops below the specified tolerance, which explains the identical performance for ROM/TR and EQP/TR methods with $\epsilon = 10^{-3}$ and $\epsilon = 10^{-4}$ cutoffs.}
\label{tab:inv2:con_1e4}
\end{table}
\begin{table}[H]
\centering
\begin{tabular}{c c c c c c c} 
\hline
Method & $\epsilon$ & \# HDM & \# ROM &\# EQP & Cost(s) & Speedup\\ [0.5ex] 
\hline
HDM & 1.0000e-03 &    100 &      - &      - & 1.41e+04 & - \\
    & 1.0000e-04 &    121 &      - &      - & 1.69e+04 & - \\
    & 5.0000e-05 &    121 &      - &      - & 1.69e+04 & - \\
\hline 
ROM/TR & 1.0000e-03 &     18 &    334 &      - & 3.50e+03 & 4.04 \\
    & 1.0000e-04 &     18 &    334 &      - & 3.50e+03 & 4.83 \\
    & 5.0000e-05 &     18 &    334 &      - & 3.50e+03 & 4.83 \\
\hline 
EQP/TR & 1.0000e-03 &     14 &     13 &    425 & 2.06e+03 & 6.85 \\
    & 1.0000e-04 &     14 &     13 &    425 & 2.06e+03 & 8.18 \\
    & 5.0000e-05 &     14 &     13 &    425 & 2.06e+03 & 8.18 \\
\hline 

\end{tabular}
\caption{Performance comparison for various methods applied to the inverse design problem in (\ref{airfoil:eqn:prob}) for a constraint violation tolerance of $\norm{\cbm_i}_\infty \leq 10^{-5}$. Speedup is defined in Table~\ref{tab:inv1:con_1e4}. The objective value has already reached a tolerance of $\epsilon = 10^{-4}$ when the constraint violation first drops below the specified tolerance, which explains the identical performance for the ROM/TR and EQP/TR methods across the values of $\epsilon$ considered.}
\label{tab:inv2:con_1e5}
\end{table}
\begin{table}[H]
\centering
\begin{tabular}{c c c c c c c} 
\hline
Method & $\epsilon$ & \# HDM & \# ROM &\# EQP & Cost(s) & Speedup\\ [0.5ex] 
\hline
HDM & 1.0000e-03 &    100 &      - &      - & 1.41e+04 & - \\
    & 1.0000e-04 &    121 &      - &      - & 1.69e+04 & - \\
    & 5.0000e-05 &    121 &      - &      - & 1.69e+04 & - \\
\hline 
ROM/TR & 1.0000e-03 &     44 &   3705 &      - & 1.58e+04 & 0.89 \\
    & 1.0000e-04 &     44 &   3705 &      - & 1.58e+04 & 1.07 \\
    & 5.0000e-05 &     44 &   3705 &      - & 1.58e+04 & 1.07 \\
\hline 
EQP/TR & 1.0000e-03 &     48 &    106 &   3200 & 8.80e+03 & 1.61 \\
    & 1.0000e-04 &     48 &    106 &   3200 & 8.80e+03 & 1.92 \\
    & 5.0000e-05 &     48 &    106 &   3200 & 8.80e+03 & 1.92 \\
\hline 

\end{tabular}
\caption{Performance comparison for various methods applied to the inverse design problem in (\ref{airfoil:eqn:prob}) for a constraint violation tolerance of $\norm{\cbm_i}_\infty \leq 10^{-6}$. Speedup is defined in Table~\ref{tab:inv1:con_1e4}. The objective value has already reached a tolerance of $\epsilon = 5\times 10^{-5}$ when the constraint violation first drops below the specified tolerance, which explains the identical performance for the ROM/TR and EQP/TR methods across the values of $\epsilon$ considered.}
\label{tab:inv2:con_1e6}
\end{table}

\section{Conclusion}
\label{sec:concl}
In this work, we introduced an augmented Lagrangian trust-region method to efficiently solve constrained optimization problems governed by large-scale nonlinear systems. At each major augmented Lagrangian iteration, the expensive optimization subproblem involving the full nonlinear system is replaced by an EQP-based hyperreduced model (Section~\ref{sec:hyperreduction}) constructed on-the-fly. Convergence to a critical point of an augmented Lagrangian subproblem is guaranteed by a novel bound-constrained trust-region method that allows for inexact gradient evaluations that we developed (Section~\ref{sec:trammo:tr_inexact}) and specialized to our specific setting that leverages hyperreduced models (Section~\ref{sec:eqp_trammo}). The approach circumvents a traditional training phase because the EQP models are built on-the-fly in accordance with the requirements of the trust-region convergence theory.

In our previous work \cite{wen_globally_2023}, we carefully studied the impact
of numerous user-defined parameter introduced by the trust-region
framework with inexact gradient evaluations. In this work, we
studied the impact of the user-defined augmented Lagrangian
parameters $(\tau_0, a)$ and found they do not impact convergence,
but have a large influence on convergence rate. For the problems
considered in this work, we found $\tau_0 \in [10, 25]$ and $a \in [10, 100]$
lead to favorable convergence rates.
We also found that inheriting primal and dual solutions from previous major iterations
does not necessarily improve the convergence of the proposed method, but may be useful
in other problems or situations (e.g., unsteady problems).
Using these empirical rules, the proposed EQP/BTR method demonstrated speedups
of up to about 12.7x relative to a popular interior-point optimization method
using only HDM evaluations, even when accounting for all sources of computational
cost, i.e., HDM evaluations for snapshot generation and trust-region assessment,
reduced basis construction, EQP training, trust-region subproblem solves, etc.

HDM evaluations dominate the cost of the proposed EQP/BTR method. This suggests reduced models with improved prediction potential that could reduce the overall number of HDM evaluations would lead to future computational speedup. Additionally, there is likely overhead associated with the nested augmented Lagrangian approach and further speedup may be available from a trust-region method that directly handles the constrained problem (instead of a bound-constrained subproblem). Other interesting research directions include the extension to unsteady problems, the extension to quadrature-based (instead of element-based) EQP \cite{du_adaptive_2021,dua2022efficient} for integration with higher order finite element discretization, and using the methodology to solve relevant engineering problems.

\appendix
\section{Proof of global convergence of the inexact, box-constrained trust-region method}
\label{sec:tr-proofs}
In this appendix, we revisit the proofs of the global convergence for the inexact, box-constrained trust-region framework, and they directly follow those set forth in \cite{kouri_trust-region_2013, yano_globally_2021}. 
\begin{lemma} \label{lemma:delta-zero}
Suppose Assumptions~\ref{assum:at:tr_inexact} hold and there
exists $\epsilon>0$ and $K > 0$ such that $\chi(\mubold_k) \geq \epsilon$ for all $k > K$.
Then the sequence of trust-region radii $\{\Delta_k\}$
produced by Algorithm~\ref{alg:tr_hyp} satisfies
\begin{equation*}
 \sum_{k=1}^\infty \Delta_k < \infty.
\end{equation*}
\begin{proof}


We first consider the case with an infinite sequence of successful iterations $\{k_i\}$ with the new gradient condition in \ref{assum:at:cauchy}. In this case, for all $i$ such that $k_i > K$,
\begin{equation*}
 \begin{aligned}
  F(\mubold_{k_i}) - F(\mubold_{k_{i+1}}) &\geq
  F(\mubold_{k_i}) - F(\mubold_{k_{i}+1})  =
  F(\mubold_{k_i}) - F(\check\mubold_{k_i})
  \geq \eta_1(m_{k_i}(\mubold_{k_i}) - m_{k_i}(\check\mubold_{k_i})) \\
  &\geq \eta_1\kappa\norm{\chi_m(\mubold_{k_i})}
      \min\bracket{\frac{\norm{\chi_m(\mubold_{k_i})}}{\beta_k},~
                 \Delta_{k_i}}
  \geq \eta_1\kappa\epsilon
      \min\bracket{\frac{\epsilon}{\beta_k},~
                 \Delta_{k_i}},
 \end{aligned}
\end{equation*}
for some constant $\kappa\in(0,1)$. The subscript for the second $\mubold$ in the first and second expressions are $k_{i+1}$ and $k_i + 1$, respectively.
Due to the step acceptance condition in Algorithm~\ref{alg:tr_inexact}, the sequence $\{F(\mubold_k)\}$ is non-increasing, and thus the first inequality holds. The first equality holds because iteration $k_i$ is successful (by construction). The remaining inequalities
follow from the step acceptance condition in Algorithm~\ref{alg:tr_inexact}, the fraction of Cauchy decrease, and the assumption that
$\chi_m(\mubold_{k_i}) \geq \epsilon$ for all $k_i > K$.

Now, we sum over all $i > I$ for $k_I = K$ to yield
\begin{equation*}
\eta_1\kappa\epsilon\sum_{i \geq I}
\min\bracket{
\frac{\epsilon}{\beta_{k_i}},
~\Delta_{k_i}}
\leq \sum_{i \geq I} (F(\mubold_{k_i}) - F(\mubold_{k_{i+1}}) )
= F(\mubold_{k_I}) - \lim_{i \rightarrow \infty} F(\mubold_{k_i}) < \infty,
\end{equation*}
where the equality follows from the telescoping series, and the finiteness of the limit follows from $F$ being bounded below. Because $\epsilon/\beta_k$ is bounded away from zero, the inequality above implies that $\sum_{i=1}^\infty \Delta_{k_i} < \infty$. The result also holds for a finite number of successful iterations (Lemma A.2 of \cite{kouri_trust-region_2013}), and the desired result follows.


\end{proof}
\end{lemma}

\begin{lemma} \label{lemma:rho-one}
  Suppose Assumptions~\ref{assum:at:tr_inexact} hold and there exists $\epsilon>0$ and $K > 0$ such that
$\norm{\chi_m(\mubold_k)} \geq \epsilon$ for all $k > K$. Then the ratios
  $\{\varrho_k\}$ produced by Algorithm~\ref{alg:tr_inexact} converges to one.
  
\begin{proof}
By Taylor's theorem and Assumption~\ref{assum:at:tr_inexact}, we have
\begin{equation}
\abs{\mathrm{ared}_k-\mathrm{pred}_k}
\leq \norm{\sbm_k}\norm{\nabla F(\mubold_k) - \nabla m_k(\mubold_k)} + \frac{1}{2}(\zeta_1+\zeta_2-1)\norm{\sbm_k}^2,
\end{equation}
where $\sbm_k = \mubold-\mubold_k$ is the step from the current trust-region center $\mubold_k$. For all $k > K$, the gradient error bound (Assumption~\ref{assum:at:grad}) and the result of Lemma~\ref{lemma:delta-zero} ($\Delta_k \rightarrow 0$) imply
\begin{equation}
\abs{\mathrm{ared}_k-\mathrm{pred}_k} \leq \xi\Delta_k^2 + \frac{1}{2}(\kappa_1+\kappa_2-1)\Delta_k^2.
\end{equation}
Additionally, for $k > K$, the fraction of Cauchy decrease condition and
the result of Lemma~\ref{lemma:delta-zero} ($\Delta_k \rightarrow 0$) imply
\begin{equation}
\mathrm{pred}_k 
\geq \kappa \norm{\chi_m(\mubold_k)}\min\bracket{\frac{\norm{\chi_m(\mubold_k)}}{\beta_k}, \Delta_k}
\geq \kappa \epsilon \Delta_k.
\end{equation}
Combining these results yields
\begin{equation}
\abs{\varrho_{k}-1} 
= \abs{\frac{\mathrm{ared}_k}{\mathrm{pred}_k}-1}
\leq \frac{\xi\Delta_k+\frac{1}{2}(\zeta_1+\zeta_2-1)\Delta_k}{\kappa\epsilon}.
\end{equation}
Thus, as $\Delta_k \rightarrow 0$, $\varrho_k$ approaches $1$, which is the desired result.
\end{proof}
\end{lemma}


\begin{proof}[{Proof of Theorem~\ref{thm:xi_m-zero}}]
We prove by contradiction. Suppose there exists $\epsilon > 0$ such that $\chi_m(\mubold_k) > \epsilon$ for sufficiently large $k > K > 0$. By Lemma~\ref{lemma:rho-one}, there exists $K'' \geq K$ such that, for all $k > K''$, $\varrho_k$ is sufficiently close to $1$ and the corresponding step is successful. From Algorithm~\ref{alg:tr_hyp}, this implies $\Delta_{K''} \leq \Delta_k \leq \Delta_{\mathrm{max}}$, which contradicts Lemma~\ref{lemma:delta-zero}. Hence, $\liminf_{k \rightarrow \infty} \norm{\chi_m(\mubold_k)} = 0$. Next, we have
\begin{equation}
\norm{\chi_m(\mubold_k)-\chi(\mubold_k)} \leq \norm{\nabla m_k(\mubold_k)-\nabla f(\mubold_k)} \leq \xi \min\curlyb{\norm{\chi_m(\mubold_k)}, \Delta_k}
\end{equation}
where the first inequality is a property of the Euclidean projection and the second inequality follows from the gradient error bound. This leads to the desired result when combined with the previous result, $\liminf_{k \rightarrow \infty} \norm{\chi_m(\mubold_k)} = 0$.
\end{proof}

\section{Regularity and boundedness assumptions}
\label{sec:appendix_A}
We begin by stating a series of regularity and boundedness assumptions on both the HDM and hyperreduced model. These
assumptions were introduced in previous work \cite{wen_globally_2023,zahr_efficient_2019} and will be used to derive residual-based error estimates. We also suppress the AL variables ($\thetabold$ and $\tau$) because they are fixed in the trust-region framework.
\begin{assume}\label{assum:hdm}
Consider any open, bounded subset $\Ucal\subset\Rbb^{N_\ubm}$. We assume the HDM
 residual function in (\ref{eqn:hdm_res}) and AL function in (\ref{eqn:hdm_elem_qoi}) satisfy the following:
\begin{enumerate}[label=\textbf{(AH\arabic*)}]
\item\label{ah:r:differ} $\rbm$ is continuously differentiable with respect to both arguments on the domain $\Ucal\times\Dcal$.
\item\label{ah:l:differ} $\ell$ is continuously differentiable with respect to both arguments on the domain $\Ucal\times\Dcal$.
\item\label{ah:l:lip} $\ell$ is Lipschitz continuous with respect to its first argument on the domain $\Ucal\times\Dcal$.
\item\label{ah:c:differ} $\cbm$ is continuously differentiable with respect to both arguments on the domain $\Ucal\times\Dcal$.
\item\label{ah:c:val_deriv_bnd} The constraint function $\cbm$ and its partial derivatives are bounded on the domain $\Ucal \times \Dcal$. 
\item\label{ah:r:du_lip} The Jacobian matrix
\begin{equation}
\pder{\rbm}{\ubm} : \Rbb^{N_\ubm}\times\Rbb^{N_\mubold} \rightarrow \Rbb^{N_\ubm\times N_\ubm}, \qquad
\pder{\rbm}{\ubm} : (\ubm,\mubold) \mapsto \pder{\rbm}{\ubm}(\ubm,\mubold)
\end{equation}
is Lipschitz continuous with respect to its first argument on the domain $\Ucal\times\Dcal$.
\item\label{ah:l:du_lip} The state derivative
\begin{equation}
\pder{\ell}{\ubm} : \Rbb^{N_\ubm}\times\Rbb^{N_\mubold} \rightarrow \Rbb^{1\times N_\ubm}, \qquad
\pder{\ell}{\ubm} : (\ubm,\mubold) \mapsto \pder{\ell}{\ubm}(\ubm,\mubold)
\end{equation}
is Lipschitz continuous with respect to its first argument on the domain $\Ucal\times \Dcal$.
\item\label{ah:r:dmu_lip} The parameter Jacobian matrix
\begin{equation}
\pder{\rbm}{\mubold} : \Rbb^{N_\ubm}\times\Rbb^{N_\mubold} \rightarrow \Rbb^{N_\ubm\times N_\mubold}, \qquad
\pder{\rbm}{\mubold} : (\ubm,\mubold) \mapsto \pder{\rbm}{\mubold}(\ubm,\mubold)
\end{equation}
is Lipschitz continuous with respect to its first argument on the domain $\Ucal\times\Dcal$.
\item\label{ah:l:dmu_lip} The parameter derivative
\begin{equation}
\pder{\ell}{\mubold} : \Rbb^{N_\ubm}\times\Rbb^{N_\mubold} \rightarrow \Rbb^{1\times N_\mubold}, \qquad
\pder{\ell}{\mubold} : (\ubm,\mubold) \mapsto \pder{\ell}{\mubold}(\ubm,\mubold)
\end{equation}
is Lipschitz continuous with respect to its first argument on the domain $\Ucal\times \Dcal$.
\item\label{ah:D:inv} The matrix function
\begin{equation}
\Dbm : \Rbb^{N_\ubm}\times\Rbb^{N_\ubm}\times\Rbb^{N_\mubold} \rightarrow \Rbb^{N_\ubm\times N_\ubm}, \qquad
\Dbm : (\ubm_1,\ubm_2,\zbm) \mapsto \int_0^1 \pder{\rbm}{\ubm}(\ubm_2 + t(\ubm_1-\ubm_2),\mubold) \, dt
\end{equation}
is invertible with bounded inverse on the domain $\Ucal\times\Ucal\times\Dcal$.
\item\label{ah:r:unique} For any $\mubold\in\Dcal$, there is a unique solution $\ubm^\star$ satisfying
$\rbm(\ubm^\star,\mubold) = \zerobold$ and the set of solutions
$\left\{ \ubm \in \Rbb^{N_\ubm} \suchthat \rbm(\ubm,\mubold) = \zerobold, \forall \mubold\in\Dcal\right\}$
is a bounded set.
\end{enumerate}
\end{assume}

\begin{assume}\label{assum:eqp}
Consider any open, bounded subset $\Ycal\subset\Rbb^n$. For any full-rank
reduced basis $\Phibold\in\Rbb^{N_\ubm\times n}$, we assume the hyperreduced residual function in (\ref{eqn:eqp_elem_res}) and AL function in (\ref{eqn:eqp_elem_f}) satisfy the following: for any $\rhobold\in\Rcal$,
\begin{enumerate}[label=\textbf{(AR\arabic*)}]
\item\label{ar:r:differ}$\tilde\rbm_\Phibold(\,\cdot\,,\,\cdot\,;\rhobold)$ is continuously differentiable with respect to both
arguments on the domain $\Ycal\times\Dcal$.
\item\label{ar:l:differ}$\tilde{\ell}_\Phibold(\,\cdot\,,\,\cdot\,;\rhobold)$ is continuously differentiable with respect to both arguments on the domain $\Ycal\times\Dcal$.
\item\label{ar:l:lip} $\tilde{\ell}_\Phibold(\,\cdot\,,\,\cdot\,;\rhobold)$ is Lipschitz continuous with respect to its first argument on the domain $\Ycal\times\Dcal$.
\item\label{ar:c:differ}$\tilde{\cbm}_\Phibold(\,\cdot\,,\,\cdot\,;\rhobold)$ is continuously differentiable with respect to both arguments on the domain $\Ycal\times\Dcal$.
\item\label{ar:c:val_deriv_bnd} The constraint function $\tilde{\cbm}_\Phibold(\,\cdot\,,\,\cdot\,;\rhobold)$ and its partial derivatives are bounded on the domain $\Ycal\times\Dcal$.
\item\label{ar:r:dy_lip} The Jacobian matrix
\begin{equation}
\pder{\tilde\rbm_\Phibold}{\tilde\ybm}(\,\cdot\,,\,\cdot\,;\rhobold) : \Rbb^{n}\times\Rbb^{N_\mubold} \rightarrow \Rbb^{n\times n}, \qquad
\pder{\tilde\rbm_\Phibold}{\tilde\ybm} : (\tilde\ybm,\mubold;\rhobold) \mapsto \pder{\tilde\rbm_\Phibold}{\tilde\ybm}(\tilde\ybm,\mubold;\rhobold)
\end{equation}
is Lipschitz continuous with respect to its first argument on the domain $\Ycal\times\Dcal$.

\item\label{ar:l:dy_lip} The state derivative
\begin{equation}
\pder{\tilde{\ell}_\Phibold}{\tilde\ybm}(\,\cdot\,,\,\cdot\,;\rhobold) : \Rbb^{n}\times\Rbb^{N_\mubold} \rightarrow \Rbb^{1\times n}, \qquad
\pder{\tilde{\ell}_\Phibold}{\tilde\ybm}(\,\cdot\,,\,\cdot\,;\rhobold) : (\tilde\ybm,\mubold;\rhobold) \mapsto \pder{\tilde{\ell}_\Phibold}{\tilde\ybm}(\tilde\ybm,\mubold;\rhobold)
\end{equation}
is Lipschitz continuous with respect to its first argument on the domain $\Ycal\times \Dcal$.
\item\label{ar:r:dmu_lip} The parameter Jacobian matrix
\begin{equation}
\pder{\tilde\rbm_\Phibold}{\mubold}(\,\cdot\,,\,\cdot\,;\rhobold) : \Rbb^{n}\times\Rbb^{N_\mubold} \rightarrow \Rbb^{n\times N_\mubold}, \qquad
\pder{\tilde\rbm_\Phibold}{\mubold} : (\tilde\ybm,\mubold;\rhobold) \mapsto \pder{\tilde\rbm_\Phibold}{\mubold}(\tilde\ybm,\mubold)
\end{equation}
is Lipschitz continuous with respect to its first argument on the domain $\Ycal\times\Dcal$.
\item\label{ar:l:dmu_lip}The parameter derivative
\begin{equation}
\pder{\tilde{\ell}_\Phibold}{\mubold}(\,\cdot\,,\,\cdot\,;\rhobold) : \Rbb^{n}\times\Rbb^{N_\mubold} \rightarrow \Rbb^{1\times N_\mubold}, \qquad
\pder{\tilde{\ell}_\Phibold}{\mubold} : (\tilde\ybm,\mubold;\rhobold) \mapsto \pder{\tilde{\ell}_\Phibold}{\mubold}(\tilde\ybm,\mubold;\rhobold)
\end{equation}
is Lipschitz continuous with respect to its first argument on the domain $\Ycal\times \Dcal$.
\item\label{ar:D:inv} The matrix function
\begin{equation}
\tilde\Dbm_\Phibold(\,\cdot\,,\,\cdot\,,\,\cdot\,;\rhobold) : \Rbb^n\times\Rbb^n\times\Rbb^{N_\mubold} \rightarrow \Rbb^{n \times n}, \qquad
\tilde\Dbm_\Phibold : (\tilde\ybm_1,\tilde\ybm_2,\zbm;\rhobold) \mapsto \int_0^1 \pder{\tilde\rbm_\Phibold}{\tilde\ybm}(\tilde\ybm_2 + t(\tilde\ybm_1-\tilde\ybm_2),\mubold;\rhobold) \, dt
\end{equation}
is invertible with bounded inverse on the domain $\Ycal\times\Ycal\times\Dcal$.
\item\label{ar:r:unique} For any $\mubold\in\Dcal$, there is a unique solution $\tilde\ybm^\star$ satisfying
$\tilde\rbm_\Phibold(\tilde\ybm^\star,\mubold;\rhobold) = \zerobold$, and the set of solutions
$ \left\{ \ybm \in \Rbb^n \suchthat \tilde\rbm_\Phibold(\ybm,\mubold;\rhobold) = \zerobold, \forall \mubold\in\Dcal\right\}$
is a bounded set. 
\end{enumerate}
\end{assume}

\begin{remark}
\ref{ar:r:differ}-\ref{ar:D:inv} follow directly from Assumption~\ref{assum:hdm} in the case where $\rhobold=\onebold$ because $\hat\rbm_\Phibold(\hat\ybm,\mubold) = \tilde\rbm_\Phibold(\hat\ybm,\mubold;\onebold)$
and the relationship between $\rbm$ and $\hat\rbm_\Phibold$ in (\ref{eqn:rom_res}).
\end{remark}

\section{Proof of residual-based output error estimates}
\label{sec:appendix_B}
We prove the Corollaries~\ref{cor:qoi_errbnd} and \ref{cor:qoi_grad_errbnd} of Theorem~\ref{the:qoi_grad_errbnd} in this section. First, we define the reduced objective function, $\hat j_\Phibold: \Rbb^n \times \Dcal \rightarrow \Rbb$, and hyperreduced objective function, $\tilde j_\Phibold: \Rbb^n \times \Dcal \times \Rcal \rightarrow \Rbb$, as
\begin{equation}
  \hat j_\Phibold:(\hat\ybm, \mubold) \mapsto j(\Phibold\hat\ybm, \mubold), \qquad
  \tilde j_\Phibold:(\tilde\ybm, \mubold; \rhobold) \mapsto \sum_{e=1}^{N_\mathtt{e}} \rho_e j_e(\Phibold\tilde\ybm, \mubold).
\end{equation}
Similarly, we define the reduced and hyperreduced constraint function,
$\hat\cbm_\Phibold: \Rbb^n \times \Dcal \rightarrow \Rbb^{N_\cbm}$ and
$\tilde\cbm_\Phibold: \Rbb^n \times \Dcal \times \Rcal \rightarrow \Rbb^{N_\cbm}$,
respectively, as
\begin{equation}
 \hat\cbm_\Phibold:(\hat\ybm, \mubold) \mapsto \cbm(\Phibold\hat\ybm, \mubold), \qquad
 \tilde\cbm_\Phibold:(\tilde\ybm, \mubold; \rhobold) \mapsto \sum_{e=1}^{N_\mathtt{e}} \rho_e \cbm_e(\Phibold\tilde\ybm, \mubold).
\end{equation}
In the following proofs, we drop the subscript $\Phibold$, EQP weights $\rhobold$, and input arguments to the of the solution terms for brevity.

\begin{proof}[Proof of Corollary~\ref{cor:qoi_errbnd}]
From Theorem~\ref{the:qoi_grad_errbnd} and the proposed training procedure (\ref{eqn:reduc_basis}) that guarantees $\ubm^\star \in \mathrm{Ran}~\Phibold$, we have
\begin{equation}
 \abs{f(\mubold;\thetabold, \tau) - \tilde{f}(\mubold;\thetabold,\tau)} \leq c_2\norm{\trbm(\hat\ybm^\star, \mubold)} + \abs{\hat\ell(\hat\ybm^\star, \mubold;\thetabold,\tau)-\tilde\ell(\hat\ybm^\star, \mubold;\thetabold,\tau)}
\end{equation}
because the HDM primal solution can be recovered by the ROM, i.e., $\ubm^\star = \Phibold \hat\ybm^\star$, which implies $\rbm(\Phibold\hat\ybm^\star,\mubold) = \zerobold$. Then, from the definitions introduced in Section~\ref{sec:hyperreduction} and some basic operations, we have
\begin{equation}
\begin{aligned}
&\abs{f(\mubold;\thetabold, \tau) - \tilde{f}(\mubold;\thetabold,\tau)} \\
&\leq c_2\norm{\trbm(\hat\ybm^\star, \mubold)} + \abs{\hat\ell(\hat\ybm^\star, \mubold;\thetabold,\tau)-\tilde\ell(\hat\ybm^\star, \mubold;\thetabold,\tau)}\\
&= c_2\norm{\trbm(\hat\ybm^\star, \mubold)} + \abs{\hat j(\hat\ybm^\star, \mubold)-\thetabold^T\hat\cbm(\hat\ybm^\star, \mubold)+\frac{\tau}{2}\norm{\hat\cbm(\hat\ybm^\star, \mubold)}^2-\tilde j(\hat\ybm^\star, \mubold) + \thetabold^T\tilde\cbm(\hat\ybm^\star, \mubold) -\frac{\tau}{2}\norm{\tilde\cbm(\hat\ybm^\star, \mubold)}^2}\\
&\leq c_2\norm{\trbm(\hat\ybm^\star, \mubold)} + \abs{\hat j(\hat\ybm^\star, \mubold) - \thetabold^T\hat\cbm(\hat\ybm^\star, \mubold) - \tilde j(\hat\ybm^\star, \mubold) + \thetabold^T\tilde\cbm(\hat\ybm^\star, \mubold)}+\\
&\quad \; \frac{\tau}{2}\abs{\paren{\norm{\hat\cbm(\hat\ybm^\star, \mubold)}-\norm{\tilde\cbm(\hat\ybm^\star, \mubold)}}\paren{\norm{\hat\cbm(\hat\ybm^\star, \mubold)}+\norm{\tilde\cbm(\hat\ybm^\star, \mubold)}}}\\
&\leq c_2\norm{\trbm(\hat\ybm^\star, \mubold)} + \abs{\hat\ell^\mathtt{L}(\hat\ybm^\star,\mubold;\thetabold)-\tilde\ell^\mathtt{L}(\hat\ybm^\star,\mubold;\thetabold)}+\frac{\tau}{2}\norm{\hat\cbm(\hat\ybm^\star, \mubold)-\tilde\cbm(\hat\ybm^\star, \mubold)}\paren{\norm{\hat\cbm(\hat\ybm^\star, \mubold)}+\norm{\tilde\cbm(\hat\ybm^\star, \mubold)}}\\
&\leq c_2\norm{\trbm(\hat\ybm^\star, \mubold)} + \abs{\hat\ell^\mathtt{L}(\hat\ybm^\star, \mubold;\thetabold)-\tilde\ell^\mathtt{L}(\hat\ybm^\star, \mubold;\thetabold)}+\tau c_3\norm{\hat\cbm(\hat\ybm^\star, \mubold)-\tilde\cbm(\hat\ybm^\star, \mubold)}, \\
\end{aligned}
\end{equation}
for some constant $c_3 > 0$, where the last inequality used boundedness of $\cbm$ and $\tilde\cbm$. Because $\rhobold$ is taken to be the solution of (\ref{eqn:linprog}) with
$\Ccal_{\Phibold,\Xibold,\deltabold}\subset\Ccal_{\Phibold,\Xibold,\delta_\mathtt{rp}}^\mathtt{rp}\cap\Ccal_{\Phibold,\Xibold,\delta_\mathtt{lq}}^\mathtt{lq}\cap\Ccal_{\Phibold,\Xibold,\delta_\mathtt{c}}^\mathtt{c}$, we have
\begin{equation}
\norm{\tilde\rbm(\hat\ybm^\star, \mubold)} \leq \delta_\mathtt{rp}, \quad
\abs{\hat\ell^\mathtt{L}(\hat\ybm^\star,\mubold;\thetabold)-\tilde\ell^\mathtt{L}(\hat\ybm^\star,\mubold;\thetabold)} \leq \delta_\mathtt{lq}, \quad
\tau\norm{\hat\cbm(\hat\ybm^\star, \mubold)-\tilde\cbm(\hat\ybm^\star, \mubold)} \leq \delta_\mathtt{c},
\end{equation}
which follows directly from (\ref{eqn:eqp:rescon}) and leads to the desired result in (\ref{eqn:eqp:qoi_errbnd2}).
\end{proof}

\begin{proof}[Proof of Corollary~\ref{cor:qoi_grad_errbnd}]
From Theorem~\ref{the:qoi_grad_errbnd} and the proposed training procedure (\ref{eqn:reduc_basis}) that guarantees $\ubm^\star, \lambdabold^\star \in \mathrm{Ran}~\Phibold$, we have
\begin{equation}
\norm{\nabla f(\mubold;\thetabold,\tau) - \nabla\tilde{f}(\mubold;\thetabold,\tau)} \leq c_3'\norm{\tilde\rbm(\hat\ybm^\star,\mubold)}+c_4'\norm{\tilde\rbm^\lambda(\hlam^\star, \hat\ybm^\star,\mubold;\thetabold,\tau)}+
 \norm{\hat\gbm^\lambda(\hlam^\star,\hat\ybm^\star,\mubold;\thetabold,\tau) - \tilde\gbm^\lambda(\hlam^\star,\hat\ybm^\star,\mubold;\thetabold,\tau)}
\end{equation}
because the HDM primal and adjoints solution can be recovered by the ROM, i.e., $\ubm^\star = \Phibold \hat\ybm^\star$ and $\lambdabold^\star = \Phibold\hat\lambdabold^\star$, which implies $\rbm(\Phibold\hat\ybm^\star,\mubold) = \zerobold$ and $\rbm^\lambda(\Phibold\hat\lambdabold^\star,\Phibold\hat\ybm^\star,\mubold) = \zerobold$. Then, from the definitions introduced in Section~\ref{sec:hyperreduction} and some basic operations, we have
\begin{equation}
\begin{aligned}
&\norm{\nabla f(\mubold;\thetabold,\tau) - \nabla\tilde{f}(\mubold;\thetabold,\tau)}\\ 
&\leq c_3'\norm{\tilde\rbm(\hat\ybm^\star,\mubold)}+c_4'\norm{\tilde\rbm^\lambda(\hlam^\star, \hat\ybm^\star,\mubold;\thetabold,\tau)}+
 \norm{\hat\gbm^\lambda(\hlam^\star,\hat\ybm^\star,\mubold;\thetabold,\tau) - \tilde\gbm^\lambda(\hlam^\star,\hat\ybm^\star,\mubold;\thetabold,\tau)}\\
&= c_3'\norm{\tilde\rbm(\hat\ybm^\star,\mubold)}
+c_4'\norm{\pder{\trbm}{\ybm}(\hat\ybm^\star,\mubold)^T\hlam^\star-\pder{\tilde{\ell}}{\ybm}(\hat\ybm^\star,\mubold;\thetabold,\tau)^T-\pder{\hrbm}{\ybm}(\hat\ybm^\star,\mubold)^T\hlam^\star+\pder{\hat{\ell}}{\ybm}(\hat\ybm^\star,\mubold;\thetabold,\tau)^T}+\\
&\quad \; \norm{\pder{\tilde{\ell}}{\mubold}(\hat\ybm^\star,\mubold;\thetabold,\tau)-{\hlam^\star}^T\pder{\trbm}{\mubold}(\hat\ybm^\star,\mubold)-\pder{\hat{\ell}}{\mubold}(\hat\ybm^\star,\mubold;\thetabold,\tau)+{\hlam^\star}^T\pder{\hrbm}{\mubold}(\hat\ybm^\star,\mubold)}\\
&= c_3'\norm{\tilde\rbm(\hat\ybm^\star,\mubold)}+
c_4'\norm{\tilde\rbm^{\mathtt{L},\lambda}(\hlam^\star,\hat\ybm^\star,\mubold;\thetabold)-\hat\rbm^{\mathtt{L},\lambda}(\hlam^\star,\hat\ybm^\star,\mubold;\thetabold)+\paren{\tau\tilde\cbm(\hat\ybm^\star,\mubold)^T\pder{\tilde\cbm}{\ybm}(\hat\ybm^\star,\mubold)-\tau\hat\cbm(\hat\ybm^\star,\mubold)^T\pder{\hat\cbm}{\ybm}(\hat\ybm^\star,\mubold)}}+\\
&\quad \; \norm{\tilde\gbm^{\mathtt{L},\lambda}(\hlam^\star,\hat\ybm^\star,\mubold;\thetabold)-\hat\gbm^{\mathtt{L},\lambda}(\hlam^\star,\hat\ybm^\star,\mubold;\thetabold)+\paren{\tau\tilde\cbm(\hat\ybm^\star,\mubold)^T\pder{\tilde\cbm}{\mubold}(\hat\ybm^\star,\mubold)-\tau\hat\cbm(\hat\ybm^\star,\mubold)^T\pder{\hat\cbm}{\mubold}(\hat\ybm^\star,\mubold)}}\\
&\leq c_3'\norm{\tilde\rbm(\hat\ybm^\star,\mubold)}+c_4'\norm{\tilde\rbm^{\mathtt{L},\lambda}(\hlam^\star,\hat\ybm^\star,\mubold;\thetabold)-\hat\rbm^{\mathtt{L},\lambda}(\hlam^\star,\hat\ybm^\star,\mubold;\thetabold)}+\tau c_4'\norm{\tilde\cbm(\hat\ybm^\star,\mubold)}\norm{\pder{\tilde\cbm}{\ybm}(\hat\ybm^\star,\mubold)-\pder{\hat\cbm}{\ybm}(\hat\ybm^\star,\mubold)}+\\
&\quad \; \tau c_4'\norm{\tilde\cbm(\hat\ybm^\star,\mubold)-\hat\cbm(\hat\ybm^\star,\mubold)}\norm{\pder{\hat\cbm}{\ybm}(\hat\ybm^\star,\mubold)}+\norm{\tilde\gbm^{\mathtt{L},\lambda}(\hlam^\star,\hat\ybm^\star,\mubold;\thetabold)-\hat\gbm^{\mathtt{L},\lambda}(\hlam^\star,\hat\ybm^\star,\mubold;\thetabold)}+\\
&\quad \; \tau\norm{\tilde\cbm(\hat\ybm^\star,\mubold)}\norm{\pder{\tilde\cbm}{\mubold}(\hat\ybm^\star,\mubold)-\pder{\hat\cbm}{\mubold}(\hat\ybm^\star,\mubold)}+\tau\norm{\tilde\cbm(\hat\ybm^\star,\mubold)-\hat\cbm(\hat\ybm^\star,\mubold)}\norm{\pder{\hat\cbm}{\mubold}(\hat\ybm^\star,\mubold)}\\
&\leq c_3'\norm{\tilde\rbm(\hat\ybm^\star,\mubold)}+c_4'\norm{\hat\rbm^{\mathtt{L},\lambda}(\hlam^\star,\hat\ybm^\star,\mubold;\thetabold)-\tilde\rbm^{\mathtt{L},\lambda}(\hlam^\star,\hat\ybm^\star,\mubold;\thetabold)}+\norm{\hat\gbm^{\mathtt{L},\lambda}(\hlam^\star,\hat\ybm^\star,\mubold;\thetabold)-\tilde\gbm^{\mathtt{L},\lambda}(\hlam^\star,\hat\ybm^\star,\mubold;\thetabold)} + \\
&\quad \; +\tau c_5'\norm{\hat\cbm(\hat\ybm^\star,\mubold)-\tilde\cbm(\hat\ybm^\star,\mubold)}+\tau c_6'\norm{\partial_\ybm\hat\cbm(\hat\ybm^\star,\mubold)-\partial_\ybm\tilde\cbm(\hat\ybm^\star,\mubold)} 
+ \tau c_7'\norm{\partial_\mubold\hat\cbm(\hat\ybm^\star,\mubold)-\partial_\mubold\tilde\cbm(\hat\ybm^\star,\mubold)}
\end{aligned}
\end{equation}
for some constants $c_5',c_6',c_7' > 0$, where the last inequality used boundedness of $\cbm$, $\tilde\cbm$, and their partial derivatives. Because $\rhobold$ is taken to be the solution of (\ref{eqn:linprog}) with
$\Ccal_{\Phibold,\Xibold,\deltabold}\subset
\Ccal_{\Phibold,\Xibold,\delta_\mathtt{rp}}^\mathtt{rp} \cap \Ccal_{\Phibold,\Xibold,\delta_\mathtt{lra}}^\mathtt{lra} \cap
\Ccal_{\Phibold,\Xibold,\delta_\mathtt{lga}}^\mathtt{lga} \cap
\Ccal_{\Phibold,\Xibold,\delta_\mathtt{c}}^\mathtt{c} \cap
\Ccal_{\Phibold,\Xibold,\delta_\mathtt{dcy}}^\mathtt{dcy} \cap
\Ccal_{\Phibold,\Xibold,\delta_\mathtt{dc\mu}}^\mathtt{dc\mu}
$, we have
\begin{equation}
\begin{aligned}
 \norm{\tilde\rbm(\hat\ybm^\star,\mubold)} &\leq \delta_\mathtt{rp} \\
 \norm{\hat\rbm^{\mathtt{L},\lambda}(\hlam^\star, \hat\ybm^\star,\mubold;\thetabold)-\tilde\rbm^{\mathtt{L},\lambda}(\hlam^\star, \hat\ybm^\star,\mubold;\thetabold)} &\leq \delta_\mathtt{lra} \\
 \norm{\hat\gbm^{\mathtt{L},\lambda}(\hlam^\star,\hat\ybm^\star,\mubold;\thetabold)-\tilde\gbm^{\mathtt{L},\lambda}(\hlam^\star,\hat\ybm^\star,\mubold;\thetabold)} &\leq \delta_\mathtt{lga} \\
\tau\norm{\hat\cbm(\hat\ybm^\star,\mubold) - \tilde\cbm(\hat\ybm^\star,\mubold)} &\leq \delta_\mathtt{c} \\
\tau\norm{\partial_\ybm\hat\cbm(\hat\ybm^\star,\mubold) - \partial_\ybm\tilde\cbm(\hat\ybm^\star,\mubold)} &\leq \delta_\mathtt{dcy} \\
\tau\norm{\partial_\mubold\hat\cbm(\hat\ybm^\star,\mubold) - \partial_\mubold\tilde\cbm(\hat\ybm^\star,\mubold)} &\leq \delta_\mathtt{dc\mu},
\end{aligned}
\end{equation}
which follows directly from (\ref{eqn:eqp:rescon}) and leads to the desired result in (\ref{eqn:eqp:qoi_grad_errbnd2}).
\end{proof}

\section*{Acknowledgments}
This material is based upon work supported by the Air Force Office of
Scientific Research (AFOSR) under award numbers FA9550-20-1-0236
and FA9550-22-1-0004, and National Science Foundation (NSF) under
award number NSF CBET-2338843. The content of this publication does not
necessarily reflect the position or policy of any of these supporters,
and no official endorsement should be inferred.

\bibliographystyle{plain}
\bibliography{biblio.bib}

\end{document}